%% file: zigzag.tex
\documentclass[letterpaper,11pt]{article}

\usepackage[english]{babel}
\usepackage[utf8]{inputenc}
\usepackage{palatino}
\usepackage{amsmath, amssymb, amsthm, amsfonts}
\usepackage{mathrsfs}
\usepackage[
colorlinks=true,
linkcolor=blue,
anchorcolor=blue,
citecolor=blue,
urlcolor=blue,
plainpages=false,
pdfpagelabels
]{hyperref}
\usepackage{graphicx}
\usepackage[colorinlistoftodos,textsize=scriptsize]{todonotes}
\usepackage{fullpage}
\usepackage[compact]{titlesec}
\usepackage[margin=0.5in]{caption}
\usepackage{enumerate}
\usepackage{url}
\usepackage{pgf,tikz}
\usepackage[all]{xy}
\usepackage{tikz-cd}
\usetikzlibrary{calc}
\usetikzlibrary{decorations.pathreplacing}
\usetikzlibrary{arrows}
\usepackage{mathrsfs}  
\usepackage{blkarray}
\usepackage{subcaption}
\usepackage{etoolbox}
\usepackage{mathtools}

\usepackage[capitalise]{cleveref}
\usepackage{xspace}
\usepackage{xstring}
\usepackage[numbers,sort,compress]{natbib}
\theoremstyle{plain}
\newtheorem{theorem}{Theorem}[section]
\newtheorem*{utheorem}{Theorem}
\newtheorem{proposition}[theorem]{Proposition}
\newtheorem{lemma}[theorem]{Lemma} 
\newtheorem{corollary}[theorem]{Corollary}
\newtheorem{conjecture}[theorem]{Conjecture}

\theoremstyle{definition}
\newtheorem{example}[theorem]{Example}
\newtheorem{definition}[theorem]{Definition}

\theoremstyle{definition}
\newtheorem{remark}[theorem]{Remark}
\newtheorem{remarks}[theorem]{Remarks}

\let\originalparagraph\paragraph
\renewcommand{\paragraph}[2][.]{\originalparagraph{#2#1}}

\DeclareMathOperator{\im}{\mathrm im}

\DeclareMathOperator{\coim}{\mathrm coim}

\DeclareMathOperator{\End}{\mathrm End}
\DeclareMathOperator{\ob}{\mathrm Ob}
\DeclareMathOperator{\coker}{\mathrm coker}

\DeclareMathOperator{\diag}{diag}

\DeclareMathOperator{\comp}{\circ}
\DeclareMathOperator{\Reeb}{Reeb}
\DeclareMathOperator{\osum}{\bigoplus}

\DeclareMathOperator{\Lan}{Lan}
\DeclareMathOperator{\Ran}{Ran}

\newcommand{\DL}{\mathchoice
{\begin{tikzpicture}
\draw (1.0ex,0ex)--(0ex,0ex) -- (0ex,1.0ex);
\end{tikzpicture}}
{\begin{tikzpicture}
\draw (1.0ex,0ex)--(0ex,0ex) -- (0ex,1.0ex);
\end{tikzpicture}}
{\begin{tikzpicture}
\draw[thick] (1.0ex,0ex)--(0ex,0ex) -- (0ex,1.0ex);
\end{tikzpicture}}
{\begin{tikzpicture}
\draw[line width=.6pt](.7ex,0ex)--(0ex,0ex) -- (0ex,.7ex);
\end{tikzpicture}}
}

\newcommand{\TL}{\mathchoice
{\begin{tikzpicture}%
\draw[thick] (0ex,0ex)--(0ex,1.0ex) -- (1.0ex,1.0ex);%
\end{tikzpicture}}
{\begin{tikzpicture}%
\draw[thick] (0ex,0ex)--(0ex,1.0ex) -- (1.0ex,1.0ex);%
\end{tikzpicture}}
{\begin{tikzpicture}%
\draw[thick] (0ex,0ex)--(0ex,1.0ex) -- (1.0ex,1.0ex);%
\end{tikzpicture}}
{\begin{tikzpicture}%
\draw[line width=.6pt] (0ex,0ex)--(0ex,0.7ex) -- (0.7ex,0.7ex);%
\end{tikzpicture}}
}

\newcommand{\DR}{%
\mathchoice
{\begin{tikzpicture}%
\draw[thick] (1.0ex,1.0ex)--(1.0ex,0ex) -- (0ex,0ex);%
\end{tikzpicture}}
{\begin{tikzpicture}%
\draw[thick] (1.0ex,1.0ex)--(1.0ex,0ex) -- (0ex,0ex);%
\end{tikzpicture}}
{\begin{tikzpicture}%
\draw[thick] (1.0ex,1.0ex)--(1.0ex,0ex) -- (0ex,0ex);%
\end{tikzpicture}}
{\begin{tikzpicture}%
\draw[line width=.6pt] (.7ex,.7ex)--(.7ex,0ex) -- (0ex,0ex);%
\end{tikzpicture}}
}

\newcommand{\TLE}{%
\mathchoice
{\begin{tikzpicture}%
\draw[thick] (0ex,0ex)--(0ex,1.0ex) -- (1.0ex,1.0ex)--cycle;%
\end{tikzpicture}}
{\begin{tikzpicture}%
\draw[thick] (0ex,0ex)--(0ex,1.0ex) -- (1.0ex,1.0ex)--cycle;%
\end{tikzpicture}}
{\begin{tikzpicture}%
\draw[thick] (0ex,0ex)--(0ex,1.0ex) -- (1.0ex,1.0ex)--cycle;%
\end{tikzpicture}}
{\begin{tikzpicture}%
\draw[line width=.6pt] (0ex,0ex)--(0ex,.7ex) -- (.7ex,.7ex)--cycle;%
\end{tikzpicture}}
}

\renewcommand{\phi}{\varphi}
\newcommand{\ACat}{\mathbb{A}}
\newcommand{\BCat}{\mathbb{B}}
\newcommand{\CCat}{\mathbb{D}}

\newcommand{\B}{\mathcal B}
\newcommand{\C}{\mathcal C}
\newcommand{\D}{\mathcal D}
\newcommand{\F}{\mathcal{F}}
\newcommand{\G}{\mathcal{G}}
\newcommand{\go}{g}
\newcommand{\gt}{h}
\newcommand{\gob}{\bar g}

\newcommand{\h}{\bar f}

\newcommand{\Hbb}{\mathbb{L}}
\newcommand{\R}{\mathbb{R}}
\newcommand{\Q}{\mathbb{T}}
\newcommand{\FS}{\mathcal S^\uparrow}
\newcommand{\W}{\mathcal{W}}

\newcommand{\Z}{\mathbb{Z}}
\newcommand{\N}{\mathbb{N}}
\newcommand{\Fc}{\mathfrak f}
\newcommand{\Lc}{\mathcal L}
\newcommand{\M}{\mathcal M}
\newcommand{\Nc}{\mathcal{N}}
\renewcommand{\P}{\mathbb{P}}

\newcommand{\RCat}{\mathbb{R}}
\newcommand{\ZCat}{\mathbb{Z}}
\newcommand{\PCat}{\mathbb{P}}
\newcommand{\ZZ}{\mathbb{ZZ}}
\newcommand{\ZZCat}{\mathbb{ZZ}}
\newcommand{\CoEx}[0]{E}

\newcommand{\eb}[0]{\mathbf{e}}
\newcommand{\E}[1]{\ifthenelse{\equal{#1}{}}{E}{E(#1)}}
\newcommand{\FI}[1]{\mathcal S(#1)}
\newcommand{\Set}{\mathbf{Set}}
\newcommand{\idf}[1]{\emph{#1}}

\newcommand{\fl}{\mathrm{fl}}
\newcommand{\fn}{g}

\newcommand{\npGen}[1]{{#1{\kern .1ex}}^{\DL}}
\newcommand{\Gen}[1]{{(#1)}^{\DL}}
\newcommand{\GenR}[1]{{(#1)^{\DR}}}
\newcommand{\GenRR}[1]{{(#1)^{\TL}}}
\newcommand{\TopRR}[2]{{(#1)_{#2}^{\TLE}}}

\newcommand{\LSB}{\mathcal L}

\newcommand{\ORD}{\mathbf{Ord}}

\newcommand{\I}{{\mathcal I}}
\newcommand{\J}{{\mathcal J}}
\newcommand{\K}{{\mathcal K}}
\newcommand{\id}{{\rm id}}
\newcommand{\Top}{\mathbf{Top}}
\newcommand{\vect}{\mathbf{vec}}
\newcommand{\Vect}{\mathbf{Vec}}
\newcommand{\Hom}{{\rm Hom}}
\newcommand{\oo}{\mathbf{o}}

\newcommand{\oc}{\mathbf{oc}}
\newcommand{\co}{\mathbf{co}}
\newcommand{\cc}{\mathbf{c}}
\newcommand{\ccc}{\mathbf{c}^{\cff}}
\newcommand{\INTR}{{\mathbb{U}}}

\newcommand{\ex}{\mathrm{thk}}
\newcommand{\bd}{\mathrm{BL}}
\newcommand{\EMB}{\mathrm{emb}}
\newcommand{\iem}{\mathrm{e}}

\newcommand{\pfd}{p.f.d.\@\xspace}
\newcommand{\op}{{\rm op}}
\newcommand{\BL}{\mathrm{BL}}
\newcommand{\blk}{\mathrm{BL}}
\newcommand{\cii}{{(\,)}}
\newcommand{\cif}{{(\,]}}
\newcommand{\cfi}{{[\,)}}
\newcommand{\cff}{{[\,]}}
\newcommand{\RCv}{\RCat^2_{\leq a}}

\newcommand{\REXD}[1]{\overrightarrow{#1}}

\newcommand{\RE}{R_\epsilon}
\newcommand{\RET}{R_{3\epsilon/2}}
\newcommand{\REComp}[2]{\ifthenelse{\equal{#1}{}}{X_{#2}}{X_{#2}(#1)}}
\newcommand{\TopSpace}{T}

\newcommand{\U}{\mathbb{U}}

\renewcommand{\subseteq}{\subset}

\date{}

\title{Algebraic Stability of Zigzag Persistence Modules}

\author{Magnus Bakke Botnan\thanks{TU M\"unchen, Munich, Germany; \texttt{botnan@ma.tum.de}} \and Michael Lesnick\thanks{Princeton University, Princeton, NJ, USA; \texttt{mlesnick@princeton.edu}}}

\begin{document}
\maketitle

\begin{abstract}
The stability theorem for persistent homology is a central result in topological data analysis.  While the original formulation of the result concerns the persistence barcodes of $\R$-valued functions, the result was later cast in a more general algebraic form, in the language of \emph{persistence modules} and \emph{interleavings}.  In this paper, we establish an analogue of this algebraic stability theorem for zigzag persistence modules.  To do so, we functorially extend each zigzag persistence module to a two-dimensional persistence module, and establish an algebraic stability theorem for these extensions.  One part of our argument yields a stability result for free two-dimensional persistence modules. 
 As an application of our main theorem, we strengthen a result of Bauer et al. on the stability of the persistent homology of Reeb graphs.  Our main result also yields an alternative proof of the stability theorem for level set persistent homology of Carlsson et al.
\end{abstract}
 
\input{ZZ_Intro}

\input{motivation}

\input{conttodisc2}
\input{freematching}
\input{summands}
\input{discussion}
\bibliography{zz_refs}

\end{document}

%% file: ZZ_Intro.tex
\section{Introduction}\label{Sec:Persistence_Modules}

\paragraph{Persistence Modules}
Let $\Vect$ denote the category of vector spaces over some fixed field $k$, and let $\vect$ denote the subcategory of finite dimensional vector spaces.  We define a \emph{persistence module} to be a functor $M:\PCat\to \Vect$, for $\PCat$ a poset.  We will often refer to such $M$ as a \emph{$\PCat$-indexed module}. If $M$ takes values in $\vect$, we say $M$ is \emph{pointwise finite dimensional (\pfd)}.  The $\PCat$-indexed persistence modules form a category $\Vect^\PCat$ whose morphisms are the natural transformations.  

Persistence modules are the basic algebraic objects of study in the theory of persistent homology.  
The theory begins with the study of \emph{1-D persistence modules}, i.e. functors $\RCat \to \Vect$ or $\ZCat \to \Vect$, where $\R$ and $\Z$ are taken to have the usual total orders.  The structure theorem for 1-D persistence modules \cite{webb1985decomposition,crawley2012decomposition} tells us that the isomorphism type of a \pfd 1-D persistence module $M$ is completely described by a collection $\B(M)$ of intervals in $\R$, called the \emph{barcode of $M$}; $\B(M)$ specifies the decomposition of $M$ into indecomposable summands.

\paragraph{Persistent Homology}
In topological data analysis, one often studies a data set by associating to the data a persistence module.  To do so, we first associate to our data a \emph{filtration}, i.e., a functor $\F:\RCat\to \Top$ such that the map $\F_a\to \F_b$ is an inclusion whenever $a\leq b$.  For example, if our data is an $\R$-valued function $\gamma:\TopSpace\to \R$, for $\TopSpace$ a topological space, we may take $\F$ to be the \emph{sublevel set filtration $\FS(\gamma)$}, defined by
\[\FS(\gamma)_a=\{y\in \TopSpace\mid \gamma(y)\leq a\},\quad a\in \R.\]
Since $\FS(\gamma)_a\subseteq \FS(\gamma)_b$ whenever $a\leq b$, this indeed gives a filtration.  If our data set is instead a point cloud, we often consider a \emph{Vietoris-Rips} or \emph{\v{C}ech} filtration; see e.g. \cite{carlsson2009topology} for details.

Letting $H_i:\Top\to \Vect$ denote the $i^{\mathrm{th}}$ singular homology functor with coefficients in $k$, we obtain a (typically \pfd) persistence module $H_i \F$ for any $i\geq 0$.  The barcodes $\B(H_i\F)$ serve as concise descriptors of the coarse-scale, global, non-linear geometric structure of the data set.  These descriptors have been applied to many problems in science and engineering, e.g., to natural scene statistics, evolutionary biology, periodicity detection in gene expression data, sensor networks, and clustering \cite{carlsson2008local,chan2013topology,perea2015sw1pers,de2007coverage,chazal2013persistence}.
\paragraph{Stability}
The \emph{stability theorem} for persistent homology guarantees that in several settings, the barcode descriptors of data are stable with respect to perturbations of the data.  The original formulation of the stability theorem \cite{cohen2007stability} concerns the persistent homology of $\R$-valued functions, and is formulated with respect to a standard metric $d_b$ on barcodes called the \emph{bottleneck distance}, which we define in \cref{Sec:Isometry}.  In the generality provided by \cite{chazal2009proximity}, the result is as follows:

\begin{theorem}[Stability of Persistent Homology for Functions \cite{cohen2007stability,chazal2009proximity}]\label{Thm:Stability_For_Functions}
For $\TopSpace$ a topological space, $i\geq 0$, and functions $\gamma,\kappa:\TopSpace\to \R$ such that $H_i\FS(\gamma)$ and $H_i\FS(\kappa)$ are \pfd, we have
\[d_b(\B(H_i\FS(\gamma)),\B(H_i\FS(\kappa)))\leq d_\infty(\gamma,\kappa),\]
where $d_\infty(\gamma,\kappa)=\sup_{x\in T}|\gamma(x)-\kappa(x)|.$
\end{theorem}

As a corollary of \cref{Thm:Stability_For_Functions}, one obtains a stability theorem for persistent homology of Rips and \v{C}ech filtrations on finite metric spaces; see \cite{chazal2009gromov,chazal2013persistence}.

\paragraph{Algebraic Stability}
A purely algebraic formulation of the stability theorem was introduced in \cite{chazal2009proximity}, generalizing the stability results for $\R$-valued functions and point cloud data.  This \emph{algebraic stability theorem} asserts that an \emph{$\epsilon$-interleaving} (a pair of ``approximately inverse" morphisms) 
between \pfd 1-D  persistence modules $M$, $N$ 
induces an \emph{$\epsilon$-matching} (approximate isomorphism) between the barcodes $\B(M)$, $\B(N)$.  In fact, it was shown in \cite{lesnick2015theory} that the converse of this result also holds: Given an $\epsilon$-matching between $\B(M)$, $\B(N)$ we can easily construct an \emph{$\epsilon$-interleaving} between $M$, $N$.  The algebraic stability theorem and its converse are together known as the \emph{isometry theorem}; see \cref{Thm:Isometry_1D} for the precise statement.  

A slightly weaker formulation of the isometry theorem establishes a relationship between the \emph{interleaving distance} (a pseudometric on persistence modules) and the bottleneck distance: It says that the interleaving distance between $M$ and $N$ is equal to the bottleneck distance between $\B(M)$ and $\B(N)$.

The algebraic stability theorem is perhaps the central theorem in the theory of persistent homology; it provides the core mathematical justification for the use of persistent homology in the study of noisy data.  The theorem is used, in one form or another, in nearly all available results on the approximation, inference, and estimation of persistent homology.

\paragraph{Induced Matching Theorem}
It was shown in \cite{bauer2015induced} that the algebraic stability theorem, ostensibly a result about pairs of morphisms of persistence modules, is in fact an immediate corollary of a general result about single morphisms of persistence modules.  This result, called the \emph{induced matching theorem}, concerns a simple, explicit map $\chi$ sending each morphism $f:M\to N$ of \pfd 1-D persistence modules to a matching $\chi(f):\B(M)\nrightarrow \B(N)$.  The theorem tells us that the quality of this matching is tightly controlled by the lengths of the longest intervals in $\B(\ker f)$ and $\B(\coker f)$; see \cref{teo:IMT}.

\paragraph{Zigzag  Modules}
For posets $\ACat$ and $\BCat$, the product poset $\ACat\times \BCat$ is defined by taking $(a,b)\leq (a',b')$ if and only $a\leq a'$ and $b\leq b'$.  Let $\ACat^{\op}$ denote the opposite poset of $\ACat$.

\emph{Zigzag modules} are natural generalizations of $\ZCat$-indexed modules which have received much attention from the topological data analysis community \cite{carlsson2010zigzag,carlsson2009zigzag,bendich2013homology}.  These are functors $\ZZCat \to \Vect$, where $\ZZCat$ is the sub-poset of $\ZCat^\op\times \ZCat$ given by
\[\ZZCat:=\left\{(i,j)\mid i\in \Z,\ j\in \{i,i-1\}\right\}.\]
A structure theorem for \pfd zigzag modules \cite{botnan2015interval} gives us a definition of barcode for these modules closely analogous to the one for 1-D persistence modules.  

\paragraph{$\INTR$-Indexed Modules}
Let $\INTR$ denote the sub-poset of $\RCat^{\op}\times \RCat$ consisting of objects $(a,b)$ with $a\leq b$.  
$\INTR$-indexed modules arise naturally as refinements of the sublevel set persistent homology modules introduced above:
Given a function $\gamma:\TopSpace\to \R$ with $\TopSpace$ a topological space, we obtain a functor $\FI{\gamma}:\INTR\to \Top$, the \emph{interlevel set filtration of $\gamma$}, by taking $\FI{\gamma}_{(a,b)}=\gamma^{-1}([a,b])$, with $\FI{\gamma}_{(a,b)}\to \FI{\gamma}_{(c,d)}$ the inclusion map whenever $c\leq a\leq b\leq d$.  For $i\geq 0$, $H_i\FI{\gamma}$ is clearly a $\INTR$-indexed module.  It can be shown that if $\gamma$ is continuous or bounded below, then $H_i\FI{\gamma}$ determines $H_i\FS(\gamma)$.

We will be especially interested in the case of functions $\gamma$ of \emph{Morse type}.  These are certain generalizations of Morse functions for which each $H_i\FI{\gamma}$ is completely determined by its restriction to a discrete sub-poset of $\U$; see \cref{Sec:Levelset_Persistence} for the definition.

$\INTR$-indexed modules also arise naturally in a different (but related) way: In \cref{sec:zigzags}, we use Kan extensions to define a fully faithful functor $E: \Vect^\ZZCat\to \Vect^{\INTR}$.  This functor appears implicitly in recent work on interlevel set persistent homology \cite{carlsson2009zigzag,bendich2013homology}.

\paragraph{Block Decomposable Modules}
In general, the algebraic structure of a $\INTR$-indexed module can be very complicated.  As a result, there is no nice definition of a barcode available for such a module in general; see \cite{carlsson2009theory} and \cite[Section 1.4]{lesnick2015interactive}.  However, if $M$ is a $\INTR$-indexed module such that either 
\begin{enumerate}
\item $M\cong H_i(\FI{\gamma})$ for $\gamma:\TopSpace\to \R$ of Morse type, or
\item $M\cong E(V)$ for $V$ a \pfd zigzag module,
\end{enumerate}
then $M$ decomposes into especially simple indecomposable summands, which we call \emph{block modules}; see \cref{Sec:Block_Decomposable_Persistence_Modules,Sec:Applications}.  
We call any $\INTR$-indexed module that decomposes into block modules \emph{block decomposable}.

We may define the barcode $\B(M)$ of a block decomposable module $M$ in much the same way that we do for 1-D and zigzag modules.  
The barcode of a block decomposable module is a collection of simple convex regions in $\R^2$ called \emph{blocks}; see \cref{Sec:Blocks} for the definition and an illustration.  

\paragraph{Level Set Barcodes}
The intersection of any block with the diagonal $y=x$ is either empty or an interval. Thus, for $M$ block decomposable, intersecting each block in $\B(M)$ with the line $y=x$, and identifying this line with $\R$, we obtain a collection $\diag \B(M)$ of intervals in $\R$.  For $\gamma:T\to \R$ of Morse type, we call \[\LSB_i(\gamma):=\diag \B(H_i\FI{\gamma})\] the $i^{\mathrm{th}}$ level set barcode of $\gamma$.  Level set barcodes were introduced in \cite{carlsson2009zigzag}.  $\LSB_i(\gamma)$ tracks how homological features are born and die as one sweeps across the level sets of $\gamma$.

\begin{theorem}[Stability of Level Set Barcodes, \cite{carlsson2009zigzag}]\label{Thm:Levelset_Stabilty_Intro}
For $\TopSpace$ a topological space, $\gamma$, $\kappa:\TopSpace\to \R$ of Morse type and $i\geq 0$,
\[d_b(\LSB_i(\gamma),\LSB_i(\kappa))\leq d_\infty(\gamma,\kappa).\]
\end{theorem}

\subsection{Our Results: Algebraic Stability for Zigzag and Block Decomposable Modules}\label{Sec:Block_Decomposition_Theorem}
$\epsilon$-Interleavings and the interleaving distance $d_I$ are readily defined on $\INTR$-indexed persistence modules.  Moreover, we will see in \cref{Sec:Isometry} that we can define $\epsilon$-matchings and a bottleneck distance $d_b$ for the barcodes of block decomposable modules in much the same way we do for 1-D persistence modules.  Given this, it is natural to wonder whether an algebraic stability result holds for block decomposable modules.  Our \cref{lem:converseAST,teo:IMTinterleaving} give the following such result:
\begin{utheorem}\mbox{}
\begin{enumerate}[(i)]
\item If there exists an $\epsilon$-interleaving between \pfd block decomposable modules $M$ and $N$, then there exists a $\frac{5\epsilon}{2}$-matching between $\B(M)$ and $\B(N)$.  
\item Conversely, if there exists an $\epsilon$-matching between $\B(M)$ and $\B(N)$, then there exists an $\epsilon$-interleaving between $M$ and $N$.
\end{enumerate}
In particular, \[d_I(M,N)\leq d_b(\B(M),\B(N))\leq \frac{5}{2}d_I(M,N).\]
\end{utheorem}
The proof of (ii) is trivial.  We refer to (i) as the \emph{block stability theorem}.  The block stability theorem was conjectured (independently) by Ulrich Bauer and Dmitriy Morozov, who were motivated by an application to the stability of Reeb graphs described below.  Discussions with Bauer and Morozov inspired this work.

We show in \cref{sec:zigzags} that by way of the functor $E:\Vect^\ZZCat\to \Vect^{\INTR}$, our forward and converse algebraic stability results for block decomposable modules specialize to corresponding algebraic stability results for zigzag modules.  The problem of establishing an  algebraic stability theorem for zigzag modules is well known amongst researchers working on the theoretical foundations of topological data analysis, and has been mentioned in print in several places; see \cite{lesnick2015theory,oudot2014zigzag,oudot2015persistence}, and also the mention of the more general problem of ``hard stability" in \cite{bubenik2014metrics}.

We obtain the block stability theorem as a corollary of induced matching results for block decomposable modules analogous to those known to hold in 1-D.  As part of the proof, we establish an induced matching theorem for free $2$-D persistence modules; this yields an isometry theorem for such modules as a corollary.
 
The block stability theorem yields an alternative proof of \cref{Thm:Levelset_Stabilty_Intro}, the stability result  for level set persistent homology.  In contrast to the earlier proof, our proof does not require us to consider extended persistence or relative homology. 

\paragraph{Algebraic Stability of Constructible Sheaves over $\R$}
Interleavings and barcodes can be defined for \pfd (co)sheaves of vector spaces over $\R$ that are constructible with respects to a locally finite partition of $\R$, much as we define them for block decomposable persistence modules; see \cite{curry} and \cite{curry2016classification}.  As a corollary, the block stability theorem yields a similar algebraic stability theorem for such (co)sheaves.  However, we will not explicitly consider (co)sheaves in this paper.

%
%
%
%
%
%
%

\subsection{Stability of the Persistent Homology of Reeb Graphs}\label{sec:introreeb}
We briefly describe the application of the block stability theorem to Reeb graphs; details are given in \cref{Sec:Reeb_Main}.  

We define a Reeb graph to be a continuous function $\gamma:G\to \R$ of Morse type, where $G$ is a topological graph and the level sets of $\gamma$ are discrete.  A well known construction associates a Reeb graph, $\Reeb(\kappa)$, to $\R$-valued function $\kappa$ of Morse type.  These invariants of $\R$-valued functions are readily computed and easy to visualize.  As such, they are popular objects of study in computational geometry and topology, and have found many applications in data visualization and exploratory data analysis.  In particular, the topological data analysis tool Mapper, commercialized by Ayasdi, constructs certain discrete approximations to Reeb graphs from point cloud data \cite{singh2007topological}. 

If we want to study the stability of Reeb graphs and Mapper in the presence of noise, we need a good metric on Reeb graphs.  In the last few years, several works have introduced such metrics and have studied their stability properties \cite{bauer2014measuring, di2014edit, de2015categorified,bauer2014strong}.  In particular, \cite{de2015categorified} presents an appealing definition of the interleaving distance $d_I$ on Reeb graphs.  

The $0^{\mathrm{th}}$ level set barcode $\LSB_0(\gamma)$ of a Reeb graph $\gamma$ encodes all non-trivial persistent homology information in the Reeb graph \cite{bauer2014measuring}.  A basic question about $d_I$, then, is whether Reeb graphs which are close with respect to $d_I$ have close $0^{\mathrm{th}}$ level set barcodes. 
Building on a result of \cite{bauer2014measuring}, Bauer, Munch, and Wang recently provided an affirmative answer to this question \cite{bauer2014strong}.  A simple formulation of their result says that for Reeb graphs $\gamma$ and $\kappa$,
\[d_b(\LSB_0(\gamma),\LSB_0(\kappa))\leq 9\, d_I(\gamma,\kappa).\]
A somewhat stronger formulation of the result can be given using the language of extended persistence; see \cite{bauer2014strong}.

As an easy corollary of the block stability theorem, our \cref{Thm:Reeb_Corollary} gives an improvement of the result of \cite{bauer2014strong}:
\begin{equation}
d_b(\LSB_0(\gamma),\LSB_0(\kappa))\leq 5\, d_I(\gamma,\kappa).
\end{equation}

\subsection{Bjerkevik's Related Work} \label{sec:bjerkevik}The version of the block stability theorem we establish here is
not tight. To prove the result, we show that it suffices to establish the result for each of four
subtypes of block decomposable modules. Our
algebraic stability results for three of the four subtypes are tight, but our result for the remaining
subtype, denoted type $\oo$, turns out to be weaker than the optimal one by factor of 5/2.  

Following the release of the first version of this paper, H\aa vard Bakke Bjerkevik has obtained a tight algebraic stability result for modules of type $\oo$, via an elegant new argument  \cite{bjerkevik2016stability}.  Together with our arguments in \cref{Sec:Block_Stability_Theorem}, this gives a tight form of the block stability theorem.  As a corollary, our stability results for zigzag modules strengthen correspondingly to an isometry theorem for zigzag modules, and the constant in our stability result for the levelset persistent homology of Reeb graphs improves is strengthened from 5 to 2, which is tight.  (On the other hand, the problem of giving a tight \emph{single-morphism} algebraic stability result remains open; see \cref{Sec:Discussion}.)
  Notably, the approach of \cite{bjerkevik2016stability} also adapts readily to give algebraic stability results for some other types of modules to which our approach does not readily extend, such as for \emph{rectangle-decomposable} persistence modules; see \cref{Sec:Discussion}. 

The main advantage of the approach to block stability taken in our paper, relative to that of \cite{bjerkevik2016stability}, is that by extending the induced matching approach to 1-D algebraic stability, our approach provides explicit matchings of barcodes.  In 1-D, the induced matching approach  is very effective, and it is natural to study how the simple, explict constructions of that approach extend to block-decomposable modules; our work makes clear both what can be done in this direction and where one encounters difficulties.  
%
%
We imagine that there could be a way to strengthen our arguments 
to recover the optimal constants for the block stability theorem obtained in \cite{bjerkevik2016stability}, via explicit matchings.  However, this would require further technical advances; see the end of \cref{Sec:Discussion}.
\subsection{Outline}
\cref{Sec:Preliminaries} of this paper reviews algebraic aspects of persistent homology, introducing generalized definitions of barcodes and the bottleneck distance along the way.  In \cref{Sec:Block_Decomposable_Persistence_Modules}, we introduce block decomposable modules and their barcodes, and state the block stability theorem.  \cref{Sec:Applications} presents our applications of the block stability theorem, including our treatment of algebraic stability for zigzag modules.  

\cref{Sec:Decomposition_Top,Sec:Free,Sec:Block_Stability_Theorem} are devoted to the proof of the block stability theorem.  \cref{Sec:Decomposition_Top} introduces a way of decomposing a monomorphism of 2-D persistence modules.  Using this decomposition, \cref{Sec:Free} proves the induced matching theorem for free 2-D persistence modules, as well as  a similar induced matching result of a more technical nature for a class of 2-D persistence modules we call $R_\epsilon$-free.  \cref{Sec:Block_Stability_Theorem} applies the results of \cref{Sec:Free} to prove the block stability theorem.  

\cref{Sec:Almost_Block_Stability} gives an easy extension of the block stability theorem to a slightly more general class of modules, and speculates on an application of this to the stability of level set persistence 
for non-Morse type functions.  We conclude in \cref{Sec:Discussion} with a brief exploration of the problem of further generalizing the results of this paper.

\paragraph{Acknowledgements}
\normalsize{This work would not have been possible if it were not for conversations with Ulrich Bauer, Justin Curry, Vin de Silva, Dmitriy Morozov, Sara Kali\v snik, Amit Patel, and Bob MacPherson that shaped our understanding of zigzag persistence.  We especially thank Ulrich Bauer and Dmitriy Morozov for (independently) introducing us to the main conjecture which underlies this work and explaining the application to Reeb graphs, and Justin Curry for many enlightening discussions in the early stages of this project.  We also thank H\aa vard Bakke Bjerkevik for valuable discussions about generalized algebraic stability, and Peter Landweber for suggesting several corrections to the paper. 
MBB wishes to thank Johan Steen for invaluable help with category theory. MBB has been partially supported by the DFG Collaborative Research Center SFB/TR 109 “Discretization in Geometry and Dynamics”. The authors began collaboration on this project at the Institute for Mathematics and its Applications, and continued the work while ML was a member of Raul Rabadan's lab at Columbia University.  We thank everyone at the IMA and Columbia for their support and hospitality. }

\section{Preliminaries}\label{Sec:Preliminaries}
For $\P$ a poset and $\C$ an arbitrary category, $M:\P\to \C$ a functor, and $a,b\in \P$, let $M_a= M(a)$, and let $\phi_M(a,b) : M_a \to M_b$ denote the morphism $M(a\leq b)$. 

\subsection{Barcodes of Interval Decomposable Persistence Modules}
\label{Sec:Barcodes}
An \emph{interval} of $\P$ is a subset $\J\subseteq \P$ such that 
\begin{enumerate}
\item $\J$ is non-empty.
\item If $a,c\in \J$ and $a\leq b\leq c$, then $b\in \J$.
\item {}[connectivity] For any $a,c\in \J$, there is a sequence $a=b_0,b_1,\ldots, b_l=c$ of elements of $\J$ with $b_i$ and $b_{i+1}$ comparable for $0\leq i\leq l-1$.   
\end{enumerate}
We refer to a multiset of intervals in $\P$ as a \emph{barcode (over $\P$)}.  

\begin{definition}
For $\J$ an interval in $\P$, the interval module $I^\J$ is the $\PCat$-indexed module such that
\begin{align*}
I^\J_a&=
\begin{cases}
k &{\textup{if }} a\in \J, \\
0 &{\textup{ otherwise}.}
\end{cases}
& \varphi_{I^\J}(a,b)=
\begin{cases}
\id_k &{\textup{if }} a\leq b\in I,\\
0 &{\textup{ otherwise}.}
\end{cases}
\end{align*}
\end{definition}
We say a persistence module $M$ is \emph{decomposable} if it can be written as $M\cong V\oplus W$ for non-trivial persistence modules $V$ and $W$; otherwise, we say that $M$ is \emph{indecomposable}. 

\begin{proposition}\label{prop:intervals_are_indec}
$I^\J$ is indecomposable. 
\end{proposition}
\begin{proof}
For $M$ a persistence module, let $\End(M)$ denote the $k$-vector space of endomorphisms of $M$.  An endomorphism of $I^\J$ acts locally by multiplication, so it follows by commutativity and connectivity that  $\End(I^\J)\cong k$. Assume that $I^\J \cong M\oplus N$ for persistence modules $M$ and $N$.  Then $\End(M)\oplus\End(N)$ is a subspace of $\End(M\oplus N)\cong \End(I^\J)\cong k$. The only subspaces of $k$ are 0 and $k$, so either $\End(M) = 0$ or $\End(N) = 0$, implying that either $M$ or $N$ is trivial. 
\end{proof}

A $\PCat$-indexed module $M$ is \emph{interval decomposable} if there exists a (possibly infinite) multiset $\B(M)$ of intervals in $\P$ such that 
\[ M\cong \bigoplus_{\J\in \B(M)} I^{\J}.\]
Since the endomorphism rings of interval persistence modules are local (in fact, isomorphic to $k$), it follows from the Azumaya--Krull--Remak--Schmidt theorem \cite{azumaya1950corrections} that the multiset $\B(M)$ is uniquely defined.  We call $\B(M)$ the \emph{barcode} of $M$.

\begin{theorem}[Structure of 1-D and Zigzag Persistence Modules \cite{botnan2015interval,crawley2012decomposition}]\label{Structure_Theorem}
Suppose $M$ is a \pfd $\PCat$-indexed module, for $\PCat\in \{\RCat,\ZCat,\ZZCat\}$. Then $M$ is interval decomposable.
\end{theorem}

\begin{remark}
For $\ZZCat$-indexed modules, this structure theorem has typically appeared in the TDA literature under an additional finiteness assumption---see \cite{carlsson2010zigzag}, for example.  However, a proof of the general result as stated above can be found in \cite{botnan2015interval}.  
\end{remark}

\subsection{Multidimensional Persistence Modules and Interleavings}\label{Sec:Multi_D_And_Interleavings_Intro}
\paragraph{Multidimensional Persistence Modules}
For $n\geq 1$, let $\RCat^n$ denote the poset obtained by taking the product of $\RCat$ with itself $n$ times.  $\RCat^n$-indexed modules are known in the TDA literature as \emph{$n$-dimensional persistence modules}.  They arise naturally in the study of data with noise or non-uniformities in density; see e.g. \cite{carlsson2009theory,chazal2011geometric,lesnick2015interactive}.

\begin{remark}
The analogue of \cref{Structure_Theorem} does not hold for $\PCat=\RCat^n$ when $n\geq 2$.  Indeed, it is a basic lesson from the representation theory of quivers that an arbitrary $\PCat$-indexed module $M$ is interval decomposable only for very special choices of $\PCat$.
\end{remark}

\paragraph{Interleavings of $\RCat^n$-indexed Functors}
For $\C$ an arbitrary category and $u\in \R^n$, define the \emph{$u$-shift functor} $(-)(u): \C^{\RCat^n} \to \C^{\RCat^n}$ on objects by $M(u)_a = M_{u+a}$, together with the obvious internal morphisms, and on morphisms $f: M\to N$ by $f(u)_a = f(u+a): M(u)_a \to N(u)_a$. For $u\in [0,\infty)^n$, let $\phi_M^u: M \to M(u)$ be the morphism whose restriction to each $M_a$ is the linear map $\phi_M(a, a+u)$. For $\epsilon\in [0,\infty)$ we will abuse notation slightly by letting $(-)(\epsilon)$ denote the $\epsilon(1, \ldots, 1)$-shift functor, and letting $\phi_M^\epsilon$ denote $\phi^{\epsilon(1, \ldots, 1)}_M$.

\begin{definition}
Given $\epsilon\in [0,\infty)$, we say functors $M,N:\RCat^n\to \C$ are $\epsilon$-interleaved if there exist morphisms $f: M\to N(\epsilon)$ and $g: N\to M(\epsilon)$ such that 
\[g(\epsilon)\circ f = \phi_M^{2\epsilon},\qquad f(\epsilon)\circ g = \phi_N^{2\epsilon}.\]
\end{definition}
We call $f$ and $g$ \emph{$\epsilon$-interleaving morphisms}. 
The \emph{interleaving distance} \[d_I: \ob(\C^{\RCat^n})\times\ob(\C^{\RCat^n})\to [0,\infty]\] is given by 
\[ d_I(M,N) = \inf \{\epsilon \geq 0 \mid \text{$M$ and $N$ are $\epsilon$-interleaved}\}.\]
$d_I$ is an \emph{extended pseudometric}; that is, $d_I$ is symmetric, $d_I$ satisfies the triangle inequality, and $d_I(M,M)=0$ for all $\RCat^n$-indexed modules $M$.  

\paragraph{Interleavings and $\epsilon$-trivial (co-)kernels}
For $u\in [0,\infty)^n$, we say an $n$-D persistence module $M$ is \emph{$u$-trivial} if $\phi_M^u=0$.  For $\epsilon\in [0,\infty)$, we say $M$ is $\epsilon$-trivial if $M$ is $(\epsilon,\epsilon,\ldots,\epsilon)$-trivial.  Note that $M$ is $2\epsilon$-trivial if and only if $M$ is $\epsilon$-interleaved with $0$.

\begin{remark}\label{Rem:Interleaving_And_Small_(Co)Kernel}
It is an easy exercise to show that if $f:M\to N(\epsilon)$ is an $\epsilon$-interleaving morphism, then $\ker f$ and $\coker f$ are each $2\epsilon$-trivial.  For $n=1$, the converse is also true; for $n>1$, only a weaker converse holds: if $f:M\to N(\epsilon)$ has $2\epsilon$-trivial (co)kernel, then $f$ is a $2\epsilon$-interleaving morphism, but it may not be the case that $M$ and $N$ are $\epsilon'$-interleaved for any $\epsilon'<2\epsilon$; see \cite{bauer2015induced} for details.  
\end{remark}

\paragraph{Duals of Persistence Modules}
Dualizing each vector space and each linear map in a $\PCat$-indexed module $M$ yields an $\PCat^{\op}$-indexed module $M^*$.  
As in the case of finite dimensional vector spaces, when $M$ is \pfd, $M^{**}$ is canonically isomorphic to $M$.  Moreover, given a map $f:M\to N$ of $\PCat$-indexed modules, we have a dual map $f^*:N^*\to M^*$.  This gives a functor \[(-)^*:\Vect^{\PCat}\to \Vect^{\PCat^{\op}}.\]
We omit the proof of the following:
\begin{proposition}\mbox{}\label{Prop:Kernel_Dualization}
\begin{enumerate}[(i)]
\item If $f:M\to N$ is a morphism of $\RCat^{n}$-indexed modules with $\epsilon$-trivial kernel, then $f^{*}$ has $\epsilon$-trivial cokernel.
\item Dually, if $f$ has $\epsilon$-trivial cokernel, then $f^{*}$ has $\epsilon$-trivial kernel.
\end{enumerate}
\end{proposition}

\subsection{The Isometry Theorem}
\label{Sec:Isometry}

\paragraph{Matchings}
A \emph{matching} $\sigma$ between multisets $S$ and $T$ (written as $\sigma: S \nrightarrow T$) is a bijection $\sigma: S\supseteq S^\prime \to T^\prime\subseteq T$. Formally, we regard $\sigma$ as a relation $\sigma\subseteq S\times T$ where $(s,t)\in \sigma$ if and only if $s\in S^\prime$ and $\sigma(s) = t$. We call $S^\prime$ and $T^\prime$ the \emph{coimage} and \emph{image} of $\sigma$, respectively, and denote them by $\coim\sigma$ and $\im\sigma$.  If $w\in \coim \sigma\cup \im \sigma$, we say that \emph{$\sigma$ matches $w$.}  We say that $\sigma$ is \emph{bijective} if $S'=S$ and $T'=T$.   

For two matchings $\sigma: S\nrightarrow R$ and $\tau: R\nrightarrow T$ we define the \emph{composite matching} $\tau\circ \sigma: S\nrightarrow T$ by taking $(s,t)\in \tau \circ\sigma$ if and only if $(r,t)\in \tau$ and $(s,r)\in \sigma$ for some $r\in R$. 

\paragraph{Generalized $\epsilon$-Matchings and Bottleneck Distance}
We now introduce a generalization of the bottleneck distance to barcodes over $\R^n$.  

We say intervals $\J,\K\subseteq \R^n$ are $\epsilon$-interleaved if $I^\J$ and $I^\K$ are $\epsilon$-interleaved.  Similarly, we say $\J$ is $\epsilon$-trivial if $I^\J$ is $\epsilon$-trivial, i.e., if for each $a\in \J$, $a+\epsilon(1, \ldots, 1)\not \in \J$.   For $\C$ a barcode over $\R^n$ and $\epsilon\geq 0$, define $\C_{\epsilon}\subseteq \C$ to be the multiset of intervals $\J$ in $\C$ such that are not $\epsilon$-trivial.  

Define an \emph{$\epsilon$-matching} between barcodes $\mathcal C$ and $\mathcal D$ to be a matching 
$\sigma:\C\nrightarrow \D$ satisfying the following properties:
\begin{enumerate}
\item $\C_{2\epsilon} \subseteq \coim \sigma$ and $\D_{2\epsilon}\subseteq \im \sigma$.
\item If $\sigma(\J)=\K$, then $\J$ and $\K$ are $\epsilon$-interleaved.
\end{enumerate}
For barcodes $\C$ and $\D$, we define the bottleneck distance $d_b$ by
\[d_b(\mathcal C,\mathcal D)=\inf\, \{\epsilon\in [0,\infty) \mid \exists\textup{ an }\epsilon\textup{-matching between }\mathcal C\textup{ and }\mathcal D\}.\]
It is not hard to check that $d_b$ is an extended pseudometric.  In particular, it satisfies the triangle inequality.

\paragraph{$\epsilon$-Matchings of Barcodes Over $\R$}
For $\J\subset \R$ an interval and $\epsilon\geq 0$, let the interval $\ex^\epsilon(\J)$ be given by \[\ex^\epsilon(\J)=\{a\in \R\mid \exists\, b\in \J\textup{ with }|a-b|\leq \epsilon\}.\]
It is easy to check that intervals $\J,\K\subseteq \R$ are $\epsilon$-interleaved if and only if either $\J\subseteq \ex^\epsilon(\K)$  and $\K\subseteq \ex^\epsilon(\J)$, or $\J$ and $\K$ are both 2$\epsilon$-trivial.  Moreover, $\J$ is $2\epsilon$-trivial if and only if for some $a\in \R$, $\J$ is strictly contained in the interval $[a,a+2\epsilon]$. This gives us a concrete description of $\epsilon$-matchings of barcodes over $\R$.

\begin{remark}
In the 1-D setting, our definition of $\epsilon$-matching is slightly different from the one given in \cite{bauer2015induced}, because it allows us to match $2\epsilon$-trivial intervals that are far away from each other.  However, this difference turns out to be of no importance; in particular, it is easy to see that the two definitions of $\epsilon$-matching yield equivalent definitions of $d_b$.  
\end{remark}

\paragraph{The Isometry Theorem}
In its strong formulation for \pfd persistence modules \cite{bauer2015induced}, the isometry theorem says the following:
\begin{theorem}[Isometry, \cite{chazal2009proximity,lesnick2015theory,chazal2012structure, bauer2015induced}]\label{Thm:Isometry_1D}
P.f.d. $\RCat$-indexed persistence modules $M$ and $N$ are $\epsilon$-interleaved if and only if there exists an \emph{$\epsilon$}-matching between $\B(M)$ and $\B(N)$.  In particular, \[d_I(M,N)=d_b(\B(M),\B(N)).\]
\end{theorem}
See also \cite{chazal2012structure} or \cite{bauer2015induced} for a version of the isometry theorem which applies to a more general class of 1-D persistence modules called q-tame.

\paragraph{The Induced Matching Theorem}
As noted in the introduction, the induced matching theorem \cite{bauer2015induced} concerns a simple map $\chi$ sending each morphism $f:M\to N$ of \pfd $\RCat$-indexed modules to a matching $\chi(f):\B(M)\nrightarrow \B(N)$.  
We will not need the full strength of the induced matching theorem, and so to minimize the amount of notation we introduce, we present a slightly weaker version of the result.  

For $a\leq b\in \R$, let $\langle a,b\rangle\subset \R$ denote an interval in $\R$ with left endpoint $a$ and right endpoint $b$.  Thus $\langle a,a\rangle=[a,a]$, and for $a<b$, $\langle a,b\rangle$ denotes one of the intervals $(a,b)$, $[a,b]$, $[a,b)$, $(a,b]$. 

\begin{theorem}[Induced Matchings]
\label{teo:IMT}
For $f:M\to N$ a morphism of \pfd $\RCat$-indexed modules with $\epsilon$-trivial kernel and $\delta$-trivial cokernel,
\begin{enumerate}[(i)]
\item $\B(M)_\epsilon \subset \coim \chi(f)$,
\item $\B(N)_\delta\subset \im \chi(f)$,
\item If $\chi(f)$ matches $\langle a,b\rangle \in \B(M)$ to $\langle a',b'\rangle \in \B(N)$ then 
\[a'\leq a\leq a'+\delta, \qquad   b-\epsilon \leq b'\leq b,\qquad a\leq b'.\]
\end{enumerate}
\end{theorem} 

\paragraph{Converse Algebraic Stability}
One direction of \cref{Thm:Isometry_1D} generalizes immediately to interval decomposable $\R^n$-indexed modules; given the way we have defined $\epsilon$-matchings, the proof is essentially trivial.
\begin{proposition}[Converse Algebraic Stability]\label{lem:converseAST}
For interval decomposable $\RCat^n$-indexed modules $M$ and $N$, if there exists an $\epsilon$-matching between $\B(M)$ and $\B(N)$, then $M$ and $N$ are $\epsilon$-interleaved.  In particular,  
\[ d_I(M,N) \leq d_b(\B(M), \B(N)).\]
\end{proposition}

\subsection{$\INTR$-Indexed Modules as 2-D Persistence Modules}\label{Sec:Cld_To_2D}
Recalling the definition of $\INTR$ from \cref{Sec:Persistence_Modules}, we define a functor $\EMB:\Vect^\INTR\to \Vect^{\RCat^{\op}\times \RCat}$, given on objects $M$ by taking $\EMB(M)$ to be trivial outside of $\U$; explicitly, we define $\EMB(M)$ by
\begin{align*}
\EMB(M)_{(a,b)}&=\begin{cases}
M_{(a,b)} &\textup{ if }a\leq b,\\
0 & \textup{otherwise, }
\end{cases}\\
\varphi_{\EMB(M)}((a,b),(c,d))&=\begin{cases}
\varphi_M((a,b),(c,d))&\textup{ if }c\leq a\leq b\leq d,\\
0 & \textup{otherwise,}\\
\end{cases}
\end{align*}
with the action of $\EMB$ on morphisms defined in the obvious way.  
Clearly, $\EMB$ is fully faithful, so by way of this functor, we may regard $\Vect^\INTR$ as full subcategory of $\Vect^{\RCat^{\op}\times \RCat}$. 

\begin{remark}[$\RCat^{\op}\times \RCat$-Indexed and $\RCat\times \RCat^{\op}$-Indexed Modules]\label{Rmk:RopVsR2}
The isomorphism $\RCat\to\RCat^{\op}$ sending each $a\in \R$ to $-a$ induces an isomorphism $\RCat^{\op}\times \RCat\to\RCat^2$.  This in turn induces an isomorphism $\Vect^{\RCat^2}\to \Vect^{\RCat^{\op}\times \RCat}.$  By way of these isomorphisms, all the definitions introduced in \cref{Sec:Isometry} in the $\RCat^n$-indexed case, e.g. of $\epsilon$-interleavings and $\epsilon$-matchings, carry over to the $\RCat^{\op}\times \RCat$-indexed setting. Similarly, they carry over to the $\RCat\times \RCat^\op$-indexed setting.  
\end{remark}

\subsection{Kan Extensions}
In several places in this paper, we introduce functors $\Vect^\ACat\to \Vect^\BCat$ for distinct posets $\ACat$ and $\BCat$, as we have in \cref{Sec:Cld_To_2D} above.  For this, it will be convenient to adopt the language of Kan extensions.  We now briefly review Kan extensions in the specific setting of interest to us, giving concrete formulae in terms of limits and colimits.  See \cite{mac1998categories} for the standard, fully general definition of Kan extensions.

Given a functor of posets $F: \ACat\to \BCat$, and $b\in \BCat$, let \[\ACat[F\leq b]:=\{a\in \ACat\mid F(a)\leq b\}.\]  Define $\ACat[F\geq b]\subseteq \ACat$ analogously.

Given a persistence module $M:\ACat\to \Vect$, one defines a persistence module $\Lan_F(M):\BCat\to \Vect$, called the \emph{left Kan extension of $M$ along $F$}, by taking 
\[\Lan_F(M)(b)=\varinjlim M|_{\ACat[F\leq b]},\] 
with the internal maps $\Lan_F(b)\to \Lan_F(b')$ given by universality of colimits for all $b\leq b'$.  Given $M,N:\ACat\to \Vect$ and a natural transformation $f:M\to N$, universality of colimits also yields an induced morphism $\Lan_F(f):\Lan_F(M)\to \Lan_F(N).$  We thus obtain a functor $\Lan_F(-):\Vect^\ACat\to \Vect^\BCat$.

\begin{example}
For $\EMB$ the functor defined in \cref{Sec:Cld_To_2D}, letting $\iem:\U\hookrightarrow \RCat^{\op}\times \RCat$ denote the inclusion, we have $\EMB=\Lan_\iem(-)$.
\end{example}

Dually, one also defines a persistence module $\Ran_F(M):\ACat\to \Vect$, the 
\emph{right Kan extension of $M$ along $F$} by taking 
\[\Ran_F(M)(b)=\varprojlim M|_{\ACat[F\geq b]},\] 
with the internal maps given by universality of limits.  As with left Kan extensions, this definition is functorial, so that we obtain a functor $\Ran_F(-):\Vect^\ACat\to \Vect^\BCat$.

\begin{proposition}\label{Prop:Preservation_of_Coproducts}
\mbox{}
\begin{enumerate}[(i)]
\item $\Lan_F(-)$ preserves direct sums, i.e., for any indexing set $\mathcal{A}$ and persistence modules $\{M_i:\ACat\to \Vect\}_{i\in \mathcal{A}}$, we have 
\[\Lan_F\left(\bigoplus_i M_i\right)\cong \bigoplus_i \Lan_F(M_i).\]
\item Dually, $\Ran_F(-)$ preserves direct products, i.e., for any persistence modules $\{M_i:\ACat\to \Vect\}_{i\in \mathcal{A}}$, we have 
\[\Ran_F\left(\prod_i M_i\right)\cong \prod_i \Ran_F(M_i).\]
\end{enumerate}
\end{proposition}

\begin{proof}
This follows directly from standard category theory results: $\Lan_F(-)$ is left adjoint to the restriction $\Vect^\BCat\to \Vect^\ACat$ along $F$; see for example \cite[(1.1)]{riehl2014categorical}.  Since $\Lan_F(-)$ is a left adjoint it    preserves coproducts \cite[Theorem V.5.1]{mac1998categories}.  This establishes (i); the dual argument establishes (ii).
\end{proof}

\begin{remark}\label{Rem:Preservation_Of_Products_pfd}
Given an indexing set $\mathcal A$ and persistence modules $\{M_i:\ACat\to \vect\}_{i\in \mathcal{A}}$, if $\oplus_i M_i$ is p.f.d., then
\[\bigoplus_i M_i=\prod_i M_i.\]
It follows that if in \cref{Prop:Preservation_of_Coproducts}\,(ii), both $\oplus_i M_i$ and $\oplus_i \Ran_F(M_i)$ are p.f.d., then 
\[\Ran_F\left(\bigoplus_i M_i\right)\cong \bigoplus_i \Ran_F(M_i).\]
\end{remark}

%% file: motivation.tex
\section{Block Decomposable Modules}\label{Sec:Block_Decomposable_Persistence_Modules}
In general, a $\INTR$-indexed module does not decompose into a direct sum of interval modules. However, as noted in the introduction, we shall restrict our attention to $\INTR$-indexed modules called \emph{block decomposables} which admit a particularly simple decomposition. 

\paragraph{Blocks}
\label{Sec:Blocks}
For any interval $\J$ in $\R$, we define an interval $\J_\blk$ in $\U$ as follows:
\begin{align*}
(a,b)_\blk &:= \{ (x,y)\in \U \mid a<x,\, y<b\}&\textup{ for }a<b\in \R\cup\{-\infty,\infty\},\\
[a,b)_\blk &:= \{ (x,y)\in \U \mid a\leq y< b\}&\textup{ for } a<b\in \R\cup \{\infty\},\\  
(a,b]_\blk &:= \{ (x,y)\in \U \mid a< x\leq b\}&\textup{ for } a<b\in \R\cup \{-\infty\},\\
[a,b]_\blk &:= \{(x,y)\in \U\mid x\leq b,\, y\geq a\}&\textup{ for } a\leq  b\in \R.
\end{align*}
In addition, for $a<b\in \R$, we define an interval
\[ [b,a]_\blk := \{(x,y)\in \U\mid x\leq a<b\leq y\}. \]
We call an interval in $\U$ having one of the five forms above a \emph{block}, and we let $\BL$ denote the set of all blocks.  
Each of the five types of blocks is depicted in \cref{fig:intervals}.

For $a,b\in \R\cup\{\pm \infty\}$, let $\langle a,b\rangle_\blk$ denote a block of the form $(a,b)_\blk$, 
$[a,b)_\blk$, $(a,b]_\blk$, or $[a,b]_\blk$.  For example, for $a\in \R$, $\langle a,\infty\rangle_\blk$ denotes a block in $\{[a,\infty)_\blk,  (a,\infty)_\blk\}$, and for $a<b \in \R$, $\langle b,a\rangle_\blk=[b,a]_\blk$.
 
\paragraph{Block Barcodes}
We call a multiset of blocks a \emph{block barcode}.  Note that in view of \cref{Rmk:RopVsR2}, $\epsilon$-matchings and the bottleneck distance $d_b$ between block barcodes are well defined.

\paragraph{Partitions of Block Barcodes}
It will be convenient to partition $\BL$ into four subsets, as follows:
\begin{align*}
\BL^{\oo}&:=\{(a,b)_\blk\mid a<b\in \R\},\\
\BL^\co&:=\{[a,b)_\blk\mid a<b\in \R\}\cup \{(-\infty,b)_\blk\mid b\in \R\},\\
\BL^\oc&:=\{(a,b]_\blk\mid a<b\in \R\}\cup \{(a,\infty)_\blk\mid a\in \R\},\\
\BL^\cc&:=\{[a,b]_\blk\mid a,b\in \R\}\cup \{[a,\infty)_\blk\mid a\in \R\}\cup\{(-\infty,b]_\blk\mid a\in \R\}\cup\{(-\infty,\infty)_\BL\}.
\end{align*}
If $\star\in \{\oo,\co,\oc,\cc\}$ and $\J\in\BL^{\star}$, we say $\J$ is \emph{is of type} $\star$.  For example, $[0,1]_\blk$ and $(-\infty,\infty)_\blk$ are both of type $\cc$.

For $\B$ a block barcode and $\star\in \{\oo,\co,\oc,\cc\}$, we let $\B^\star$ denote the multi-subset of blocks in $\B$ of type $\BL^\star$.

\paragraph{$\epsilon$-Matchings of Block Barcodes}
The following result, whose straightforward proof we omit, yields a concrete description of an $\epsilon$-matching of block barcodes:

\begin{lemma}\label{lem:Erosion_And_Blocks}\label{Prop:Block_Barcodes}
\mbox{}
\begin{enumerate}[(i)]
\item ${\langle a,b\rangle_\blk}$ is $2\epsilon$-trivial if and only if one of the following is true:
\begin{itemize}
\item $\langle a,b\rangle_\blk$ is of type $\co$ or $\oc$, and $b-a\leq 2\epsilon$,
\item $\langle a,b\rangle_\blk$ is of type $\oo$ and $b-a\leq 4\epsilon$.
\end{itemize}
\item Blocks $\langle a,b\rangle_\blk$ and $\langle a',b'\rangle_\blk$ are $\epsilon$-interleaved if and only if either
 $\langle a,b\rangle_\blk$ and $\langle a',b'\rangle_\blk$ are of the same type and
\[|a-a'|\leq \epsilon,\qquad |b-b'|\leq \epsilon,\]
or both ${\langle a,b\rangle_\blk}$ and ${\langle a',b'\rangle_\blk}$ are $2\epsilon$-trivial.  
\end{enumerate}
\end{lemma}

\paragraph{Diagonals of Block Barcodes}
Let $D:\R\to \R^2$ denote the diagonal map, i.e., $D(t)=(t,t)$, and for any block $\J\subseteq \INTR$, let $\diag \J=D(\R)\cap \J$.  Note that for any interval $\I\subset \R$, \[\diag \I_\blk=D(\I).\]  In this sense, $\I_\blk$ is labelled by its intersection with the diagonal.

For $\B$ a block barcode, we define $\diag \B$, the diagonal of $\B$, to be the barcode over $\R$ given by \[\diag \B = \{\diag \J \mid \J\in \B,\ \diag \J \ne \emptyset \}.\]

\begin{proposition}\label{prop:olddb}
For block barcodes $\B$ and $\C$, 
\begin{enumerate}[(i)]
\item An $\epsilon$-matching $\sigma:\B\nrightarrow\C$ induces a $2\epsilon$-matching $\diag \sigma:\diag \B\nrightarrow \diag \C$.  In particular,
\[d_b(\diag \B, \diag \C) \leq 2 d_b(\B, \C).\] 
\item If additionally, $\sigma$ matches each interval $(a,b)_\blk$ in $\B^{\oo}_\epsilon\cup  \C^{\oo}_\epsilon$ to an interval $(a',b')_\blk$ with \[|a-a'|\leq \epsilon\quad\textup{and}\quad|b-b'|\leq \epsilon,\] then $\diag \sigma$ is an $\epsilon$-matching.
\end{enumerate}
\end{proposition}
\begin{proof}
This is immediate from \cref{Prop:Block_Barcodes} and the definition of an $\epsilon$-matching.
\end{proof}

\definecolor{aqaqaq}{rgb}{0.6274509803921569,0.6274509803921569,0.6274509803921569}
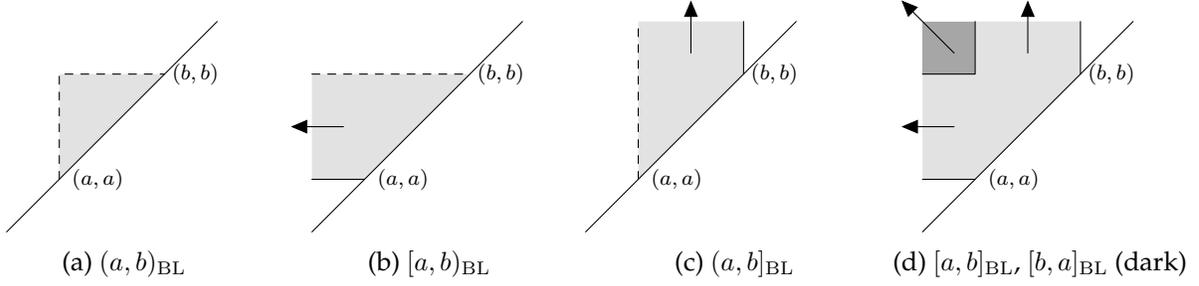
\begin{figure}
\centering
\begin{subfigure}{0.24\textwidth}

\begin{tikzpicture}[line cap=round,line join=round,>=triangle 45,x=.7cm,y=0.7cm, scale=0.2]
\clip(-13.,-10.) rectangle (13.,13.);
\fill[color=aqaqaq,fill=aqaqaq,fill opacity=0.30] (-5.,-5.) -- (-5,5) -- (5,5) -- cycle;
\draw (-10,-10) -- (10,10);
\draw[dashed] (-5.,-5.) -- (-5,5) -- (5,5);

\begin{scriptsize}
\draw[color=black] (-4.5, -5) node[right] {$(a,a)$};
\draw[color=black] (5,5) node[right] {$(b,b)$};
\end{scriptsize}
\end{tikzpicture}
\subcaption{$(a,b)_\blk$}
\end{subfigure}
\begin{subfigure}{0.24\textwidth}
\begin{tikzpicture}[line cap=round,line join=round,>=triangle 45,x=.7cm,y=0.7cm, scale=0.2]
\clip(-13.,-10.) rectangle (10.2,13.);
\fill[color=aqaqaq,fill=aqaqaq,fill opacity=0.30] (-5.,-5.) -- (-10,-5) -- (-10,5) -- (5,5) -- cycle;
\draw (-10,-10) -- (10,10);
\draw (-5,-5) -- (-10,-5);
\draw[->] (-7,0) -- (-12, 0); 
\draw[dashed] (-10, 5) -- (5,5); 
\begin{scriptsize}
\draw[color=black] (-4.5, -5) node[right] {$(a,a)$};
\draw[color=black] (5,5) node[right] {$(b,b)$};
\end{scriptsize}
\end{tikzpicture}
\subcaption{$[a,b)_\blk$}
\end{subfigure}
\begin{subfigure}{0.24\textwidth}
\begin{tikzpicture}[line cap=round,line join=round,>=triangle 45,x=.7cm,y=0.7cm, scale=0.2]
\clip(-10.,-10.) rectangle (13.,13.);
\fill[color=aqaqaq,fill=aqaqaq,fill opacity=0.30] (-5.,-5.) -- (-5,10) -- (5,10) -- (5,5) -- cycle;
\draw (-10,-10) -- (10,10);
\draw (5, 10) -- (5,5);
\draw[->] (0,7) -- (0,12);
\draw[dashed] (-5, -5) -- (-5, 10);
\begin{scriptsize}
\draw[color=black] (-4.5, -5) node[right] {$(a,a)$};
\draw[color=black] (5,5) node[right] {$(b,b)$};
\end{scriptsize}
\end{tikzpicture}
\subcaption{$(a,b]_\blk$}
\end{subfigure}
\begin{subfigure}{0.24\textwidth}
\begin{tikzpicture}[line cap=round,line join=round,>=triangle 45,x=.7cm,y=0.7cm, scale=0.2]
\clip(-13.,-10.) rectangle (13.,13.);
\fill[color=aqaqaq,fill=aqaqaq,fill opacity=0.30] (-5.,-5.) -- (-10,-5) -- (-10, 10) -- (5,10) --  (5,5) -- cycle;
\draw (-10,-10) -- (10,10);
\draw (-5, -5) -- (-10,-5);
\draw (5,10) -- (5,5);

\fill[color=aqaqaq,fill=aqaqaq,fill opacity=0.90] (-5.,5.) -- (-10,5) -- (-10, 10) -- (-5,10) -- cycle;
\draw[->] (-7,0) -- (-12, 0);
\draw[->] (0,7) -- (0,12);
\draw[->] (-7,7) -- (-12,12);
\draw (-10,5) -- (-5, 5) -- (-5,10);

\begin{scriptsize}
\draw[color=black] (-4.5, -5) node[right] {$(a,a)$};
\draw[color=black] (5,5) node[right] {$(b,b)$};
\end{scriptsize}
\end{tikzpicture}
\subcaption{$[a,b]_\blk$, $[b,a]_\blk$ (dark)}
\end{subfigure}

\caption{The five different types of blocks.}
\label{fig:intervals}
\end{figure} 

\paragraph{Block Decomposable Modules}
It follows from \cref{prop:intervals_are_indec} that for any block $\J$, the $\INTR$-indexed interval module $I^\J$ is indecomposable; we call $I^{\J}$ a \emph{block module}.  We say a $\INTR$-indexed module is \emph{block decomposable} if it decomposes into a direct sum of block modules.  We say an $\RCat^{\op}\times \RCat$-indexed module $M$ is block decomposable if $M=\EMB(N)$ for $N$ block decomposable.

With these definitions, we may work interchangeably with block decomposable modules over $\INTR$ and their embeddings under $\EMB$.  We will work primarily in the $\RCat^\op\times \RCat$-indexed setting.  

\paragraph{Block Stability}
\label{Sec:BlockStabilityTheorem}
We now state the main result of this paper, which establishes a relationship between the interleaving distance and bottleneck distance on block decomposable modules:

\begin{theorem}[Block Stability Theorem]\label{teo:IMTinterleaving}
Let $M$ and $N$ be $\epsilon$-interleaved \pfd block decomposable modules. Then there exists a matching $\chi: \B(M)\nrightarrow \B(N)$ that matches each block in
\[\B(M)^\cc  \cup \B(M)^\oo_{5\epsilon} \cup \B(M)^\co_{2\epsilon}  \cup \B(M)^\oc_{2\epsilon}\quad \text{ and}\quad \B(N)^\cc  \cup \B(N)^\oo_{5\epsilon} \cup \B(N)^\co_{2\epsilon}  \cup \B(N)^\oc_{2\epsilon},\]
such that if $\chi(\I)=\J$, then $\I$ and $\J$ are $\epsilon$-interleaved and of the same type.  
In particular, $\chi$ is a $\frac{5}{2}\epsilon$-matching. 
\end{theorem}
We give the proof of \cref{teo:IMTinterleaving} in \cref{Sec:Decomposition_Top,Sec:Free,Sec:Block_Stability_Theorem}.

\section{Applications of the Block Stability Theorem}\label{Sec:Applications}
Before turning to the proof of the block stability theorem, we consider three applications.  First, we explain how the block stability theorem induces an algebraic stability theorem for zigzag modules.  Next, we show how the stability result for level set zigzag persistence of \cite{carlsson2009zigzag} follows from the block stability theorem.  Last, we explain the application to the stability of Reeb graphs. 

\subsection{Algebraic Stability of Zigzag Persistence Modules}\label{sec:zigzags}
In this section, we define the fully faithful functor $\E{}$ sending each zigzag module to a block decomposable module, first mentioned in \cref{Sec:Persistence_Modules}.  We use $E$ to define interleaving and bottleneck distances for zigzag modules and their barcodes.  With these definitions, the block stability theorem and its converse extend trivially to zigzag modules.

Our functor $E$ is closely analogous to the functor sending a cellular cosheaf over $\R$ to a constructible cosheaf over $\R$; see \cite{curry} and the references therein.

\paragraph{Block Extensions of Zigzags}
Let $\iota:\ZZCat\hookrightarrow \RCat^{\op}\times \RCat$ denote the inclusion, and let
\[(-)|_{\INTR}:\Vect^{\RCat^{\op}\times \RCat}\to \Vect^{\INTR}\] denote the restriction.  We define the \emph{block extension functor} $\E{}:\Vect^{\ZZCat}\to \Vect^{\INTR}$  by
\[E:=(-)|_{\INTR}\comp\Lan_\iota(-).\]
\cref{fig:vint} illustrates the action of $E$ on objects.

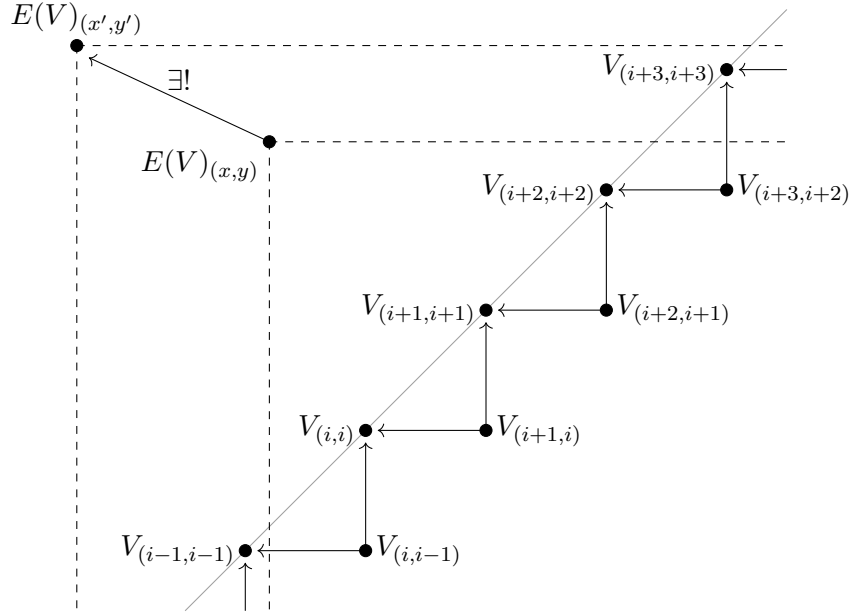
\begin{figure}
\centering
\begin{tikzpicture}[scale=1.6]
\draw[->] (1,0.5) -- (1,1-0.1);
\draw[->] (2,1) -- (2,2-0.1);
\draw[->] (3,2) -- (3,3-0.1);
\draw[->] (4,3) -- (4,4-0.1);
\draw[->] (5,4) -- (5,5-0.1);
\draw[->] (2,1) -- (1+0.1,1);
\draw[->] (3,2) -- (2+0.1,2);
\draw[->] (4,3) -- (3+0.1,3);
\draw[->] (5,4) -- (4+0.1,4);
\draw[color=aqaqaq,fill opacity=0.30] (0.5,0.5) -- (5.5,5.5);
\draw[->] (5.5, 5) -- (5+0.1,5);

\node[right] at (2,1) {$V_{({i}, {i-1})}$};
\node[right] at (3,2) {$V_{({i+1}, {i})}$};
\node[right] at (4,3) {$V_{({i+2}, {i+1})}$};
\node[right] at (5,4) {$V_{({i+3}, {i+2})}$};

\draw [fill] (2,1) circle [radius=0.05];
\draw [fill] (3,2) circle [radius=0.05];
\draw [fill] (4,3) circle [radius=0.05];
\draw [fill] (5,4) circle [radius=0.05];

\draw [fill] (5,5) circle [radius=0.05];
\draw [fill] (1,1) circle [radius=0.05];
\draw [fill] (2,2) circle [radius=0.05];
\draw [fill] (3,3) circle [radius=0.05];
\draw [fill] (4,4) circle [radius=0.05];

\node[left] at (1,1) {$V_{({i-1}, {i-1})}$};
\node[left] at (2,2) {$V_{({i}, {i})}$};
\node[left] at (3,3) {$V_{({i+1}, {i+1})}$};
\node[left] at (4,4) {$V_{({i+2}, {i+2})}$};
\node[left] at (5,5) {$V_{({i+3}, {i+3})}$};

\draw[->]  (1.2, 4.4) -- (-0.3, 5.1) node[above,midway] {$\exists !$};

\node[left] at (1.2, 4.2) {$E(V)_{(x,y)}$};
\node[above] at (-0.4, 5.2) {$E(V)_{(x',y')}$};
\draw [fill] (1.2, 4.4) circle [radius=0.05];
\draw [fill] (-0.4, 5.2) circle [radius=0.05];

\draw[dashed] (1.2, 4.4) -- (1.2, 0.5); 
\draw[dashed] (1.2, 4.4) -- (5.5, 4.4);
\draw[dashed] (-0.4, 5.2) -- (-0.4, 0.5); 
\draw[dashed] (-0.4, 5.2) -- (5.5, 5.2);
\end{tikzpicture}
\caption{The vector space $\protect\E{V}_{(x,y)}$ is the colimit of the restriction of $V$ to indices contained in the box with upper left corner $(x,y)$.} 
\label{fig:vint}
\end{figure}

\paragraph{Intervals in the Zigzag Category}
We partition the intervals of $\ZZ$ into four types; letting $<$ denote the partial order on $\Z^2$ (not on $\Z^{\op}\times \Z$), these are given as follows:
\begin{align*}
(b,d)_\ZZ &:= \{ (i,j)\in \ZZ \mid (b,b) < (i,j) < (d,d) \}&\textup{ for }b<d\in \Z\cup\{-\infty,\infty\},\\
[b,d)_\ZZ &:= \{ (i,j)\in \ZZ \mid (b,b) \leq (i,j) < (d,d) \}&\textup{ for } b<d\in \Z\cup \{\infty\},\\  
(b,d]_\ZZ &:= \{ (i,j)\in \ZZ \mid (b,b)< (i,j)\leq (d,d) \}&\textup{ for } b<d\in \Z\cup \{-\infty\},\\
[b,d]_\ZZ &:= \{ (i,j)\in \ZZ\mid (b,b) \leq (i,j) \leq (d,d) \}&\textup{ for } b\leq d\in \Z.
\end{align*}
We shall let $\langle b,d\rangle_\ZZ$ denote any of the intervals above.

\paragraph{Properties of the Block Extension Functor}
The following lemma is illustrated by \cref{fig:extending_zigzags}.  The proof is left to the reader. 
\begin{lemma}\label{lem:extendzz}
The block extension functor sends interval modules to block interval modules.  Specifically, for any zigzag interval $\langle b,d\rangle_\ZZ$, 
\[\E{I^{\langle b,d\rangle_\ZZ}} \cong I^{\langle b,d\rangle_\blk}.\]
\end{lemma}

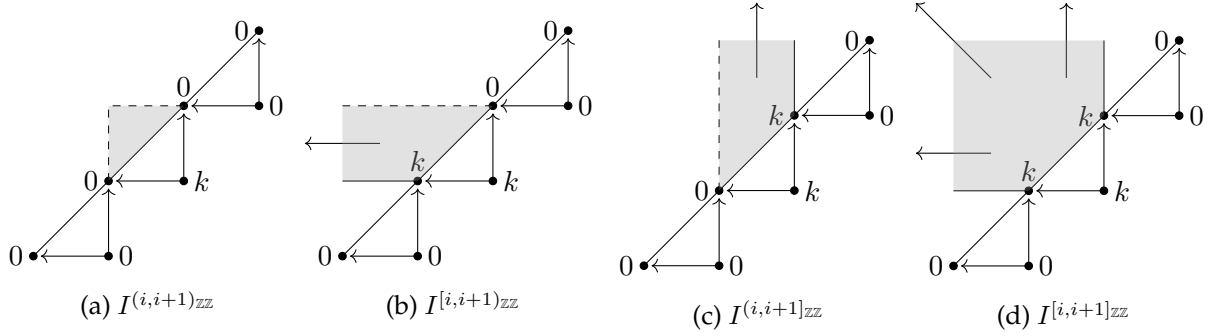
\begin{figure}
\centering
\begin{subfigure}{0.24\textwidth}
\begin{tikzpicture}
\draw[->] (2,1) -- (2,2-0.1);
\draw[->] (3,2) -- (3,3-0.1);
\draw[->] (2,1) -- (1.1,1);
\draw[->] (3,2) -- (2+0.1,2);
\draw[->] (4,3) -- (3+0.1,3);
\draw[->] (4,3) -- (4, 3.9);
\draw (1,1) -- (4,4);

\draw [fill] (2,1) circle [radius=0.05];
\draw [fill] (3,2) circle [radius=0.05];
\draw [fill] (4,3) circle [radius=0.05];
\draw [fill] (1,1) circle [radius=0.05];
\draw [fill] (2,2) circle [radius=0.05];
\draw [fill] (3,3) circle [radius=0.05];
\draw [fill] (4,4) circle [radius=0.05];

\node[right] at (2,1) {$0$};
\node[right] at (3,2) {$k$};
\node[right] at (4,3) {$0$};
\node[left] at (1,1) {$0$};
\node[left] at (2,2) {$0$};
\node[above] at (3,3) {$0$};
\node[left] at (4,4) {$0$};

\draw[dashed] (2,2) -- (2,3) -- (3,3);

\fill[color=aqaqaq,fill=aqaqaq,fill opacity=0.30] (2,2) -- (2, 3) -- (3,3) -- cycle;
\end{tikzpicture}
\caption{$I^{(i, {i+1})_\ZZ}$}
\end{subfigure}
\begin{subfigure}{0.24\textwidth}
\begin{tikzpicture}
\draw[->] (2,1) -- (2,2-0.1);
\draw[->] (3,2) -- (3,3-0.1);
\draw[->] (2,1) -- (1.1,1);
\draw[->] (3,2) -- (2+0.1,2);
\draw[->] (4,3) -- (3+0.1,3);
\draw[->] (4,3) -- (4, 3.9);
\draw (1,1) -- (4,4);

\draw [fill] (2,1) circle [radius=0.05];
\draw [fill] (3,2) circle [radius=0.05];
\draw [fill] (4,3) circle [radius=0.05];
\draw [fill] (1,1) circle [radius=0.05];
\draw [fill] (2,2) circle [radius=0.05];
\draw [fill] (3,3) circle [radius=0.05];
\draw [fill] (4,4) circle [radius=0.05];

\node[right] at (2,1) {$0$};
\node[right] at (3,2) {$k$};
\node[right] at (4,3) {$0$};
\node[left] at (1,1) {$0$};
\node[above] at (2,2) {$k$};
\node[above] at (3,3) {$0$};
\node[left] at (4,4) {$0$};

\draw[dashed] (3,3) -- (1,3);
\draw[] (2,2) -- (1,2);

\draw[->] (1.5, 2.5) -- (0.5, 2.5);

\fill[color=aqaqaq,fill=aqaqaq,fill opacity=0.30] (2,2) -- (1, 2) -- (1,3) -- (3,3) --  cycle;
\end{tikzpicture}
\caption{$I^{[i, {i+1})_\ZZ}$}
\end{subfigure}
\begin{subfigure}{0.24\textwidth}
\begin{tikzpicture}
\draw[->] (2,1) -- (2,2-0.1);
\draw[->] (3,2) -- (3,3-0.1);
\draw[->] (2,1) -- (1.1,1);
\draw[->] (3,2) -- (2+0.1,2);
\draw[->] (4,3) -- (3+0.1,3);
\draw[->] (4,3) -- (4, 3.9);
\draw (1,1) -- (4,4);

\draw [fill] (2,1) circle [radius=0.05];
\draw [fill] (3,2) circle [radius=0.05];
\draw [fill] (4,3) circle [radius=0.05];
\draw [fill] (1,1) circle [radius=0.05];
\draw [fill] (2,2) circle [radius=0.05];
\draw [fill] (3,3) circle [radius=0.05];
\draw [fill] (4,4) circle [radius=0.05];

\node[right] at (2,1) {$0$};
\node[right] at (3,2) {$k$};
\node[right] at (4,3) {$0$};
\node[left] at (1,1) {$0$};
\node[left] at (2,2) {$0$};
\node[left] at (3,3) {$k$};
\node[left] at (4,4) {$0$};

\draw[dashed] (2,2) -- (2,4);
\draw[] (3,3) -- (3,4);

\draw[->] (2.5, 3.5) -- (2.5, 4.5);

\fill[color=aqaqaq,fill=aqaqaq,fill opacity=0.30] (2,2) -- (2, 4) -- (3,4) -- (3,3) --  cycle;
\end{tikzpicture}
\caption{$I^{(i, {i+1}]_\ZZ}$}
\end{subfigure}
\begin{subfigure}{0.24\textwidth}
\begin{tikzpicture}
\draw[->] (2,1) -- (2,2-0.1);
\draw[->] (3,2) -- (3,3-0.1);
\draw[->] (2,1) -- (1.1,1);
\draw[->] (3,2) -- (2+0.1,2);
\draw[->] (4,3) -- (3+0.1,3);
\draw[->] (4,3) -- (4, 3.9);
\draw (1,1) -- (4,4);

\draw [fill] (2,1) circle [radius=0.05];
\draw [fill] (3,2) circle [radius=0.05];
\draw [fill] (4,3) circle [radius=0.05];
\draw [fill] (1,1) circle [radius=0.05];
\draw [fill] (2,2) circle [radius=0.05];
\draw [fill] (3,3) circle [radius=0.05];
\draw [fill] (4,4) circle [radius=0.05];

\node[right] at (2,1) {$0$};
\node[right] at (3,2) {$k$};
\node[right] at (4,3) {$0$};
\node[left] at (1,1) {$0$};
\node[above] at (2,2) {$k$};
\node[left] at (3,3) {$k$};
\node[left] at (4,4) {$0$};

\draw (2,2) -- (1,2);
\draw (3,3) -- (3,4);

\draw[->] (1.5, 2.5) -- (0.5, 2.5);
\draw[->] (2.5, 3.5) -- (2.5, 4.5);
\draw[->] (1.5, 3.5) -- (0.5, 4.5);

\fill[color=aqaqaq,fill=aqaqaq,fill opacity=0.30] (2,2) -- (1,2) -- (1, 4) -- (3,4) -- (3,3) --  cycle;
\end{tikzpicture}
\caption{$I^{[i, {i+1}]_\ZZ}$}
\end{subfigure}
\caption{Extension to block interval modules of the four different types of zigzag interval modules. Compare with \cref{fig:intervals}. }
\label{fig:extending_zigzags}
\end{figure}

\begin{proposition}\label{prop:zigzagext}
For any \pfd zigzag module $V$, $\E{V}$ is block decomposable, and we have a bijective matching $\B(V) \leftrightarrow \B(\E{V})$ which matches each zigzag interval $\langle b,d\rangle_\ZZ$ to the block interval $\langle b,d\rangle_\BL$.
\end{proposition}

\begin{proof}
By \cref{Prop:Preservation_of_Coproducts}, $\Lan_\iota(-)$ preserves direct sums.  Clearly $(-)|_{\INTR}$ preserves direct sums as well, so $E=(-)|_{\INTR}\comp \Lan_\iota(-)$ also preserves direct sums.  The result now follows from \cref{Structure_Theorem} and \cref{lem:extendzz}.
\end{proof}

The following result, not used elsewhere in the paper, describes an additional sense in which $E$ preserves the structure of $\Vect^{\ZZCat}$:

\begin{proposition}\mbox{}
$\E{}:\Vect^{\ZZCat}\to \Vect^{\INTR}$ is fully faithful.
\end{proposition}
\begin{proof}
$\Lan_\iota(-)$ is left adjoint to the restriction functor $(-)|_{\iota}: \Vect^{\RCat^\op\times\RCat}\to \Vect^{\ZZCat}$ \cite[(1.1)]{riehl2014categorical}.  
The reader may easily verify that $(-)|_{\iota}\comp \Lan_\iota(-) \cong \id_{\vect^{\ZZCat}}$; this also follows from \cite[Corollary X.3.3]{mac1998categories}. Hence, $\Lan_\iota(-)$ is fully faithful \cite{johnstone1989local}.
It is easy to check that this property is preserved by post-composition with $(-)|_\INTR$.
\end{proof}

\paragraph{Algebraic Stability of Zigzag Modules}
\begin{definition}
We define the interleaving and bottleneck distances on \pfd zigzag persistence modules and their barcodes by \[d_I(V,W) := d_I(E(V),E(W)),\qquad d_b^\ZZCat(\B(V), \B(W)) := d_b(\B(E(V)),\B(E(W)).\] 
\end{definition} 
Given these definitions, we get forward and converse algebraic stability results for zigzags immediately from \cref{teo:IMTinterleaving} and \cref{lem:converseAST}.

\begin{remark}
The interleaving distance on zigzag modules defined in this section is in fact an extension of the usual interleaving distance on $\ZCat$-indexed modules: We have an obvious fully faithful functor $D:\Vect^\ZCat\to \Vect^\ZZCat$ which sends a $\ZCat$-indexed module to a zigzag module by taking all leftwards arrows to be isomorphisms; that is, for $V$ a zigzag module, we take \begin{align*}
D(V)_{(i,i)}=D(V)_{(i+1,i)}&=V_i,\\
\varphi_{D(V)}((i+1,i),(i,i))&=\id_{V_i},\\
\varphi_{D(V)}((i,i-1),(i,i))&=\varphi_V(i-1,i).
\end{align*}
The ordinary interleaving distance can be defined on $\ZCat$-indexed modules just as for $\RCat$-indexed modules, and it can be checked that $D$ preserves interleaving distances.
\end{remark}

\subsection{Stability of (Inter)level Set Persistence}\label{Sec:Levelset_Persistence}
We next explain how the stability of level set and interlevel set zigzag persistence, as established in \cite{carlsson2009zigzag,bendich2013homology}, follows from the block stability theorem.  To begin, we introduce the necessary definitions, following \cite{carlsson2009zigzag}.  

\paragraph{Interlevel Set Persistent Homology}
For $\TopSpace$ a topological space, we say a continuous function $\gamma: \TopSpace\to \mathbb{R}$ is of \idf{Morse type} if
\begin{enumerate}
\item There exists a strictly increasing function $\G: \Z\to \R$ such that $\lim_{z\to \pm\infty} \G_z =\pm \infty$, and such that for each open interval $I=(\G_z,\G_{z+1})$, $\gamma^{-1}(I)$ is homeomorphic to a product $I\times Y$ with $\gamma$ the projection down on $I$. Note that $Y$ may be different for different choices of $I$.
\item
Each homeomorphism $h: I\times Y\to \gamma^{-1}(I)$ extends to a continuous function 
\[\bar{h}: \bar{I}\times Y \to \gamma^{-1}(\bar{I})=\FI{\gamma}_{\bar I},\] 
where $\bar{I}$ denotes the closure of $I$.
\item $\dim H_i(\gamma^{-1}(t))<\infty$ for all $t\in \R$ and $i\geq 0$.  
\end{enumerate}
\begin{example}\label{Ex:Immersed_Curve}
Let $\TopSpace$ be the immersed curve in $\R^2$ depicted in \cref{ex:foursummands}, and let $\gamma:\TopSpace\to \R$ denote the projection onto the $x$-axis.  Then $\gamma$ is of Morse type; we may take the function $\G:\Z\to \R$ to be the usual inclusion.
\end{example}

\begin{figure}
\centering
\begin{tikzpicture}[line cap=round,line join=round,>=triangle 45,x=1.50cm,y=1.0cm]
\draw[ultra thick,color=black] (-2.3,0.) -- (2.3,0.);
\foreach \x in {1,2,3,4,5}
\draw[dashed, opacity=0.6] (\x-3, 3) node[above] {$ $} -- (\x-3,-0.05);
\foreach \x in {-2,-1,0,1,2}
\draw[shift={(\x,0)},color=black] (0pt,-2pt) node[below] {\footnotesize $\x$};
\draw [shift={(0.,1.)}] plot[domain=-3.14159265359:0.,variable=\t]({1.*1.*cos(\t r)+0.*1.*sin(\t r)},{0.*1.*cos(\t r)+1.*1.*sin(\t r)+.1});
\draw [shift={(0.5,1.)}] plot[domain=0.:3.1415926535,variable=\t]({1.*1.5*cos(\t r)+0.*1.5*sin(\t r)},{0.*1.5*cos(\t r)+1.*1.5*sin(\t r)+.1});
\draw [shift={(-0.5,1.)}] plot[domain=0.:3.1415926535,variable=\t]({1.*1.5*cos(\t r)+0.*1.5*sin(\t r)},{0.*1.5*cos(\t r)+1.*1.5*sin(\t r)+.1});
\end{tikzpicture}
\caption{The immersed curve $T$ of \cref{Ex:Immersed_Curve,,Ex:Immersed_Curve2}.}
\label{ex:foursummands}
\end{figure}
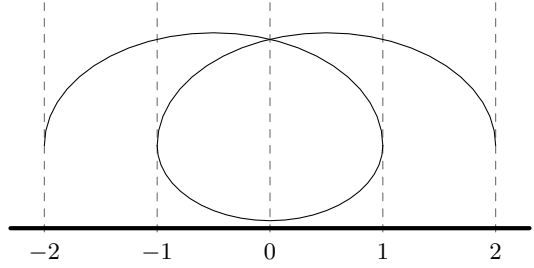

\paragraph{Structure of Interlevel Set Persistent Homology}
Recall the definition of the interlevel set filtration $\FI{\gamma}$ from \cref{Sec:Persistence_Modules}.
\begin{theorem}[\cite{carlsson2009zigzag,bendich2013homology}]\label{Thm:Morse_Block_Decomposition}
For $\gamma:\TopSpace\to \R$ of Morse type and $i\geq 0$,
\begin{enumerate}[(i)]
\item $H_i\FI{\gamma}$ is block decomposable, so that $\B_i(\gamma):=\B(H_i\FI{\gamma})$ is well defined.
\item There is a one-to-one correspondence between blocks $[b,a]_\BL\in \B_{i+1}(\gamma)$ with $a<b$ and blocks $(a,b)_\BL\in \B_i(\gamma)$.
\end{enumerate}
\end{theorem}
\cref{Thm:Morse_Block_Decomposition}\,(i) is proven by appealing to the structure theorem for zigzag persistence modules and exploiting the connection between block decomposable and zigzag persistence modules.  \cref{Thm:Morse_Block_Decomposition}\,(ii) is an application of the Mayer-Vietoris theorem.

\begin{remarks}\mbox{}
\begin{enumerate}
\item In fact, \cref{Thm:Morse_Block_Decomposition} is proven in \cite{carlsson2009zigzag,bendich2013homology} under an additional finiteness assumption.  
In view of the structure theorem for modules over infinite zigzags given in \cite{botnan2015interval}, the finiteness assumption is not necessary.
\item \cref{Thm:Morse_Block_Decomposition} admits an extension to a \emph{relative interlevel set persistence}; see \cite{carlsson2009zigzag,bendich2013homology}.  We will consider only the absolute version of the theorem here.
\item $B_i(\gamma)$ can be computed in practice by doing an extended persistence or zigzag persistent homology computation, and appealing to the formulae in \cite{carlsson2009zigzag}.  
\end{enumerate}
\end{remarks}

\paragraph{Level Set Barcodes}
Recall from \cref{Sec:Persistence_Modules} that the barcode $\LSB_i(\gamma):=\diag \B_i(\gamma)$ is called the $i^{\mathrm{th}}$ \emph{level set (zigzag) barcode of $\gamma$}.

\begin{remark}
In view of \cref{Thm:Morse_Block_Decomposition}~(ii), the block barcodes $\{\B_i(\gamma)\}_{i\geq 0}$ and the level set barcodes $\{\LSB_i(\gamma)\}_{i\geq 0}$ determine each other, so there is no loss in passing from interlevel set (block)  barcodes to level set barcodes, as long as we consider homology in all degrees.
\end{remark}

\begin{example}\label{Ex:Immersed_Curve2} 
It can be shown that for $\gamma:T\to \R$ as in \cref{Ex:Immersed_Curve}, 
\[\B_0(\gamma)=\left\{[-2,2]_\BL, (-1, 1)_\BL, [-1,0)_\BL, (0, 1]_\BL\right\}.\]

Thus the $0^{\mathrm{th}}$ level set barcode of $\gamma$ is
\[\LSB_0(\gamma)=\left\{[-2,2], (-1, 1), [-1,0), (0, 1]\right\}.\]
\end{example}

\paragraph{Stability of Level Set Persistence}
The stability theorem for level set persistence first appeared in \cite{carlsson2009zigzag}.  The original proof is an application of the stability of extended persistence \cite{cohen2009extending}, and hence can be seen as an application of algebraic stability for 1-D persistence modules.  We now give a different proof based on the block stability theorem which avoids consideration of extended persistence and relative homology. 

\begin{theorem}[Stability of (Inter)level Set Persistence]\label{Thm:LZZ_Stabilty}
Let $\gamma,\kappa: \TopSpace\to \mathbb{R}$ be of Morse type and let $\epsilon=d_\infty(\gamma,\kappa)$.  Then for all $i\geq 0$, 
\begin{align*}
d_b(\B_i(\gamma), \B_i(\kappa)) &\leq \epsilon,\\
d_b(\LSB_i(\gamma),\LSB_i(\kappa)) &\leq \epsilon.
\end{align*}
\end{theorem}
\begin{proof}
For all $x\leq y$, we have inclusions
\begin{align*}
\FI{\gamma}_{(x,y)} \subseteq \FI{\kappa}_{(x-\epsilon, y+\epsilon)} \subseteq \FI{\gamma}_{(x-2\epsilon, y+2\epsilon)},\\\FI{\kappa}_{(x,y)} \subseteq \FI{\gamma}_{(x-\epsilon, y+\epsilon)} \subseteq \FI{\kappa}_{(x-2\epsilon, y+2\epsilon)},
\end{align*}
By the functoriality of $H_i$, these induce an $\epsilon$-interleaving between $H_i\FI{\gamma}$ and $H_i\FI{\kappa}$. 
Applying \cref{teo:IMTinterleaving}, we obtain $\epsilon$-matchings between $\B_i(\gamma)^\star$ and $\B_i(\kappa)^\star$ for $\star\in\{\cc, \oc, \co\}$ and a $\frac{5}{2}\epsilon$-matching between $\B_i(\gamma)^\oo$ and $\B_i(\kappa)^\oo$.  To establish the theorem, we will in fact need an $\epsilon$-matching between $\B_i(\gamma)^\oo$ and $\B_i(\kappa)^\oo$ which matches each interval in $(a,b)_\blk\in \B_i(\gamma)^\oo_{\epsilon}\cup \B_i(\kappa)^\oo_{\epsilon}$ to an interval $(a',b')_\blk$ with \[|a-a'|\leq \epsilon\quad\textup{and}\quad|b-b'|\leq \epsilon,\] as in the statement of \cref{prop:olddb}\,(ii).  We obtain this as follows: Let $\chi: \B_{i+1}(\gamma)^\cc \to \B_{i+1}(\kappa)^\cc$ denote the $\epsilon$-matching provided by \cref{teo:IMTinterleaving}, and note that $\chi$ is  bijective.  \cref{Thm:Morse_Block_Decomposition}\,(ii) gives us injections
\[ i_1: \B_i(\gamma)^\oo\hookrightarrow \B_{i+1}(\gamma)^\cc\qquad \text{and}\qquad i_2: \B_i(\kappa)^\oo \hookrightarrow \B_{i+1}(\kappa)^\cc.\]
 
By composition, we get a matching
\[ i_2^{-1}\comp \chi \comp i_1 : \B_i(\gamma)^\oo \nrightarrow \B_i(\kappa)^\oo,\]
where $i_2^{-1}$ denotes the reverse of the matching $i_2$.  

$\chi\comp i_1$ matches each block $(a,b)_\BL\in \B_i(\gamma)^\oo_{\epsilon}$ to a block $\chi[b,a]_\BL = [b^\prime, a^\prime]_\BL\in \B_{i+1}(\kappa)^\cc$ with 
\[|a-a'|\leq \epsilon\quad\textup{and}\quad|b-b'|\leq \epsilon.\]
Since $b-a> 2\epsilon$, we have in particular that 
\[a^\prime \leq a+\epsilon < b - \epsilon \leq b^\prime.\] Thus, $[b^\prime, a^\prime]_\BL\in \im i_2$, and \[i_2^{-1}\comp \chi \comp i_1(a,b)_\BL=(a^\prime, b^\prime)_\BL\in \B_i(\kappa)^\oo.\]  This shows that \[\B_i(\gamma)^\oo_{\epsilon}\in \coim i_2^{-1}\comp \chi \comp i_1.\]  Applying the same argument in the opposite direction, we obtain that \[\B_i(\kappa)^\oo_{\epsilon}\in \im i_2^{-1}\comp \chi \comp i_1,\] and that $i_2^{-1}\comp \chi \comp i_1$ is an $\epsilon$-matching as in the statement of \cref{prop:olddb}\,(ii).  Applying \cref{prop:olddb}\,(ii), the result now follows. 
\end{proof}

In \cref{Sec:Almost_Block_Stability}, we discuss the stability problem for interlevel set and level set persistent homology in the case that our functions are not of Morse type.
\subsection{Interleaving Stability of Reeb Graphs}\label{Sec:Reeb_Main}
This section applies the Block Stability Theorem to strengthen the result of \cite{bauer2014strong} on the interleaving stability of Reeb graphs.  To begin, we review Reeb graphs and their interleavings.  Our discussion loosely follows \cite{de2015categorified}, which gives an in-depth treatment of the categorical interpretation of Reeb graphs; see that paper for more details.

\paragraph{Reeb Graphs}
Recall from \cref{sec:introreeb} that we define a Reeb graph to be a continuous function $\gamma:G\to \R$ of Morse type, where $G$ is a topological graph and the level sets of $\gamma$ are discrete.

We associate a Reeb graph, $\Reeb(\kappa)$, to any function $\kappa: \TopSpace\to \R$ of Morse type, in the following way: Define an equivalence relation on $\TopSpace$ by taking $x\sim y$ if and only if $x$ and $y$ lie in the same connected component of $\kappa^{-1}(s)$ for some $s\in \R$, and let $\TopSpace/{\sim}$ denote the resulting quotient space.  $\kappa$ descends to a continuous function \[\Reeb(\kappa) : \TopSpace/{\sim} \to \R.\]
It is easy to check that $\Reeb(\kappa)$ is indeed a Reeb graph as defined above.

\paragraph{Interleavings of Reeb Graphs}
In essentially the same way that we defined the functor $\EMB:\Vect^{\U}\to \Vect^{\RCat^{\op}\times \RCat}$ in \cref{Sec:Cld_To_2D}, we can define a functor \[\EMB:\Set^{\U}\to \Set^{\RCat^{\op}\times \RCat}.\]  Namely, for $M:\U\to \Set$, we take $\EMB(M)|_\U=M$, and we take $\EMB(M)_{(a,b)}=\emptyset$ whenever $b<a$.  We define an $\epsilon$-interleaving of Reeb graphs $\gamma$ and $\kappa$ to be an $\epsilon$-interleaving between $\EMB\circ \pi_0\circ \FI{\gamma}$ and $\EMB\circ \pi_0\circ \FI{\kappa}$, where $\pi_0:\Top\to \Set$ denotes the path components functor.

\begin{remark}
The definition of interleaving of Reeb graphs introduced in \cite{de2015categorified} is slightly different from ours, in that the definition of \cite{de2015categorified} is given in terms of the inverse images under $\gamma$ of bounded open intervals, rather than bounded closed intervals.   
It is easy to see, however, that the interleaving distances associated with the two definitions are equal.    
\end{remark}

\paragraph{Interlevel Persistence of Reeb Graphs}
As noted in \cref{sec:introreeb}, the following stability result for the persistent homology of Reeb graphs strengthens the result of Bauer, Munch, and Wang \cite{bauer2014strong}.

\begin{theorem}\label{Thm:Reeb_Corollary}
For $\epsilon$-interleaved Reeb graphs $\gamma$ and $\kappa$ of Morse type,
\[d_b(\LSB_0(\gamma),\LSB_0(\kappa))\leq 5\,d_I(\gamma,\kappa).\]
\end{theorem}

\begin{proof}
Note that we have isomorphisms
\[
H_0 \circ \EMB \circ \pi_0 \circ S(\gamma) 
\cong\EMB \circ H_0 \circ \pi_0 \circ S(\gamma) 
\cong\EMB \circ H_0 \circ S(\gamma),
\]
and similarly for $\kappa$.  Thus, by functoriality of $H_0$, an $\epsilon$-interleaving between $\gamma$ and $\kappa$ induces an $\epsilon$-interleaving between $H_0\FI{\gamma}$ and $H_0\FI{\kappa}$.  Applying \cref{teo:IMTinterleaving} and \cref{prop:olddb} to this interleaving gives the desired result.
\end{proof}

%% file: conttodisc2.tex
\section{Decomposition of Monomorphisms with Small Cokernel}\label{Sec:Decomposition_Top}\label{sec:cont2disc}
We now begin developing the technical machinery needed to prove the block stability theorem 
 and our induced matching theorem for free 2-D persistence modules.

A morphism $f:M\to N$ of persistence modules is a \emph{monomorphism} (respectively, \emph{epimorphism}) if each map of vector spaces $f_a:M_a\to N_a$ is an injection (respectively, surjection).
This section concerns the decomposition of a monomorphism of 2-D persistence modules with $\epsilon$-trivial cokernel into a pair of simpler monomorphisms whose cokernels are each short-lived in one of the two coordinate directions.

To give the reader a sense of the role that these decompositions play in our arguments, let us recall that in the induced matching approach to proving algebraic stability in the 1-D case, one associates a matching 
$\chi(f):\B(M)\to \B(N)$ to a morphism $f:M\to N$ of \pfd 1-D persistence modules.  To do so, one considers the epi-mono decomposition of $f$
\begin{equation}\label{eq:canonical_decompition}
M\twoheadrightarrow \im f\hookrightarrow N;
\end{equation}
$\chi(f)$ is defined as the composition of canonical matchings
\[\B(M)\nrightarrow \B(\im f)\nrightarrow \B(N).\] 
In the present paper, we use the decomposition introduced in this section in an analogous way, to define matchings between the barcodes of free or block decomposable modules.  
 
\subsection{Definition and First Properties of Our Decomposition}\label{Sec:Decomposition}
For $1\leq i\leq 2$, let $\eb_i$ denote the $i^{\mathrm{th}}$ standard basis vector in $\R^2$.

For $f: M\to N$ a morphism of $\RCat^2$-indexed modules, we define a factorization 
\begin{equation}\label{eq:Li}
\im f \xhookrightarrow{f_1} L^\epsilon(f) \xhookrightarrow{f_2} N
\end{equation}
of the inclusion $\im f\hookrightarrow N$ by taking 
\[L^\epsilon(f)_a:= \{n\in N_a | \phi_N(a, a+\epsilon \eb_{1})(n) \in \im f\},\]
for $a\in \R^n$, with $f_1$ and $f_2$ the respective inclusions.  We call the module $L^\epsilon(f)$ the \emph{interpolant}.  
The following lemma is immediate:
\begin{lemma}\label{prop:Lprops}
\mbox{}
\begin{enumerate}[(i)]
\item $f_1$ has $\epsilon \eb_1$-trivial cokernel.
\item If $f$ has $\epsilon$-trivial cokernel, then $f_2$ has $\epsilon\eb_2$-trivial cokernel. 
\end{enumerate}
\end{lemma}

\begin{remark}
If $f$ has $\delta$-trivial kernel, then dualizing the above construction, we obtain a factorization of the epimorphism $M\twoheadrightarrow \im f$ associated to $f$.  This factorization is of the form 
\begin{equation}\label{eq:LiDual}
M \twoheadrightarrow L\twoheadrightarrow\im f
\end{equation}
for some module $L$, where the morphisms have $\delta \eb_2$-trivial and $\delta\eb_1$-trivial kernels, respectively. 
In this paper, we will exploit duality in a way that allows us to work explicitly only with the decomposition \eqref{eq:Li} of a monomorphism, avoiding explicit consideration of the decomposition \eqref{eq:LiDual} of an epimorphism.
\end{remark}

\paragraph{Interpolants Between Free and $\RE$-Free Modules}
The remainder of this section is devoted to the proof of two results describing the structure of the interpolant $L^\epsilon(f)$ in special cases.  The first of these, \cref{prop:L_i_is_free}, tells us that when $f$ is a monomorphism of \pfd free $\RCat^2$-indexed modules, then $L^\epsilon(f)$ is also free.  This result is a main step in our proof of the induced matching theorem for free modules (\cref{teo:freeIMT}).  The second result, \cref{prop:openrestfree}, is a more technical variant of \cref{prop:L_i_is_free} concerning monomorphisms of \emph{$\RE$-free} modules.  An $\RE$-free module is one obtained from a \pfd free $\RCat\times \RCat^{\op}$-indexed module by setting to 0 all vector spaces below the diagonal line $y=x+2\epsilon$; see \cref{def:RE-Free}.   \cref{prop:openrestfree} plays a role in part of our proof of the block stability theorem analogous that of \cref{prop:L_i_is_free} in the proof of \cref{teo:freeIMT}.

Our strategy for proving Propositions \ref{prop:L_i_is_free} and \ref{prop:openrestfree} centers around the computation of \emph{(multigraded) Betti numbers}, standard invariants of $\ZCat^2$-indexed  modules in commutative algebra.  The starting point for our approach is the simple observation that the first Betti number of a finitely generated $\ZCat^2$-indexed module $M$ is 0 if and only if $M$ is free.  

Because we work with $\RCat^2$-indexed modules and do not assume our modules to be finitely generated, our arguments in this section are necessarily somewhat technical.  The reader may find it helpful to consider how these arguments simplify in the finitely generated, $\ZCat^2$-indexed setting.   

\subsection{Free 2-D Persistence Modules and Betti Numbers}\label{Sec:Free_And_Betti}
To prepare for the main results of this section, we review some standard definitions and facts about 2-D persistence modules.  Though we restrict attention to the 2-D setting, everything we say here in \cref{Sec:Free_And_Betti} extends immediately to $n$-D persistence modules.

\paragraph{Free Modules}
For $a\in \R^2$, define the interval \[\npGen{a} := \{b\in \R^2 \mid a\leq b\}.\]  We say an $\RCat^2$-indexed module $F$ is \emph{free} if there is a multiset $\xi(F)$ in $\R^2$ such that \[F\cong \bigoplus_{a\in \xi(F)} I^{\npGen{a}}.\]  Note that since the barcode $\B(F)$ is uniquely defined, the multiset $\xi(F)$ is unique.   

Free $\ZCat^2$-indexed modules are defined in the analogous way; for $F$ a free $\ZCat^2$-indexed module, the invariant $\xi(F)$ is defined as a multiset in $\Z^2$.
 
\begin{remark}\label{rem:genfree}
Later we shall consider free $\RCat^\op\times \RCat$-indexed and $\RCat\times\RCat^\op$-indexed modules. These are the interval indecomposable modules with barcodes consisting, respectively, of intervals of the form
\begin{align*}
\GenR{a_1, a_2}&:= \{ (b_1, b_2)\in \R^2 \mid  a_1 \geq b_1 \mbox{ and } a_2\leq b_2\}\textup{ and }\\
\GenRR{a_1, a_2}&:= \{ (b_1, b_2)\in \R^2 \mid a_1\leq b_1 \mbox{ and } a_2\geq b_2\}.
\end{align*}
\end{remark}
A \emph{basis} for a free $\RCat^2$-indexed module $F$ is a set $\W\subseteq \bigcup_{a\in \R^2} F_a$ such that any element $m\in F_d$ can be uniquely expressed as a finite sum
\begin{equation} m = c_1\phi_F(d_1, d)(w_1) + \ldots + c_l\phi_F(d_l, d)(w_l)\label{eq:freeSum} \end{equation}
for $w_i\in \W\cap F_{d_i}$ and scalars $c_i\in k$.  For $w\in \W\cap F_a$, we write $\deg(w) = a$.  Clearly, a basis exists for any free $\RCat^n$-indexed module.  

\paragraph{Bigraded Modules}
We define a \emph{bigraded} module to be a $k[x_1,x_2]$-module $M$ equipped with a direct sum decomposition as a $k$-vector space $M \cong \bigoplus_{a \in {\Z}^2} M_a$ such that the action of $k[x_1,x_2]$ on $M$ satisfies $x_i(M_a) \subseteq M_{a+\eb_i}$ for all $a\in \Z^2$ and $i\in \{1,2\}$.  The bigraded modules form a category, where the morphisms $f:M\to N$ are module homomorphisms such that $f(M_a)\subseteq N_a$ for all $a\in \Z^2$.  There is an obvious isomorphism between $\Vect^{\ZCat^2}$ and the category of bigraded modules.  Thus, we may regard $\ZCat^2$-indexed modules as modules, in the usual sense.   

\paragraph{Minimal Resolutions}
We next give a brief introduction to minimal free resolutions of finitely generated $\ZCat^2$-indexed persistence modules. For more details, consult \cite{eisenbud2, opac-b1101705}.

A \idf{free resolution} of a $\ZCat^2$-indexed module $M$ is an exact sequence
\[\mathbf{F}=\   \cdots \xrightarrow{d_3} F^2 \xrightarrow{d_2} F^1 \xrightarrow{d_1} F^0\]
of free $\ZCat^2$-indexed modules with $M\cong \coker d_1$.  We say $\mathbf{F}$ is \idf{minimal} if $\im d_i \subseteq IF_{i-1}$ for every $i$, where $I=\langle x_1, x_2\rangle$ is the maximal graded ideal of $k[x_1,x_2]$.  

\begin{theorem}[{\cite[Theorems 19.4 and 20.2]{eisenbud2}}]\label{Prop:FreeResolutionExists}\label{Thm:Uniqueness_Of_Min_Res}
For any finitely generated $\ZCat^2$-indexed module $M$,
\begin{enumerate}[(i)]
\item there exists a minimal free resolution $\mathbf{F}$ of $M$ with each $F^i$ finitely generated,
\item if $\mathbf{F}$ and $\mathbf{G}$ are minimal free resolutions of $M$, then there is an isomorphism $\mathbf{F}\to \mathbf{G}$ inducing the identity map on $M$.  
\end{enumerate}
\end{theorem}
For the remainder of \cref{Sec:Free_And_Betti}, let $M$ be a finitely generated $\ZCat^2$-indexed module.  

\paragraph{Betti Numbers}
For $i\geq 0$ and $a\in \Z^2$, we define a non-negative integer $\xi_i(M)_a$, the \emph{$i^{\mathrm{th}}$ Betti number of $M$ at degree $a$}, by choosing a minimal free resolution $\mathbf{F}$ for $M$ and letting $\xi_i(M)_a$ be the number of copies of $a$ in $\xi(F^i)$.  It follows from \cref{Thm:Uniqueness_Of_Min_Res}\,(ii) that this definition of $\xi_i(M)_a$ is independent of the choice of $\mathbf{F}$, and is thus well formed.

Observe that $\xi_1(M)_a = 0$ for all $a\in \Z^2$ if and only if $M$ is free. 

\paragraph{A Koszul Homology Formula}
For $z\in \Z^2$, we define the $\ZCat^2$-indexed module $M(z)$ to be the shift of $M$ by $z$, exactly as we did for $\RCat^2$-indexed modules in \cref{Sec:Multi_D_And_Interleavings_Intro}.  For any $a=(a_1,a_2)\in \N^2$, we have a short chain complex
\begin{equation}
M(-a_1\eb_1 -a_2\eb_2) \xrightarrow{\kappa^a} M(-a_1\eb_1)\oplus M(-a_2\eb_2) \xrightarrow{\gamma^a} M
\label{eq:koszul}
\end{equation}
where 
\begin{align*}
\kappa^a|_{M(-a_1\eb_1-a_2\eb_2)}(m) = (-{x_2^{a_2} m}, {x_1^{a_1} m}),\qquad \gamma^a|_{M(-a_i\eb_i)}(q) = x_i^{a_i} q.
\end{align*}
We will sometimes write $\kappa^a$ and $\gamma^a$ as $\kappa^a_M$ and $\gamma^a_M$, respectively.  In addition, we abbreviate $\kappa^{(1,1)}$ and $\gamma^{(1,1)}$ by $\kappa$ and $\gamma$.

The following commutative algebra result tells us that the first Betti number can be computed locally in terms of $\gamma$ and $\kappa$:
\begin{theorem}[{\cite[Proposition 2.7]{eisenbud}}]\label{Thm:Kozul}\label{teo:free}
For any $z\in \mathbb{Z}^2$,
\[
\xi_1(M)_z = \dim \ker \gamma_z/\im \kappa_z.
\]
\end{theorem}
\cite{eisenbud} establishes \cref{Thm:Kozul} in the slightly different setting of $\Z$-graded $k[x_1,x_2]$-modules, i.e., where $k[x_1, x_2]$ is given the standard grading \[\deg(x_1^{r_1}x_2^{r_2}) = r_1 + r_2.\] However, the proof in our case is essentially the same.
\begin{remark}
One can extend the short chain complex \eqref{eq:koszul} to a chain complex whose $i^{\mathrm{th}}$ homology gives the $i^{\mathrm{th}}$ Betti number of $M$ for all $i\geq 0$.  Namely, \[\xi_i(M)_a=\dim H_i(M\otimes K_\bullet)_a,\] where $K_\bullet$, the \idf{Koszul complex}, is a minimal free resolution of $k$ as a $k[x_1,x_2]$-module. For more on this see \cite{eisenbud}.
\end{remark}

We conclude this subsection with a technical result which will be useful to us later, leaving the easy proof to the reader:

\begin{lemma}\label{Lem:Freeness_And_Exactness}
If $M$ is free, then for any $a\in \N^2$, \[\ker \gamma^a_M=\im \kappa^a_M.\]%
\end{lemma}

\subsection{Continuous Extensions of Discrete Persistence Modules}
We wish to use \cref{Thm:Kozul} to study the first Betti number of the interpolant $L^\epsilon(f)$ in the decomposition \eqref{eq:Li}.  However, \cref{Thm:Kozul} applies to finitely generated $\ZCat^2$-indexed  modules, whereas the module $L^\epsilon(f)$ is $\RCat^2$-indexed and, in the settings of interest to us, need not be finitely generated.  To bridge this gap, we introduce formalism for extentending a $\ZCat^2$-indexed module to an $\RCat^2$-indexed one.  

\paragraph{Grid Functions}
We define a \emph{(injective) 2-D grid} to be a function $\G:\Z^2\to \R^2$ given by \[\G(z_1,z_2)=(\G_1(z_1), \G_2(z_2))\] for strictly increasing functions $\G_i:\Z\to \R$ with $\lim_{i\to -\infty}=-\infty$ and $\lim_{i\to \infty}=\infty$.  

Define $\fl_\G:\R^2\to \im(\G)$ by \[\fl_\G(t)=\max \{s\in \im(\G)\mid s\leq t\}.\]

\paragraph{Continuous Extensions} 
For $\G$ a $2$-D grid, we let $\CoEx_\G$ denote the functor \[\Lan_\G(-):\Vect^{\ZCat^2}\to \Vect^{\RCat^2};\] equivalently, but more concretely, we may specify $\CoEx_\G$ as follows:
\begin{enumerate}
\item For $M$ a $\ZCat^2$-indexed persistence module and $a,b\in \R^2$, \[\CoEx_\G(M)_a=M_y,\qquad  \varphi_{\CoEx_\G(M)}(a,b)=\varphi_M(y,z),\] where $y,z\in \Z^2$ are given by $\G(y)=\fl_\G(a)$ and $\G(z)=\fl_\G(b)$.
\item The action of $\CoEx_\G$ on morphisms is the obvious one.
\end{enumerate}

\paragraph{Interpolants of a Morphism Between Free Modules as Continuous Extensions}
\begin{lemma}\label{Lem:Grids_Isos}
If $F$ is a free $\RCat^2$-indexed module and $\G:\Z^2\to \R^2$ is a $2$-D grid such that  $d\in \im \G$ whenever $d\in \xi(F)$, then for all $a\in \R^2$, $\varphi_F(\fl_\G(a),a)$ is an isomorphism.
\end{lemma}

\begin{proof}
Let $b=\fl_\G(a)$.  $\varphi_F(b,a)$ is an injection since $F$ is free, so it suffices to show that $\phi_F(b,a)$ is a surjection.  Assume that $n\in F_a$ and $n\not\in \im \phi_F(b,a)$. Then there must exist $d  \in \xi(F)$ such that $d\leq a$ and $d_l>b_l$ for at least one $l\in \{1,2\}$. Assuming $l=1$, then the point \[b' = (d_1, b_2)\] is in $\im \G$ and $b<b'\leq a$, contradicting the maximality of $b$. Similarly if $l=2$. 
\end{proof}

Let $f: M\to N$ be a morphism of finitely generated free $\RCat^2$-indexed modules. We define finite subsets $W_1$ and $W_2$ of $\R$ by taking 
\begin{align*}
W_1&:=\{a_1\mid a\in \xi(M)\cup \xi(N)\}\cup \{ a_1- \epsilon \mid a\in \xi(M)\},\\  
W_2&:=\{a_2\mid a\in \xi(M)\cup \xi(N)\}.
\end{align*}
Let $W = W_1\times W_2$ and choose a $2$-D grid $\G:\Z^2\to \R^2$ whose image contains $W$.   Let \[(-)|_\G:\Vect^{\RCat^2}\to \Vect^{\ZCat^2}\] denote the restriction along $\G$.
\begin{proposition}\label{prop:gprime}
For $f$ and $\G$ as immediately above,
\[L^\epsilon(f)\cong\CoEx_\G(L^\epsilon(f)|_\G).\]
\end{proposition} 
\begin{proof}
It suffices to show that for all $a\in \mathbb{R}^2$,
\[\phi_{L^\epsilon(f)}(\fl_\G(a),a): L^\epsilon(f)_{\fl_\G(a)}\to L^\epsilon(f)_a\] is an isomorphism.

Let $b=\fl_\G(a)$.  By \cref{Lem:Grids_Isos}, $\varphi_N(b,a)$ is an isomorphism.  Moreover, an argument similar to the proof of \cref{Lem:Grids_Isos} shows that $\varphi_M(b+\epsilon \eb_{1},a+\epsilon \eb_{1})$ is an isomorphism: Let $m\in M_{a+\epsilon \eb_{1}}$ and assume that $m\not\in \im \phi_M(b+\epsilon \eb_{1}, a+\epsilon \eb_{1})$. Then, as above, there must exist $d \in \xi(M)$ such that $d\leq a+\epsilon \eb_{1}$ and $d_l>(b+\epsilon \eb_{1})_l$ for at least one $l\in \{1,2\}$, contradicting the maximality of $b$. 

\begin{sloppypar}
Since $L^\epsilon(f)$ is a submodule of $N$, $\varphi_{L^\epsilon(f)}(b,a)$ is injective.  Let $n\in L^\epsilon(f)_a$ with ${\phi_N(a,a+\epsilon \eb_{1})}(n) = f(m)$. Since $\phi_N(b,a)$ and $\phi_M(b+\epsilon \eb_{1}, a+\epsilon \eb_{1})$ are isomorphisms, there exist $n'\in N_b$ and $m'\in M_{b+\epsilon \eb_{1}}$ with $n = \phi_N(b,a)(n^\prime)$ and $m = \phi_M(b+\epsilon \eb_{1}, a+\epsilon \eb_{1})(m^\prime)$.  The commutativity of $f$ and injectivity of $\varphi_N(b+\epsilon \eb_{1},a+\epsilon \eb_{1})$ imply that $\varphi_N(b,b+\epsilon \eb_{1})(n^\prime) = f(m^\prime)$,  and thus $n' \in L^\epsilon(f)_b$.  This shows that $\varphi_{L^\epsilon(f)}(b,a)$ is surjective, and hence an isomorphism.\qedhere
\end{sloppypar}
\end{proof}

\subsection{Trivial First Betti Numbers and Freeness}\label{Sec:Last_Sec_Of_Sec_5}

\begin{lemma}\label{teo:localformula}
For $f:M\to N$ and $\G$ as in \cref{prop:gprime},  
\begin{enumerate}[(i)]
\item $L^\epsilon(f)|_\G$ is finitely generated,
\item $\xi_1(L^\epsilon(f)|_\G)_z=0$ whenever $\G(z)\leq a-\epsilon \eb_{1}$ for some $a\in \R^2$ with $f_a$ an injection.
\end{enumerate}
\end{lemma}

\begin{proof}
(i) holds because $L^\epsilon(f)|_\G$ is a submodule of the finitely generated persistence module $N|_\G$; the standard result that a submodule of a finitely generated module over a Noetherian ring is itself finitely generated \cite{eisenbud2} also holds in the bigraded case.  

To prove (ii), let us simplify notation by writing \[\Lc=L^\epsilon(f)|_\G,\quad \Nc=N|_\G,\quad  \M=M|_\G,\quad\Fc=f|_\G.\]  Assume without loss of generality that $z=0$.  We will prove that $\xi_1(\Lc)_{0} = 0$ by showing that the quotient $\ker \gamma_{\Lc}/\im \kappa_{\Lc}$ of Theorem \ref{teo:free} vanishes at $0$. 

For $y\in \Z^2$, let $y^+$ denote the maximum element of $\Z^2$ with $\G(y^+)\leq \G(y)+\epsilon \eb_{1}$.  Note that by \cref{Lem:Grids_Isos}, for $y\in \Z^2$ and $v\in \Nc_y$, $v\in \Lc_y$ if and only if $\varphi_\Nc(y,y^+)(v)\in \im \Fc_{y^+}$.  

Note that, in view of the way we define grid functions, the $y$-coordinates of $0^+$ and $(-\eb_1)^+$ are equal, as are the $y$-coordinates of $(-\eb_2)^+$ and $(-\eb_1-\eb_2)^+$.   Symmetrically, the $x$-coordinates of $0^+$ and $(-\eb_2)^+$ are equal, as are the $x$-coordinates of $(-\eb_1)^+$ and $(-\eb_1-\eb_2)^+$.

Let $b=(0^+_1 - (-\eb_1)^+_1, 0^+_2 - (-\eb_2)^+_2)$, and let
\begin{align*}
\gamma^+_{\Nc}&=(\gamma^b_{\Nc})_{0^+}:\bigoplus_{j\in\{1,2\}} \Nc_{(-\eb_j)^+}\rightarrow \Nc_{0^+},\\
\kappa^+_{\Nc}&=(\kappa^b_{\Nc})_{0^+}: \Nc_{(-\eb_1-\eb_2)^+}\rightarrow \bigoplus_{j\in \{1,2\}} \Nc_{(-\eb_j)^+}.
\end{align*} 
Define $\gamma^+_{\M}$ and $\kappa^+_{\M}$ analogously.

In addition, let
\begin{align*}
\varphi_{\bullet}&=\bigoplus_{j\in \{1,2\}} \varphi_\Nc(-\eb_j,(-\eb_j)^+):\bigoplus_{j\in \{1,2\}} \Nc_{-\eb_j} \rightarrow  \bigoplus_{j\in \{1,2\}} \Nc_{(-\eb_j)^+},\\
\varphi_{\bullet\bullet}&=\varphi_\Nc(-\eb_1-\eb_2,(-\eb_1-\eb_2)^+):\Nc_{-\eb_1-\eb_2} \rightarrow  \Nc_{(-\eb_1-\eb_2)^+},\\
\Fc_\bullet&=\bigoplus_{j\in \{1,2\}} \Fc_{(-\eb_j)^+}:\bigoplus_{j\in \{1,2\}} \M_{(-\eb_j)^+}\to \bigoplus_{j\in \{1,2\}} \Nc_{(-\eb_j)^+},\\
\Fc_{\bullet\bullet}&= \Fc_{(-\eb_1-\eb_2)^+}: \M_{(-\eb_1 -\eb_2)^+}\to \Nc_{(-\eb_1 -\eb_2)^+}.
\end{align*}

Consider the following commutative diagram of vector spaces:
\[ 
\xymatrixcolsep{3.2pc}
\xymatrixrowsep{3.7pc}
\xymatrix{
&\Lc_{-\eb_1 -\eb_2}\ar[r]^-{(\kappa_{\Lc})_0}\ar@{^{(}->}[d] & \bigoplus_{j\in \{1,2\}} \Lc_{-\eb_j}\ar[r]^-{(\gamma_{\Lc})_0} \ar@{^{(}->}[d] & \Lc_{0}\ar@{^{(}->}[d] \\
&\Nc_{-\eb_1 -\eb_2}\ar[r]^-{(\kappa_{\Nc})_0}\ar[d]^-{\varphi_{\bullet\bullet}} & \bigoplus_{j\in \{1,2\}} \Nc_{-\eb_j}\ar[r]^-{(\gamma_{\Nc})_0} \ar[d]^-{\varphi_{\bullet}} & \Nc_{0}\ar@{^{(}->}[d]^-{\varphi_\Nc(0,0^+)}   \\
&\Nc_{(-\eb_1 -\eb_2)^+}\ar[r]^-{\kappa^+_{\Nc}} &\bigoplus_{j\in \{1,2\}} \Nc_{(-\eb_j)^+}\ar[r]^-{\gamma^+_{\Nc}} & \Nc_{0^+}\\
&\M_{(-\eb_1 -\eb_2)^+}\ar[r]^-{\kappa^+_{\M}}\ar[u]_-{\Fc_{\bullet\bullet}} &\bigoplus_{j\in \{1,2\}} \M_{(-\eb_j)^+}\ar[r]^-{\gamma^+_{\M}}\ar[u]_-{\Fc_\bullet} & \M_{0^+}.\ar[u]_-{\Fc_{0^+}}}
\]

$\G(0^+)\leq \G(0)+\epsilon \eb_{1}\leq a$, where the second inequality holds by assumption, so since $f_a$ is an injection and $M$ is free, $\Fc_{0^+}$ is an injection as well.

Let $l\in \ker{}(\gamma_{\Lc})_0$ and observe that $l\in \ker \gamma_{\Nc}$ by commutativity of the top-right square.  
Thus, since the second row of the diagram is exact by \cref{teo:free},  there exists \[l'\in \Nc_{-\eb_1 -\eb_2}\]  such that $\kappa_{\Nc}(l') = l$.  To establish the result, it suffices to show that \[l'\in \Lc_{-\eb_1 -\eb_2},\] or equivalently, that $\varphi_{\bullet\bullet}(l^\prime) \in \im  \Fc_{\bullet\bullet}.$

There exists $m^\prime\in \osum_{j\in \{1,2\}} \M_{(-\eb_j)^+}$ such that $\Fc_{\bullet}(m^\prime) = \varphi_{\bullet}(l)$.  By the injectivity of  $\Fc_{0^+}$ and the commutativity of the middle-right and bottom-right squares in the diagram above, $m^\prime \in \ker \gamma_\M^+$.  It follows from \cref{Lem:Freeness_And_Exactness} that the bottom row of the diagram is exact, so there exists $m$ such that $\kappa^+_\M(m) = m^\prime$. Moreover, commutativity of the bottom-left square yields \[\kappa^+_\Nc\comp \Fc_{\bullet\bullet}(m) = \Fc_{\bullet}\comp \kappa^+_\M(m) = \varphi_\bullet(l).\]  On the other hand, from the definition of $l'$, we have that \[\varphi_{\bullet}(l)=\varphi_{\bullet}\comp \kappa_{\Nc}(l')=\kappa^+_{\Nc}\comp \varphi_{\bullet\bullet}(l^\prime).\]  
The injectivity of $\kappa^+_{\Nc}$ implies that $\Fc_{\bullet\bullet}(m) = \varphi_{\bullet\bullet}(l^\prime)$, and (ii) follows. 
\end{proof}

For $M$ a $\ZCat^2$-indexed or $\RCat^2$-indexed persistence module, we define a \emph{presentation} of $M$ to be a morphism $\Phi:F^1\to F^0$ of free persistence modules with $M\cong \coker \Phi$.  When $M$ is $\ZCat^2$-indexed, we'll say $\Phi$ is \emph{minimal} if $\im \Phi\subseteq IF_0$.

From \cref{teo:localformula}, we obtain the following:

\begin{lemma}\label{Lem:Pres_Of_L}
For $f:M\to N$ a morphism of finitely generated free $\RCat^2$-indexed modules, there exists a presentation $\Phi:F^1\to F^0$ of $L^\epsilon(f)$ with $F^0$ and $F^1$ finitely generated, such that $F^1_v\cong 0$ whenever $v\leq a-\epsilon \eb_1$ for some $a\in\R^2$ with $f_a$ an injection.
\end{lemma}

\begin{proof}
For $\G$ a $2$-D grid as above, \cref{teo:localformula}\,(i) tells us that $L^\epsilon(f)|_\G$ is finitely generated.  Thus, by \cref{Prop:FreeResolutionExists}\,(i) there exists a minimal presentation 
\[\Phi': G^1 \to G^0\]
 for $L^\epsilon(f)|_\G$.  The functor $\CoEx_\G$ is easily seen to be exact, so by \cref{prop:gprime}, applying this functor to $\Phi'$ yields a presentation \[\CoEx_\G(\Phi'): \CoEx(G^1) \to \CoEx(G^0)\] for $L^\epsilon(f)$.  We take $\Phi=\CoEx_\G(\Phi')$ and $F^i=\CoEx_\G(G^i)$ for $i=0,1$.  Since $G_0$ and $G_1$ are finitely generated, the same is true for $F_0$ and $F_1$.  

If $f_a$ is an injection, then in view of \cref{teo:localformula}\,(ii), $G^1_z\cong 0$ for all $z\in \Z^2$ with $\G(z)\leq a-\epsilon \eb_{1}$.  If $v\leq a-\epsilon\eb_{1}$, then clearly $\fl_\G(v)\leq a-\epsilon\eb_{1}$, and we thus have \[F^1_v\cong F^1_{\fl_\G(v)}\cong G^1_{\G^{-1}(\fl_\G(v))}\cong 0.\qedhere\]
\end{proof}

\paragraph{Persistence Modules Free Below $a$}
For $a\in \R^2$, let $\RCv$ denote the sub-poset of $\RCat^2$ with objects $\{v\in \R^2\mid v\leq a\}$. We say that an $\RCat^2$-indexed module $M$ is \emph{free below $a$} if there exists a free $\RCat^2$-indexed module $F$ such that the restrictions of $M$ and $F$ to $\RCv$ are isomorphic.

Let $M^a$ denote the $\RCat^2$-indexed module for which $M^a_v=M_{\min(a,v)}$, where \[\min(a,v)=(\min(a_1,v_1),\min(a_2,v_2)),\] with the internal morphisms in $M^a$ induced by those of $M$.  A morphism $f: M\to N$ induces a morphism $f^a: M^a \to N^a$ in an obvious way. 

We omit the following lemma's easy proof:
\begin{lemma}\label{Lem:a-free}
If $M$ is free below $a$, then $M^a$ is free.
\end{lemma}

\begin{lemma}\label{lem:L_is_free_restriction}
For $f$ a morphism of finitely generated free $\RCat^2$-indexed modules  and $a\in \R^2$ with $f_a$ an injection, $L^\epsilon(f)^{a-\epsilon \eb_{1}}$ is free. 
\end{lemma}
\begin{proof}
For $\Phi:F^1\to F^0$ a presentation for $L^\epsilon(f)$ as in the statement of \cref{Lem:Pres_Of_L}, the restrictions of $L^\epsilon(f)$ and $F^0$ to $\RCat^2_{\leq a-\epsilon \eb_{1}}$  are isomorphic.  Thus, $L^\epsilon(f)$ is free below $(a-\epsilon \eb_{1})$.  The result now follows from \cref{Lem:a-free}.
\end{proof}

Here is the first main result of this section:

\begin{proposition}\label{prop:L_i_is_free}
If $f: M\to N$ is a monomorphism of \pfd free $\RCat^2$-indexed modules, then $L^\epsilon(f)$ is free. 
\end{proposition}
\begin{proof}
For $j\in \{0,1,2,\ldots\}$, let $a_j = (j,j)$.  Note that $f^{a_j}: M^{a_j}\to N^{a_j}$ is a monomorphism of finitely generated free persistence modules.  In particular, $f^{a_j}_{a_j}$ is an injection.  
By \cref{lem:L_is_free_restriction} then, $L^\epsilon_i(f)^{a_j -\epsilon \eb_{1}} = L^\epsilon(f^{a_j})^{a_j-\epsilon \eb_{1}}$ is free.  

Letting $L^j=L^\epsilon(f)^{a_j - \epsilon \eb_{1}}$, note that there is a canonical monomorphism $L^j \hookrightarrow L^{j+1}$, so that we may identify $L^j$ with a submodule of $L^{j+1}$, and that $\varinjlim L^j\cong L^\epsilon(f)$.  
We inductively define a basis $\W_{j}$ for each $L^{j}$ such that $\W_j \subseteq \W_{j+1}$: Take $\W_0$ to be any basis for $L^0$.  Now assume that we have defined $\W_j$.   If $\W'$ is any basis for $L^{j+1}$ then \[\W'' = \{w'\in \W'\mid \deg(w')\leq a_j-\epsilon \eb_{1}\}\] is a basis for $L^j$.  Hence, $\W_{j+1} = \W_j\cup (\W'-\W'')$ is a basis for $L^{j+1}$ with $\W_j \subseteq \W_{j+1}$.
Clearly,
\[\W_0\cup (\W_1-\W_0) \cup (\W_2-\W_1)\cup \cdots\]
is a basis for $\varinjlim L_j$, so $\varinjlim L_j\cong L^\epsilon(f)$ is free.  
\end{proof}

\paragraph{Interpolants of $\RE$-free Modules}
For $\epsilon\geq 0$, define an endofunctor $R_\epsilon$ on $\Vect^{\RCat\times \RCat^\op}$ by
\begin{align*}
R_\epsilon(M)_{(s,t)}=
\begin{cases}
 M_{(s,t)}&\textup{ for all }t-s>2\epsilon,\\ 
 0&\textup{ otherwise,}
 \end{cases}
\end{align*}
with the internal maps $\varphi_{R_\epsilon(M)}(-,-)$ and the action of $R_\epsilon$ on morphisms defined in the obvious way.  Note that we have a canonical epimorphism $M\twoheadrightarrow R_\epsilon(M)$.  

\begin{definition}\label{def:RE-Free}
We say that an $\RCat\times \RCat^\op$-indexed module $M$ is \emph{$R_\epsilon$-free} if $M \cong R_\epsilon(F_M)$ for $F_M$ a \pfd free $\RCat\times \RCat^\op$-indexed module.
\end{definition}
Observe that an $\RE$-free module $M$ is interval decomposable, with \[\B(M) = \{\TopRR{a,b}{\epsilon}\mid \GenRR{a,b}\in \B(F_M), \TopRR{a,b}{\epsilon}\neq \emptyset\},\] where
\[ \TopRR{a,b}{\epsilon} = \{(s,t)\in \GenRR{a,b} \mid t-s> 2\epsilon\};\] see \cref{fig:interval_restricted}.
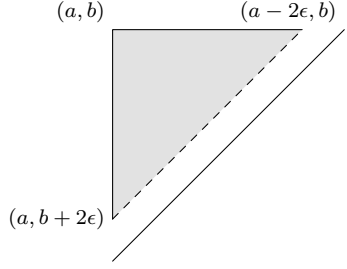
\begin{figure}
\centering
\begin{tikzpicture}[line cap=round,line join=round,>=triangle 45,x=.7cm,y=0.7cm, scale=0.4]
\fill[color=aqaqaq,fill=aqaqaq,fill opacity=0.30] (-5.,-4.) -- (-5,5) -- (4,5) -- cycle;
\draw (-5.,-4.) -- (-5,5) -- (4,5);
\draw[dashed](-5,-4)--(4,5); 
\draw  (-5, -6) -- (6, 5);

\begin{scriptsize}
\draw[color=black] (-5,5) node[above left] {$(a,b)$};
\draw[color=black] (6,5) node[above left] {$(a-2\epsilon,b)$};
\draw[color=black] (-5,-4) node[left] {$(a,b+2\epsilon )$};
\end{scriptsize}
\end{tikzpicture}
\caption{An interval $\protect\TopRR{a,b}{\epsilon}$}
\label{fig:interval_restricted}
\end{figure}    

We omit the easy proof of the following:
\begin{lemma}\label{lem:Retbasis}
$M$ is $R_\epsilon$-free if and only if there exists a set \[\W\subseteq \bigcup_{t-s>2\epsilon} M_{(s,t)}\] such that for any $(t,s)\in \R^2$ with $t-s>2\epsilon$ and $m\in M_{(s,t)}$,  $m$ can be uniquely expressed as a linear combination of elements of $\W$, as in \eqref{eq:freeSum}.
\end{lemma}
In analogy with the free case, we call the set $\W$ above an $\RE$-\emph{basis}.
Finally, we come to the second main result of this section:
\begin{proposition}\label{prop:openrestfree}
Let $f: M\to N$ be a monomorphism of $\RE$-free persistence modules. Then $\RET(L^\epsilon(f))$ is $\RET$-free. 
\end{proposition}
\begin{proof}
Let $\alpha_M:M\to \RE(F_M)$ and $\alpha_N:N\to \RE(F_N)$ be isomorphisms.  The map $\alpha_N \comp f\comp \alpha_M^{-1}: \RE(F_M)\to  \RE(F_N)$ lifts to a map $\tilde f:F_M\to F_N$ such that the following diagram commutes:  
\[
\xymatrix{
M\ar[r]^-{\alpha_M}_-{\cong} \ar[d]^{f} & \RE(F_M)\ar@{.>}[d]^{\RE(\tilde f)} & \ar@{->>}[l] \ar@{.>}[d]^{\tilde f}  F_M\\
N \ar[r]^-{\alpha_N}_-{\cong}& \RE(F_N) & \ar@{->>}[l]F_N.}
\]

Observe that \[\RET(L^\epsilon(f)) \cong \RET(L^\epsilon(\RE(\tilde f))) = \RET(L^\epsilon(\tilde f)),\] where the isomorphism on the left follows from commutativity of the left square in the diagram. Hence, it suffices to show that $\RET(L^\epsilon(\tilde f))$ is $\RET$-free.  Our argument is similar to the proof of \cref{prop:L_i_is_free}.

Let $a^j = (j, -j)$ for $j\in \{0, 1,2,\ldots \}$.  We first show that $\RET(L^\epsilon(\tilde f^{a^j}))$ is $\RET$-free. Note that since $F_M$ and $F_N$ are \pfd, $\tilde f^{a^j}: F_M^{a^j}\to F_N^{a^j}$ is a morphism of finitely generated free persistence modules.  By commutativity of the above diagram, $\RE(\tilde{f})$ is a monomorphism, i.e., $\tilde{f}_{(s,t)}$ is an injection for $t-s>2\epsilon$.  Further, $\tilde f^{a^j}_{(s,t)}$ is also an injection for $t-s>2\epsilon$.  To see this, 
note that there exist $u\leq s$ and $v\geq t$ with $\tilde f^{a_j}_{(s,t)}=\tilde f_{(u,v)}$.  We have $v-u\geq t-s> 2\epsilon$, so $\tilde f^{a_j}_{(s,t)}=\tilde f_{(u,v)}$ is injective.

By \cref{Lem:Pres_Of_L} then, there exists a presentation $\Phi:F^1\to F^0$ for $L^\epsilon(\tilde f^{a^j})$ with $F^0$ finitely generated, such that $F^1_{(s,t)}=0$ whenever $t-s>3\epsilon$.  Thus, $\RET(L^\epsilon(\tilde f^{a^j}))$ is $\RET$-free, as claimed.  

For any $\RCat\times \R^{\op}$-indexed module $Q$ such that $\RET(Q)$ is $\RET$-free, $\RET(Q^a)$ is also $\RET$-free for all $a\in \R^2$: If $\RET(Q)\cong\RET(F)$ for $F$ \pfd and free, then $\RET(Q^a)\cong\RET(F^a)$.  Thus, since $\RET(L^\epsilon(\tilde f^{a^j}))$ is $\RET$-free, $\RET(L^\epsilon(\tilde f^{a^j})^{a^j-\epsilon\eb_1})$ is $\RET$-free as well.  Moreover, \[L^j:=\RET(L^\epsilon(\tilde f)^{a^j-\epsilon\eb_1})=\RET(L^\epsilon(\tilde f^{a^j})^{a^j-\epsilon\eb_1}),\] 
so $L^j$ is also $\RET$-free.

\begin{sloppypar}
Note that we have a canonical monomorphism $L^j \hookrightarrow L^{j+1}$, and that $\varinjlim L^j \cong \RET(L^\epsilon(\tilde f))$.  By choosing an $\RET$-basis for each $L^j$, we may inductively construct an $\RET$-basis for $\RET(L^\epsilon(\tilde f))$ precisely as in the proof of \cref{prop:L_i_is_free}.  By \cref{lem:Retbasis} then, $L^\epsilon(\tilde f)$ is $\RET$-free. \qedhere
\end{sloppypar}
\end{proof}
\begin{figure}
\centering
\begin{subfigure}{.4\textwidth}
\centering
\definecolor{aqaqaq}{rgb}{0.6274509803921569,0.6274509803921569,0.6274509803921569}
\begin{tikzpicture}[line cap=round,line join=round,>=triangle 45,x=.7cm,y=0.7cm, scale=0.32,baseline=0pt]
\clip(-13,-10.) rectangle (12,13.);
\fill[color=aqaqaq,fill=aqaqaq,fill opacity=0.20] (-10,-6) -- (-10,10) -- (6,10) --   cycle;
\draw (-10,-6) -- (-10,10) -- (6,10);
\draw[dashed] (-10,-6) -- (6,10);

\draw (-10,-10) -- (10,10);

\fill[color=aqaqaq,fill=green,fill opacity=0.30] (-8,-4) -- (-8,8) -- (4,8) -- cycle;
\draw (-8,-4) -- (-8,8) -- (4,8);
\draw[dashed] (-8, -4) -- (4,8);

\begin{scriptsize}
\draw[color=black] (-10, 10) node[above] {$(-10, 0)$};
\draw[color=black] (-8,7) node[right] {$(-9, -1)$};
\end{scriptsize}
\end{tikzpicture}
\caption*{The support of $M$ (green) and $N$ (gray)}
\end{subfigure}
\begin{subfigure}{.1\textwidth}
~\end{subfigure}
\begin{subfigure}{.4\textwidth}
\centering
\definecolor{aqaqaq}{rgb}{0.6274509803921569,0.6274509803921569,0.6274509803921569}
\begin{tikzpicture}[line cap=round,line join=round,>=triangle 45,x=.7cm,y=0.7cm, scale=0.32,baseline=0pt]
\clip(-13,-10.) rectangle (12,13.);
\fill[color=aqaqaq,fill=blue,fill opacity=0.20] (-10,-6) -- (-10, 8) -- (2,8) -- (4,10) -- (6,10) --  cycle; 
\draw (-10,-6) -- (-10, 8) -- (2,8) -- (4,10) -- (6,10);
\draw[dashed] (-10, -6) -- (6,10); 
\draw (-10,-10) -- (10,10);
\begin{scriptsize}
\draw[color=black] (2.4, 9) node[left] {$(-4, -1)$};
\draw[color=black] (-10, 8) node[above] {$(-10, -1)$};
\draw[color=white] (-10, 10) node[above] {$(-10, 0)$};
\end{scriptsize}
\end{tikzpicture}
\caption*{The support of $L^\epsilon(f)$}
\end{subfigure}
\caption{Illustration of \cref{ex:tight} in the case $\epsilon=1$}
\label{fig:tight}
\end{figure}
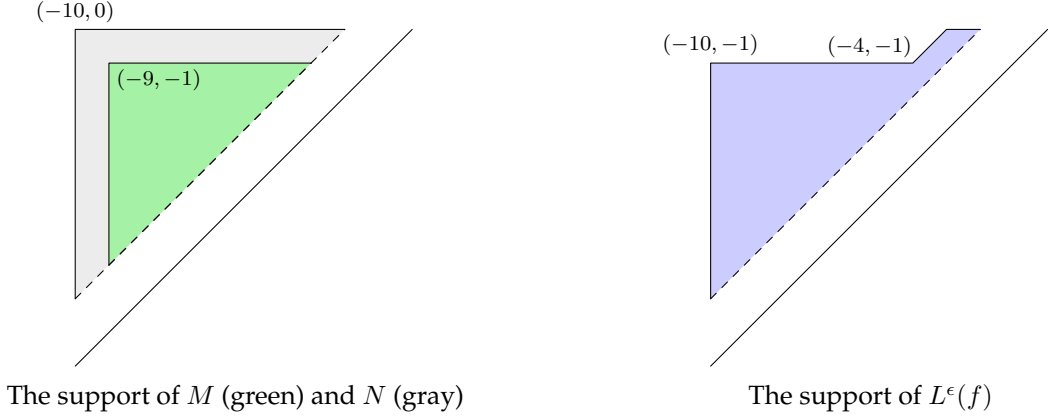

\begin{example}
The previous lemma is tight: let $M=I^{\TopRR{-9\epsilon, -\epsilon}{\epsilon}}$, $N=I^{\TopRR{-10\epsilon, 0}{\epsilon}}$, and $f:M\to N$ be any non-zero morphism. Then $f$ is a monomorphism of $\RE$-free persistence modules, but the persistence module $R_{3\epsilon/2-\delta}(L^\epsilon(f))$ is not $R_{3\epsilon/2-\delta}$-free for any $\delta>0$; see \cref{fig:tight}.
\label{ex:tight}
\end{example}

%% file: freematching.tex
\section{Induced Matching Theorem for Free Multidimensional Persistence Modules}\label{Sec:Free}
Let $f:M\to N$ be a morphism of $\RCat^2$-indexed modules.  For $i\in \{1,2\}$, let $\Hbb_i\subset\R^2$ denote the line $x_i=0$, and for fixed $a\in\Hbb_i$, define the line \[\Q_a := \{a+t\eb_i\mid t\in \mathbb{R}\}.\] 

In \cref{Sec:Derived_1_D_Morphisms}, we associate to each line $\Q_a$ a morphism $\h$ of 1-D persistence modules derived from $f$.  When $M$ and $N$ are free, intervals in the barcodes of the domain and codomain of $\h$ correspond, respectively, to the intervals in $\B(M)$ and $\B(N)$ with an edge lying on $\Q_a$.  We prove that when $f$ is a monomorphsm with $\epsilon$-trivial cokernel, then so is $\h$.  

In \cref{sec:isometryFree}, we use the morphisms $\h$, together with the decomposition \eqref{eq:Li} and the 1-D induced matchings of \cite{bauer2015induced}, to define the matching \[\chi(f):\B(M)\to \B(N)\] induced by a monomorphism $f:M\to N$ of \pfd free modules.  We use this matching to formulate our induced matching theorem for free modules.

In \cref{Sec:Technical_IMT}, we establish a similar induced matching result for monomorphisms of $\RE$-free modules.  

\subsection{Induced Morphisms of 1-D Persistence Modules}\label{Sec:Derived_1_D_Morphisms}

For $c\in \R^2$, we say $c<\Q_a$ if $c=\hat c_i+c_i\eb_i$ for $a>\hat c_i\in \Hbb_i$.  Thus for $i=1$, $c<\Q_a$ if and only if $c$ lies below the horizontal line $\Q_a$; similarly, 
for $i=2$, $c<\Q_a$ if and only if $c$ lies to the left of the vertical line $\Q_a$.

For $M$ an $\RCat^2$-indexed module, define the submodule $M^{\prime\prime}\subseteq M$ by
\begin{equation}
M^{\prime\prime}_b = \{m\in M_b \mid m\in\im \phi_M(c,b)~\text{for some } c<\Q_a\},
\label{mprime}
\end{equation}
and let  $M^\prime := M/M^{\prime\prime}$.  Note that if $M$ is free, then $M^\prime$ and $M^{\prime\prime}$ are both free, and
\begin{align*}
\B(M^{\prime}) &= \{\npGen{c}\in \B(M) \mid c\not <\Q_a\},\\
\B(M^{\prime\prime}) &= \{\npGen{c}\in \B(M) \mid c<\Q_a\}.
\end{align*}

Given a morphism $f: M\to N$, we have that $f(M^{\prime\prime})\subseteq N^{\prime\prime}$, so $f$ induces a morphism $f^\prime: M^\prime\to N^\prime$.  Restricting $f'$ to the line $\Q_a$, we obtain a morphism of 1-D persistence modules  
\begin{equation}
\h:= f^\prime|_{\Q_a} : M^\prime|_{\Q_a} \to N^\prime|_{\Q_a}. 
\label{eq:h}
\end{equation}

\begin{lemma}
If $f$ has $\epsilon \eb_i$-trivial cokernel then $\h$ has $\epsilon$-trivial cokernel.
\label{lem:eitrivialcok}
\end{lemma}

\begin{proof}
We show that $f^\prime$ has $\epsilon \eb_i$-trivial cokernel; the result then follows by restricting the indexing category to $\Q_a$.  For any $b\in \R^2$ and $n\in N_b$, let $[n]\in N'_b$ denote the corresponding coset.  Suppose that $m\in M_{b+\eb_i}$ satisfies $f(m) = \phi_N(b, b+\epsilon \eb_i)(n)$. Then 
\[ f^\prime[m] = [f(m)] = [\phi_N(b, b+\epsilon \eb_i)(n)] = \phi_{N^\prime}(b, b+\epsilon \eb_i)[n].\qedhere
\] 
\end{proof}

\begin{lemma}
Assume that $M$ is free below $a+t\eb_i$ (see end of \cref{sec:cont2disc}) and that $f$ has $\epsilon \eb_i$-trivial cokernel. If $f_{a+t\eb_i}$ is injective, then $\h_{t-\epsilon}$ is injective.  In particular, if $f$ is a monomorphism of free modules, then $\h$ is a monomorphism.
\label{lem:inj}
\end{lemma}

\begin{proof}
Let $b=a+(t-\epsilon)\eb_i$.  $\h_{t-\epsilon}=f^\prime_{b}$, so we need to show that $f^\prime_{b}$ is injective, i.e., that for any $m\in M_{b}$ with $f_{b}(m)\in N''_b$, we have $m\in M^{\prime\prime}_b$. 

Since $f_{b}(m)\in N^{\prime\prime}_b$, $f_{b}(m) = \phi_N(c, b)(n)$ for some $c=\hat c_i+c_i\eb_i$ with $a>\hat c_i\in \Hbb_i$, $c_i\leq (t-\epsilon)$, and $n\in N_c$.  
Since $f$ has $\epsilon \eb_i$-trivial cokernel, there exists $m^\prime\in M_{c+\epsilon \eb_i}^{\prime\prime}$ such that $f_{c+\epsilon \eb_i}(m^\prime) = \phi_N(c, c+\epsilon \eb_i)(n)$. This, together with the following commutative diagram
\begin{center}
\includegraphics{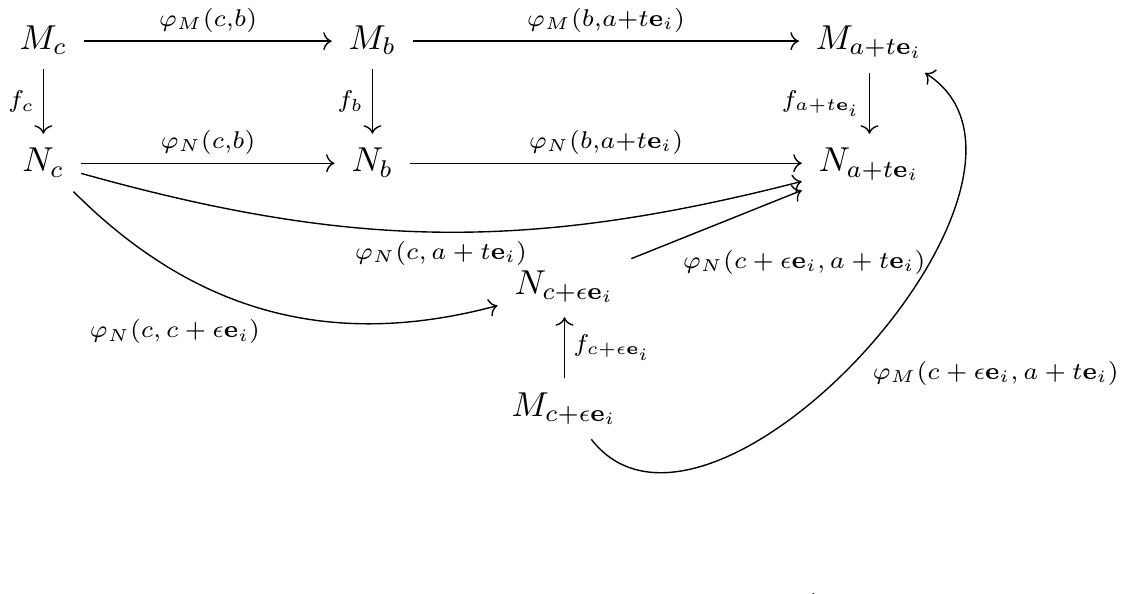} 
\end{center}
yields the chain of equalities
\begin{align*} 
f_{a+t\eb_i}\circ \phi_M(b, a+t\eb_i)(m) &=\phi_N(b, a+t\eb_i) \circ f_{b}(m)\\
 & = \phi_N(b, a+t\eb_i)\circ \phi_N(c, b)(n)\\
 &= \phi_N(c, a+t\eb_i)(n) \\ 
 &= \phi_N(c+\epsilon \eb_i, a+t\eb_i) \circ \phi_N(c, c+\epsilon \eb_i)(n)\\
 &= \phi_N(c+\epsilon \eb_i, a+t\eb_i) \circ f_{c+\epsilon \eb_i}(m^\prime)\\
 &= f_{a+t\eb_i}\circ \phi_M(c+\epsilon \eb_i, a+t\eb_i)(m^\prime). 
\end{align*}
 The injectivity of $f_{a+t\eb_i}$ implies 
\[\phi_M(b, a+t\eb_i)(m) = \phi_M(c+\epsilon \eb_i, a+t\eb_i)(m^\prime).\]  Since $M$ is free below $a+t\eb_i$, it follows that $m\in \im \phi_M(d, b)$, where $d = \hat c_i+\min(c_i+\epsilon,(t-\epsilon))\eb_i$.  Since $d<\Q_a$, we thus have $m\in M^{\prime\prime}_b$, as desired.
\end{proof}

\subsection{Induced Matchings of Free 2-D Persistence Modules}\label{sec:isometryFree}
For $M$ a \pfd $\RCat^2$-indexed module and $a\in \Hbb_i$, let $\B(M;{i,a}):=\B(M^\prime|_{\Q_a}).$  For $f:M\to N$ a morphism of \pfd $\RCat^2$-indexed modules, let \[\chi(f;{i,a}):= \chi(\h) : \B(M;{i,a})\nrightarrow \B(N;{i,a}),\] where $\chi(\bar f)$ is the matching induced by $\h:M^\prime|_{\Q_a} \to N^\prime|_{\Q_a}$; see \cref{Sec:Isometry}.  The matchings $\chi(f;{i,a})$ assemble into a matching
\begin{equation}\label{Eq:Matchings_Along_All_Lines}
\bigsqcup_{a\in\Hbb_i}\chi(f;{i,a}) : \bigsqcup_{a\in\Hbb_i} \B(M;{i,a})\nrightarrow \bigsqcup_{a\in\Hbb_i} \B(N;{i,a}).
\end{equation}
\begin{definition}[Direction-$i$ Induced Matchings]\label{Def:Direction_i_Matchings}
Assume that $M$ and $N$ are free.  We then have a bijection $\sqcup_{a\in \Hbb_i} \B(M;{i,a})\rightarrow \B(M)$ matching $[t, \infty)\in \B(M;{i,a})$ to $\Gen{a+t\eb_i}\in \B(M)$, and similarly for $N$.  By way of these bijections, the matching \eqref{Eq:Matchings_Along_All_Lines} induces a matching \[\chi(f;i): \B(M)\nrightarrow \B(N).\]  We call this \idf{the direction-$i$ matching induced by $f$}.
\label{def:idirmatch}\end{definition}

Now assume that $f: M\hookrightarrow N$ is a monomorphism of \pfd free persistence modules with $\epsilon$-trivial cokernel.  We decompose $f$ using \eqref{eq:Li}:
\begin{equation}\label{Eq:Free2DDecomposition}
M\cong \im f \overset{f_1}{\hookrightarrow} L^\epsilon(f) \overset{f_{2}}{\hookrightarrow} N.
\end{equation}
$L^\epsilon(f)$ is free by \cref{prop:L_i_is_free}.

\begin{theorem}[Induced Matchings of Free Modules]\label{teo:freeIMT}
The composition 
\[
\chi(f):= \chi(f_{2};2)\comp \chi(f_1;1): \B(M)\nrightarrow \B(N)\]is a bijective matching such that for each $\npGen{b}\in \B(M)$, we have $\chi(f)(\npGen{b})  = \npGen{b'}$  where $b_i-\epsilon \leq b'_i \leq b_i$ for $i\in\{1,2\}$. 
\end{theorem}
\begin{proof}
By \cref{prop:Lprops}, for $i\in \{1,2\}$ the inclusion $f_i$ has $\epsilon \eb_{i}$-trivial cokernel. For convenience, we introduce the notation \[L_0 := \im f,\quad L_1 := L^\epsilon(f),\quad \textup{and}\quad L_2 := N.\] For $a\in\Hbb_{i}$, 
\cref{lem:eitrivialcok} and \cref{lem:inj} imply that 
\[\h=f'_i|_{\Q_a}: L_{i-1}^\prime|_{\Q_a} \to L_{i}^\prime|_{\Q_a}\] is a monomorphism with $\epsilon$-trivial cokernel. From \cref{teo:IMT} it follows that 
\[\chi(\h): \B(L_{i-1};i, a) \nrightarrow \B(L_{i};i,a)\] is a bijective matching such that $\chi(\h)[b, \infty) = [b', \infty)$ where $b-\epsilon \leq b' \leq b$. 
Thus, the direction-$i$ matching \[\chi(f_i;{i}): \B(L_{i-1})\nrightarrow \B(L_{i})\] is a bijective and matches $\Gen{a+b\eb_i}$ to $\Gen{a+b'\eb_i}$. 

Hence, $\chi(f): \B(M)\nrightarrow \B(N)$ is a bijective matching with the desired properties. 
\end{proof}

We omit the easy proof of the following:
\begin{proposition}\label{Prop:FreeEpsMatching}
For free $\RCat^2$-indexed modules $M$ and $N$, a matching $\sigma:\B(M)\nrightarrow \B(N)$ is an $\epsilon$-matching if and only if it is bijective and for all $\npGen{b}\in \B(M)$, $\sigma(\npGen{b})=\npGen{b'}$ with $\|b-b'\|_\infty<\epsilon$.  
\end{proposition}

Define a bijection $r_\epsilon: \B(N(\epsilon)) \to \B(N)$ by $r_\epsilon(\npGen{b}) = \Gen{b+(\epsilon,\epsilon)}$.

\begin{corollary}[Isometry Theorem for Free $\RCat^2$-Indexed Modules]\label{Cor: Free2DIso}
P.f.d. free $\RCat^2$-indexed modules $M$ and $N$ are $\epsilon$-interleaved if and only if there exists an $\epsilon$-matching between $\B(M)$ and $\B(N)$.  
\end{corollary}
\begin{proof}
An $\epsilon$-interleaving morphism $f: M\to N(\epsilon)$ is a monomorphism with $2\epsilon$-trivial cokernel.  By \cref{teo:freeIMT}, 
 $\chi(f): \B(M) \nrightarrow \B(N(\epsilon))$ is a bijective matching such that $\chi(f)(\npGen{b})  = \npGen{b'}$ where $b_i-2\epsilon \leq b_i' \leq b_i$ for $i\in\{1,2\}$. The composition \[r_\epsilon\circ \chi(f): \B(M)\nrightarrow \B(N)\] is a bijective matching such that for all $\npGen{b}\in \B(M)$, $\sigma(\npGen{b})=\npGen{b'}$ with $\|b-b'\|_\infty<\epsilon$.  Thus, by \cref{Prop:FreeEpsMatching}, $r_\epsilon\circ \chi(f)$ is an $\epsilon$-matching.  

The converse is a special case of \cref{lem:converseAST}.
\end{proof}

\paragraph{The Difficulty of Defining Induced Matchings for Free $\RCat^3$-Indexed Modules}
We expect that \cref{teo:freeIMT} can be generalized to an induced matching theorem for free $\RCat^n$-indexed modules for any $n$.  
However, the construction of induced matchings given here does not generalize directly to $n\geq 3$.  To explain, the decomposition \eqref{Eq:Free2DDecomposition} does generalize to a decomposition 
\[
M\xhookrightarrow{f_1}  L^\epsilon_1(f)\xhookrightarrow{f_2} \cdots\xhookrightarrow{f_{n-1}}  L^\epsilon_{n}(f) \xhookrightarrow{f_n}  N 
\]
of a monomorphism $f:M\hookrightarrow N$ of free $\RCat^n$-indexed modules with $\epsilon$-trivial cokernel, where for $a\in \R^n$ and $\eb_{[i]}:=\eb_1+\cdots+\eb_i$,  \[L^\epsilon_i(f)_a:= \{n\in N_a | \phi_N(a, a+\epsilon \eb_{[i]})(n) \in \im f\},\] 
and each $f_i$ is the inclusion, so that $f_i$ has $\epsilon \eb_i$-trivial cokernel. 
However, the next example shows that in contrast to the $n=2$ case, $L_i^\epsilon(f)$ needn't be free for $n\geq 3$.
\begin{example}\label{ex:freefail}
Take $N$ to be the free $\RCat^3$-indexed module with generators $a$, $b$, $c$ at respective grades $(1,0,0),$ $(0,1,0)$, $(0,0,1)$, and let $M\subset N$ be the free submodule generated by 
\[\left\{a-c\in N_{(1,0,1)},\ a-b\in N_{(1,1,0)},\ a\in N_{(1,1,1)}\right\},\] 
where by slight abuse of notation, we use the same label for a generator and its image under an internal map in $N$.  Let $f:M\hookrightarrow N$ be the inclusion.  Then \[\left\{a-c\in N_{(1,0,1)},\ a-b\in N_{(1,1,0)},\ b\in N_{(0,1,1)},\ c\in N_{(0,1,1)}\right\} \] is a minimal set of generators for $L^1_1(f)$; clearly, $L^1_1(f)$ is not free.  
\end{example}
When each $L_i^\epsilon(f)$ is free, the construction of this section does extend to give an induced matching $\chi(f):\B(M)\to \B(N)$ with the desired properties.  However, when one or more of the $L_i^\epsilon(f)$ is not free, the construction breaks down.  Thus, a new idea is needed to extend our definition of induced matchings to free $\RCat^n$-indexed modules for $n\geq 3$.  

\subsection{Matchings Induced by Monomorphisms of $\RE$-Free Modules}\label{Sec:Technical_IMT}
Suppose $f:M\to N$ is a morphism of $\RE$-free $\RCat\times\RCat^\op$-indexed modules.  Then for $i\in \{1,2\}$, we can define the direction-$i$ matching \[\chi(f;i):\B(M)\to \B(N)\] in essentially the same way we did for free modules in \cref{Def:Direction_i_Matchings}.  To see this, note that as illustrated in \cref{fig:idirrestricted}, for an $\RE$-free module $M$, we may define bijective matchings 
\begin{equation}\label{Eq:RE_Matchings}
\bigsqcup_{b\eb_2\in \Hbb_1} \B(M; 1,b\eb_2)\nrightarrow \B(M),\qquad  \bigsqcup_{a\eb_1\in \Hbb_2} \B(M; 2,a\eb_1)\nrightarrow \B(M)
\end{equation}
 by matching both $[a, b-2\epsilon)\in \B(M; 1,b\eb_2)$ and $[b, a+2\epsilon)\in \B(M; 2,a\eb_1)$ to $\TopRR{a,b}{\epsilon}$.  The construction of \cref{Def:Direction_i_Matchings} now carries over.
\begin{figure}
\centering
\begin{tikzpicture}[line cap=round,line join=round,>=triangle 45,x=.7cm,y=0.7cm, scale=0.3]
\fill[color=aqaqaq,fill=aqaqaq,fill opacity=0.30] (-5.,-4.) -- (-5,5) -- (4,5) -- cycle;
\draw (-5.,-4.) -- (-5,5) -- (4,5);
\draw[dashed](-5,-4)--(4,5);
\draw  (-5, -5) -- (5, 5);
\draw[dashed] (-5, -4) -- (-10, -4);
\draw[dashed] (-5, 5) -- (-10, 5);
\draw[dashed] (-5, 5) -- (-5, 10);
\draw[dashed] (4, 5) -- (4,10);

\draw[[-), line width=0.5mm] (-10, 5) -- (-10, -4);
\draw[[-), line width=0.5mm] (-5, 10) -- (4, 10);

\begin{scriptsize}
\draw[color=black] (-4.5, -5) node[below] {$(a,a)$};
\draw[color=black] (5,5) node[right] {$(b,b)$};
\draw[color=black] (-5,3.5) node[right] {$(a,b)$};
\draw[color=black] (-7.5, 0) node[thick, right, anchor=north, rotate=-90] {$[b, a+2\epsilon)$};
\draw[color=black] (-0.5, 9.5) node[thick, below] {$[a, b-2\epsilon)$};
\end{scriptsize}
\end{tikzpicture}
\caption{An illustration of the matchings \eqref{Eq:RE_Matchings}, for a single choice of interval in $\B(M)$}
\label{fig:idirrestricted}
\end{figure}
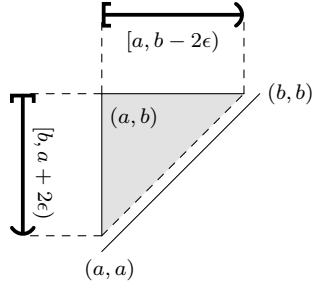

Now let $f: M\to N$ be a monomorphism of $\RE$-free modules with $\epsilon$-trivial cokernel. 
Consider the decomposition of $f$ given by \eqref{eq:Li}:
\[ M\cong \im f \overset{f_1}{\hookrightarrow} L \overset{f_2}{\hookrightarrow} N,\]
where $L=L^\epsilon(f)$.  

For the remainder of this section, we write the functor $\RET$ simply as $R$.  Note that $RL:=R(L)$ is $\RET$-free by \cref{prop:openrestfree}. Hence, we have the following sequence of $\RET$-free modules
\[
\xymatrix{
RM \ar[r]^-{Rf}_-{\cong} & R(\im f) \ar@{^{(}->}[r]^-{Rf_1} & RL\ar@{^{(}->}[r]^-{Rf_2}  & RN,}
\]

where $Rf_1$ and $Rf_2$ have $\epsilon \eb_1$ and $\epsilon (-\eb_2)$-trivial cokernel, respectively. For simplicity, we let $\go := Rf_1\comp Rf$ and $\gt := Rf_2$.  Consider \[\chi(\go;1) : \B(RM)\nrightarrow \B(RL),\] the direction-1 matching associated to $\go$, and \[\chi(\gt;2): \B(RL) \nrightarrow  \B(RN),\] the direction-2 matching associated to $\gt$. 

\begin{proposition}\label{prop:open}
The composite matching
\[\chi:=\chi(\gt;2)\circ \chi(\go;1):  \B(RM) \nrightarrow \B(RN)\]
satisfies 
\begin{enumerate}
\item $\chi(\TopRR{a,b}{3\epsilon/2}) = \TopRR{a',b'}{3\epsilon/2}$ where $a' \leq a\leq a'+\epsilon$ and $b'-\epsilon \leq b\leq b'$, 
\item $\B(RM)_{\epsilon/2} \subseteq \coim \chi$ and $\B(RN)_\epsilon\subseteq \im \chi$.
\end{enumerate}
\end{proposition} 
\begin{proof}
First, note that for any $\RET$-free module $Q$ and $b\eb_2\in \Hbb_1$, each interval in $\B(Q;1,b\eb_2)$ is of the form $[a,b-3\epsilon)$.

Let $\gob$ be the morphism of 1-D persistence modules associated to $\go$ for the point $b\eb_2\in \Hbb_1$. Then $\gob$ has $\epsilon$-trivial cokernel by  \cref{lem:eitrivialcok}.  Further, $\gob_t$ is an injection for all $t<b-4\epsilon$ by \cref{lem:inj}, so in particular $\gob$  has $\epsilon$-trivial kernel.  By \cref{teo:IMT} then, the matching \[\chi(\gob): \B(RM;1,b\eb_2) \nrightarrow \B(RL;1,b\eb_2)\] satisfies
\begin{enumerate}
\item $\left\{ [a, b-3\epsilon)\in \B(RM;1,b\eb_2) \mid a<b-4\epsilon\right\}\subseteq \coim \chi(\gob)$
\item $\left\{ [a, b-3\epsilon)\in \B(RL;1,b\eb_2) \mid a<b-4\epsilon\right\}\subseteq \im \chi(\gob)$
\item $\chi(\gob)[a_1, b-3\epsilon) = [a_2, b-3\epsilon)$ where $a_2 \leq a_1 \leq a_2+\epsilon$. 
\end{enumerate}
For $Q$ any $\RET$-free module and $\TopRR{a,b}{3\epsilon/2}\in \B(Q)$,  $a<b-4\epsilon$ if and only if $\TopRR{a,b}{3\epsilon/2}\in \B(Q)_{\epsilon/2}$. Thus, the direction-1 matching $\chi(\go;1)$ satisfies:
\begin{enumerate}
\item $\chi(\go;1)(\TopRR{a_1,b}{3\epsilon/2}) = \TopRR{a_2,b}{3\epsilon/2}$ where $a_2 \leq a_1 \leq a_2+\epsilon$
\item $\B(RM)_{\epsilon/2} \subseteq \coim \chi(\go;1)$ and $\B(RL)_{\epsilon/2}\subseteq \im \chi(\go;1)$. 
\end{enumerate}
By the symmetric argument, 
the direction-2 matching $\chi(\gt; 2)$ satisfies:
\begin{enumerate}
\item $\chi(\gt;2)(\TopRR{a,b_1}{3\epsilon/2}) = \TopRR{a,b_2}{3\epsilon/2}$ where $b_2-\epsilon \leq b_1\leq b_2$,
\item $\B(RL)_{\epsilon/2}\subseteq \coim \chi(\gt;2)$ and $\B(RN)_{\epsilon/2} \subseteq \im \chi(\gt;2)$. 
\end{enumerate}
It follows that $\chi(\TopRR{a,b}{3\epsilon/2}) = \TopRR{a',b'}{3\epsilon/2}$ where $a'\leq a\leq a'+\epsilon$ and $b'-\epsilon \leq b \leq b'$, as desired. 
Moreover, \[\B(RN)_{\epsilon} \subseteq  \chi(\gt;2)(\B(RL)_{\epsilon/2}) \subseteq  \chi(\gt;2)(\im \chi(\go;1)\cap \coim \chi(\gt;2))= \im \chi,\]
and
\[
\chi(\go;1)(\B(RM)_{\epsilon/2}) \subseteq \B(RL)_{\epsilon/2} \subseteq \coim \chi(\gt;2).
\]
The latter shows that $\B(RM)_{\epsilon/2}\subseteq \coim \chi$. 
\end{proof}

\begin{corollary}[Induced Matchings of $\RE$-Free Modules]\label{cor:IMTRFree}
Let $f: M\to N$ be a monomorphism of $\RE$-free modules with $\epsilon$-trivial cokernel. Then we have a matching
\[\chi(f): \B(M) \nrightarrow \B(N)\]
satisfying
\begin{enumerate}
\item $\chi(\TopRR{a,b}{\epsilon}) = \TopRR{a',b'}{\epsilon}$ where $a' \leq a\leq a'+\epsilon$ and $b'-\epsilon \leq b\leq b'$, 
\item $\B(M)_{\epsilon} \subseteq \coim \chi$ and $\B(N)_{\frac{3}{2}\epsilon}\subseteq \im \chi$.
\end{enumerate}
\end{corollary}

%% file: summands.tex
\section{Proof of the Block Stability Theorem}\label{Sec:Block_Stability_Theorem}
In this section, we complete the proof of our main stability result for block decomposable modules.  Throughout, we regard block-decomposable modules as $\RCat^{\op}\times \RCat$-indexed modules.

\subsection{Decomposition of Interleavings}
\begin{definition}\label{def:Mdecomp}
For a block decomposable module $M$, we choose summands
\begin{align*}
M^\oo &\cong \bigoplus_{(a,b)_\bd \in \B(M)^\oo} I^{(a,b)_\bd}&  M^\co &\cong \bigoplus_{\langle a,b\rangle_\bd \in \B(M)^\co} I^{\langle a,b\rangle_\bd }\\
~M^\oc &\cong \bigoplus_{\langle a,b\rangle_\bd \in \B(M)^\oc} I^{\langle a,b\rangle_\bd}& M^\cc &\cong \bigoplus_{\langle a,b\rangle_\bd\in \B(M)^\cc} I^{\langle a,b\rangle_\bd},
\end{align*}
such that $M= M^\oo\oplus M^\co\oplus M^\oc\oplus M^\cc$.  For $\star\in \{\co,\oc,\cc,\oo\}$, we say \emph{$M$ is of type $\star$} if $M$ is \pfd and $M=M^\star$.
\end{definition}
For $f: M\to N$ a morphism of block decomposable modules and $\star,\dag\in \{\co,\oc,\cc,\oo\}$, let $f^{\star, \dag}: M^{\star}\to N^{\dag}$ denote the morphism obtained by pre-composing $f$ with the inclusion $M^\star \hookrightarrow M$ and post-composing with the projection $N\twoheadrightarrow N^{\dag}$.
\begin{lemma}\label{lem:homI}
For block-decomposable modules $M$ and $N$, $\Hom(M^\star,N^\dag)=0$ whenever 
\[(\star,\dag)\in \{(\oo,\co),(\oo,\oc),(\oo,\cc),(\co,\oc),(\co,\cc),(\oc,\co),(\oc,\cc)\}.\]
\end{lemma}
\begin{proof}
We show that $\Hom(M^\oo,N^\co)=0.$  Similar arguments apply to the remaining cases.  It suffices to consider the case that $M$ and $N$ are indecomposables.  Assume to the contrary that we have $f\in \Hom(I^{(a,b)_\bd},I^{[c,d)_\bd})$ with $f_{(x,y)}\neq 0$. Then \[a < x \leq y < b,\qquad y<d,\] and by choosing $x^\prime < a$ we obtain the following commutative diagram
\[
\xymatrix{
k = I^{(a,b)_\bd}_{(x, y)}\ar[r]\ar[d]_-{f_{(x,y)}}   &  I^{(a,b)_\bd}_{(x^\prime, y)}=0\ar[d]^-{f_{(x^\prime, y)}} \\
k = I^{[c,d)_\bd}_{(x, y)}\ar[r]^-{\id}  &  I^{[c,d)_\bd}_{(x^\prime, y)} =k,}
\]

contradicting that $f_{(x,y)}\ne 0$.  This shows that $\Hom(I^{(a,b)_\bd},I^{[c,d)_\bd})=0.$  The same argument shows that $\Hom(I^{(a,b)_\bd},I^{(-\infty,d)_\bd})=0.$
\end{proof}

\begin{proposition}\label{Interleavings_On_Four_Types}
If $(f,g)$ is an $\epsilon$-interleaving pair between $M$ and $N$, then so is $(f^{\star,\star}, g^{\star,\star})$ for any $\star\in  \{\oo,\co,\oc,\cc\}.$  In particular, $f^{\star,\star}$ and $g^{\star,\star}$ have $2\epsilon$-trivial kernel and cokernel.
\label{lem:interleavingsplit}
\end{proposition}
\setlength\BAextrarowheight{5pt}
\begin{proof}
By decomposing $M$ and $N$ as in \cref{def:Mdecomp} and  applying Lemma \ref{lem:homI}, we can express $f$ in matrix form as
\[
f = \begin{blockarray}{ccccl}
\begin{block}{[cccc|l]}
M^\oo & M^\co & M^\oc & M^\cc & \\
\BAhline
f^{\oo, \oo} & f^{\co, \oo} & f^{\oc, \oo} & f^{\cc, \oo} & N^\oo(\epsilon) \\
0 & f^{\co, \co} & 0 & f^{\cc, \co}&N^\co(\epsilon)   \\
0 & 0 & f^{\oc, \oc} & f^{\cc, \oc}&N^\oc(\epsilon) \\
0 & 0 & 0 & f^{\cc, \cc}&N^\cc(\epsilon) \\       
\end{block}
\end{blockarray},
\] and similarly for $g(\epsilon)$. Since $g(\epsilon)\circ f = \phi_{M^{\oo}}^{2\epsilon}\oplus \phi_{M^{\co}}^{2\epsilon}\oplus \phi_{M^{\oc}}^{2\epsilon}\oplus\phi_{M^{\cc}}^{2\epsilon}$, 
we may write $g(\epsilon)\circ f$ in matrix form as
\[
g(\epsilon)\circ f = \begin{blockarray}{ccccl}
\begin{block}{[cccc|l]}
M^\oo & M^\co & M^\oc & M^\cc & \\
\BAhline
g^{\oo, \oo}(\epsilon)\circ f^{\oo, \oo} & 0 & 0 & 0 & M^\oo(2\epsilon) \\
0 & g^{\co, \co}(\epsilon)\circ f^{\co, \co} & 0 & 0 &M^\co(2\epsilon)   \\
0 & 0 & g^{\oc, \oc}(\epsilon)\circ f^{\oc, \oc} & 0 &M^\oc(2\epsilon) \\
0 & 0 & 0 & g^{\cc, \cc}(\epsilon)\circ f^{\cc, \cc}&M^\cc(2\epsilon) \\       
\end{block}
\end{blockarray},
\]
and the following equality is immediate:
\begin{align*}
&\quad\left(g^{\oo, \oo}(\epsilon)\circ f^{\oo, \oo}\right)\oplus \left(g^{\co, \co}(\epsilon)\circ f^{\co, \co}\right) \oplus \left( g^{\oc, \oc}(\epsilon)\circ f^{\oc, \oc}\right) \oplus \left(g^{\cc, \cc}(\epsilon)\circ f^{\cc, \cc}\right) \\
&= 
\phi_{M^{\oo}}^{2\epsilon}\oplus \phi_{M^{\co}}^{2\epsilon}\oplus \phi_{M^{\oc}}^{2\epsilon}\oplus\phi_{M^{\cc}}^{2\epsilon}.
\end{align*}
The result follows by applying the symmetric argument to the composition $f(\epsilon)\circ g$. 
\end{proof}

Thus we can study algebraic stability for block decomposables by considering an interleaving morphism on each of four subtypes individually.

\begin{remark}\label{Rem:Induced_Matchings_Four_Types}
In view of \cref{Interleavings_On_Four_Types}, one might wonder whether $\epsilon$-triviality of the (co-)kernel of a morphism $f:M\to N$ is inherited by $f^{\star,\star}$, for $\star\in \{\oo,\co,\oc,\cc\}$.  In fact, the answer is no:  It can be shown that if $f: M\to N$ has $\epsilon$-trivial kernel and cokernel, then so have the three morphisms $f^{\cc, \cc}, f^{\co, \co}$ and $f^{\oc, \oc}$, and the morphism $f^{\oo, \oo}$ has $\epsilon$-trivial kernel and $2\epsilon$-trivial cokernel.  This result is tight, as demonstrated by the following example.
\end{remark}

\begin{example}
Let $M=I^{(0, \epsilon]_\bd}\oplus I^{[3\epsilon, 4\epsilon)_\bd}$ and $N = I^{(0, 4\epsilon)_\bd}$. Let $f_1: I^{(0, \epsilon]_\bd}\to N$ and $f_2: I^{[3\epsilon, 4\epsilon)_\bd}\to N$ be any two non-zero morphisms and define $f(m_1, m_2) = f_1(m_1) + f_2(m_2)$. Then $f$ has $\epsilon$-trivial kernel and cokernel, but the cokernel of $0=M^\oo  \to N^\oo=N$ is $2\epsilon$-trivial and not $\delta$-trivial for any $\delta<2\epsilon$. 
\end{example}

\subsection{An Induced Matching Theorem}
\label{sec:generalIMT}

We establish the block stability theorem (\cref{teo:IMTinterleaving}) by separating the interleaving morphism $f$ into its four components via \cref{Interleavings_On_Four_Types}, and studying each of them independently.  In fact, \cref{teo:IMTinterleaving} is an easy corollary of \cref{Interleavings_On_Four_Types} and the following result:

\begin{theorem}[Induced Matchings of Block Decomposables]\label{Thm:BLK_IMT}
For fixed $\star\in \{\cc,\oo,\co,\oc\}$, let $M$ and $N$ be block decomposable modules of type $\star$, and let $f:M\to N$ be a morphism with $\epsilon$-trivial kernel and cokernel.  Then we can define an explicit matching \[\chi(f):\B(M)\nrightarrow \B(N)\] such that for $\chi(f)\langle a, b\rangle_{\BL} = \langle a', b'\rangle_{\BL}$, 
\begin{enumerate}[(i)]
\item if $\star=\co$, then $\B(M)_\epsilon \subseteq \coim \chi(f)$, $\B(N)_\epsilon \subseteq \im \chi(f)$, and
\[a-\epsilon\leq a' \leq a \qquad\qquad b-\epsilon \leq b' \leq b,\]
\item if $\star=\oc$, then $\B(M)_\epsilon \subseteq \coim \chi(f)$, $\B(N)_\epsilon \subseteq \im \chi(f)$, and
\[a\leq a' \leq a+\epsilon \qquad\qquad b \leq b' \leq b+\epsilon,\]
\item if $\star=\cc$, then $\B(M)=\coim \chi(f)$, $\B(N)=\im \chi(f)$, and 
\[a-\epsilon\leq a' \leq a\qquad\qquad b\leq b' \leq b+\epsilon,\]
\item if $\star=\oo$, then $\B(M)_{\frac{5}{2}\epsilon}\subseteq \coim \chi(f)$, $\B(N)_{2\epsilon}\subseteq \im\chi(f)$, and
\[a\leq a' \leq a+\epsilon\qquad\qquad b-\epsilon \leq b' \leq b.\]
\end{enumerate}
\end{theorem}

\begin{proof}[Proof of \cref{teo:IMTinterleaving} from \cref{Thm:BLK_IMT}]
For $Q$ an $\RCat^{\op}\times \RCat$-indexed module, let $\bar R(Q)$ denote the $\RCat^{\op}\times \RCat$-indexed module given by
\begin{align*}
\bar R(Q)_{(s,t)}=
\begin{cases}
 Q_{(s,t)}&\textup{ for all }t-s\geq 0,\\ 
 0&\textup{ otherwise,}
 \end{cases}
 \end{align*}
with the internal maps $\varphi_{\bar R(Q)}(-,-)$ inherited from $Q$.  We have an obvious morphism $\pi_Q:  Q\to \bar R(Q)$.
 
If $M$ and $N$ are block decomposable modules, then $\bar R(N(\epsilon))$ is block decomposable.  If $f: M\to N(\epsilon)$ is an $\epsilon$-interleaving morphism, then as mentioned in  \cref{Rem:Interleaving_And_Small_(Co)Kernel}, $f$ has $2\epsilon$-trivial kernel and cokernel, and the same is true for $\pi_{N(\epsilon)}\circ f: M\to \bar R(N(\epsilon))$.  By \cref{lem:interleavingsplit} then, for $\star\in \{\co, \oc, \cc, \oo\}$, $f^{\star, \star}: M^\star\to \bar R(N(\epsilon))^\star$ has $2\epsilon$-trivial kernel and cokernel as well. 

Let $r_\epsilon^\star: \B(\bar R(N(\epsilon)))^\star \nrightarrow \B(N)^\star$ be the matching given by
\[ r_\epsilon^\star\langle b,d\rangle_\blk = \left\{ 
\begin{array}{c l}
\langle a+\epsilon, b+\epsilon\rangle_\blk  &\text{if $\star = \co$,}
\\
\langle a-\epsilon, b-\epsilon\rangle_\blk  &\text{if $\star = \oc$,}\\
\langle a+\epsilon, b-\epsilon\rangle_\blk  &\text{if $\star = \cc$,}\\
\langle a-\epsilon, b+\epsilon\rangle_\blk  &\text{if $\star = \oo$,}
 \end{array}
\right.\]
If $\star\in \{\co,\oc,\cc\}$, then $r_\epsilon^\star$ is bijective; in the case that $\star=\oo$, $r_\epsilon^\oo$ matches all blocks of $\B(\bar R(N(\epsilon)))^\oo$ and all blocks $(a,b)_\blk \in \B(N)^\oo$ with $b-a> 2\epsilon$.

Let $\chi(\pi_{N(\epsilon)}\circ f^\star): \B(M)^\star\nrightarrow \B(\bar R(N(\epsilon)))^\star$ be the matching given by \cref{Thm:BLK_IMT}. We define the matching $\chi:\B(M)\nrightarrow \B(N)$ in the statement of \cref{teo:IMTinterleaving} as the (disjoint) union of the four matchings 
\[\big\{r^\star_\epsilon\circ \chi(\pi_{N(\epsilon)}\circ f^\star): \B(M)^\star\nrightarrow \B(N)^\star \mid \star\in \{\co, \oc, \cc, \oo\}\,\big\}.\]  It follows from \cref{Thm:BLK_IMT}, the definitions of the matchings $r_\epsilon^\star$, and \cref{lem:Erosion_And_Blocks}\,(i) that $\chi$ has the desired properties.  
\end{proof}

The remainder of this section is devoted to the proof of \cref{Thm:BLK_IMT}.  
The cases $\star\in \{\co,\oc\}$ can be understood in terms of an equivalence with $\RCat$-indexed persistence, whereas our proofs for the cases $\star\in \{\cc,\oo\}$  build on our results for free and $\RE$-free modules from  \cref{Sec:Free}.

\subsubsection{The cases $\star=\co$ and $\star=\oc$}\label{sec:fcoco}
As the arguments for \cref{Thm:BLK_IMT}\,(i) and (ii) are essentially identical, we will only prove (i). We shall see that the result follows easily from \cref{teo:IMT}. 

Note that if $M$ is of type $\co$, the shift map $\phi_M((x,y), (x^\prime, y))$ is an isomorphism for all $x^\prime \leq x$. Hence, there is a functorial way to identify $M$ with an $\R$-indexed module $M^\ORD$: Define 
\begin{align*}
M^\ORD_t&:= M_{(t,t)},\\
\phi_{M^\ORD}(t, t^\prime)&:= \phi_M((t^\prime, t^\prime), (t, t^\prime))^{-1}\circ\phi_M((t,t), (t, t^\prime)),
\end{align*}
 and for $f:M\to N$ a morphism of modules of type $\co$, define $f^\ORD:M^\ORD\to N^\ORD$ by \[f^\ORD_t:= f_{(t,t)}.\] 
\begin{lemma}\label{Lem:COKerAndCoker}
Let $M$ and $N$ be of type $\co$, and let $f: M \to N$ have $\epsilon$-trivial kernel and cokernel. Then $f^\ORD$ has $\epsilon$-trivial kernel and cokernel. 
\end{lemma}
\begin{proof}
Since $f$ has $\epsilon$-trivial kernel, $\varphi_{\ker f}((t,t),(t-\epsilon,t+\epsilon))=0$, so since $\varphi_M((t,t+\epsilon),(t-\epsilon,t+\epsilon))$ is an isomorphism, we also have  
\[\varphi_{\ker f}((t,t),(t,t+\epsilon))=0.\]  Similarly, \[\varphi_{\coker f}((t,t),(t,t+\epsilon))=0.\] 
Thus, the result follows from the following two commutative diagrams:
\[\xymatrix{
\ker f^\ORD_t\ar[dd]\ar[r]^-= & \ker f_{(t,t)}\ar[d]^-0 & & \coker f^\ORD_t\ar[dd]\ar[r]^-=& \coker f_{(t,t)}\ar[d]^-0\\
~ & \ker f_{(t, t+\epsilon)}\ar[d]^-\cong& & & \coker f_{(t, t+\epsilon)}\ar[d]^-\cong.\\
\ker f^\ORD_{t+\epsilon} & \ker f_{(t+\epsilon, t+\epsilon)}\ar[l]^-=& & \coker f^\ORD_{t+\epsilon} & \coker f_{(t+\epsilon, t+\epsilon)}\ar[l]^-=
}\]
\end{proof}

\begin{proof}[Proof of \cref{Thm:BLK_IMT}\, (i)]
It is easy to see that for $\langle a,b\rangle_\blk$ a block of type $\co$, $(I^{\langle a,b\rangle})^\ORD = I^{\langle a,b\rangle}$, and more generally, that for any module $Q$ of type $\co$, $Q^\ORD \cong \bigoplus_{\langle a,b\rangle_\BL\in \B(Q)} I^{\langle a,b\rangle}$.  We therefore have a bijection $\B(Q)\to \B(Q^\ORD)$ which matches $\langle a,b\rangle_\BL$ to $\langle a,b\rangle$.  For $f:M\to N$ a morphism of modules of type $\co$ with $\epsilon$-trivial kernel and cokernel, the matching $\chi(f^\ORD):\B(M^\ORD) \nrightarrow \B(N^\ORD)$ thus induces a matching $\chi(f):\B(M) \nrightarrow \B(N)$.  It follows from \cref{Lem:COKerAndCoker} and \cref{teo:IMT} that $\chi(f)$ has the desired properties.
\end{proof}

\subsection{Proof of \cref{Thm:BLK_IMT}\, (iii)}

\subsubsection{Further Decomposition of a Module of Type $\cc$}
To prove \cref{Thm:BLK_IMT}\,(iii), we shall separately match the four types of closed intervals $[a,b]_\BL$, $(-\infty, b]_\BL, [a, \infty)_\BL$ and $(-\infty, \infty)_\BL$.  First, much as we decomposed a block decomposable module into four summands in \cref{def:Mdecomp}, we choose a further decomposition of a module $M$ of type $\cc$ into four submodules \[ M=M^\cii\oplus M^\cif \oplus M^\cfi\oplus M^\cff,\] where 
\begin{align*}
M^\cii \cong &\bigoplus_{(-\infty, \infty)_\BL \in \B(M)} I^{(-\infty, \infty)_\BL} & M^\cif \cong &\bigoplus_{(-\infty, b]_\BL\in \B(M)} I^{(-\infty, b]_\BL}\\
M^\cfi \cong &\bigoplus_{[a, \infty)_\BL \in \B(M)} I^{[a, \infty)_\BL} & M^\cff \cong &\bigoplus_{[a,b]_\BL\in \B(M)} I^{[a,b]_\BL}. 
\end{align*}
For $N$ of type $\cc$ and $\dag\in \{\cii, \cif, \cfi, \cff\}$, we let $f^\dag: M^\dag \to N^\dag$ be the morphism obtained by the the composition $M^\dag \hookrightarrow M \xrightarrow{f} N \twoheadrightarrow N^\dag$, where the first morphism is inclusion and the last is projection.

\begin{proposition}\label{prop:closedSplit}
If $f: M \to N$ is a monomorphism with $\epsilon$-trivial cokernel, then so is $f^{\dag}: M^{\dag}\to N^{\dag}$ for $\dag\in \{\cii, \cif, \cfi, \cff\}$. 
\end{proposition}
\begin{proof}
We shall prove the result for $\dag=\cii$.  The proofs of the three remaining cases are similar.  Using an argument similar to the proof of \cref{lem:homI}, it is easy to see that $\Hom(M^{\cii}, N^{\dag}) = 0$ for $\dag\in \{\cif, \cfi, \cff \}.$ Hence, $f^{\cii}$ is a monomorphism.

Since $\coker f$ is $\epsilon$-trivial, for any $y-x\geq 2\epsilon$ and $n\in N^{()}_{(x,y)}$, there exists $m \in M_{(x,y)}$ with $f(m) = n$. Write \[m=m^\cii+m^\cif+m^\cfi+m^\cff\] for $m^\dag \in M^\dag$.  We shall argue that $m^\cif = m^\cfi = m^\cff = 0$, so that $f^\cii(m^\cii) = n$.  It follows that $\coker f^{()}$ is $\epsilon$-trivial.

To arrive at a contradiction, assume that $m^\cff \neq 0$.  By the structure of $M^\cff_{(x,y)}$, we may choose sufficiently large $x'>y$ such that for $y' = x'+2\epsilon$, we have
\begin{equation}\label{eq:shiftMinf}
\phi_M((x,y), (x, y'))(m^\cff) \not\in \im \phi_M((x', y'), (x, y')).
\end{equation}
Consider the unique element $n'\in N^{()}_{(x', y')}$ such that
\[ \phi_N((x', y'), (x, y'))(n') = \phi_N((x,y), (x, y'))(n).\]
Since $\coker f$ is $\epsilon$-trivial, $n'\in \im f$.  That is, there exists $m' \in M_{(x', y')}$ such that $n'= f_{(x', y')}(m')$. Hence, 
\begin{align*}
\left(f\circ \phi_M((x', y'), (x, y'))\right)(m') &= \phi_N((x', y'), (x, y'))(n')\\
&=\phi_N((x,y), (x, y'))(n)\\
&= \left(f\circ \phi_M((x,y), (x, y'))\right)(m).
\end{align*}
This, together with the injectivity of  $f$, implies 
\[\phi_M((x,y), (x, y'))(m) = \phi_M((x', y'), (x, y'))(m').\]
Letting $m'^{\cff}$ denote the component of $m'$ in $M^\cff_{(x',y')}$, it follows that
\[\phi_M((x,y), (x, y'))(m^\cff) = \phi_M((x', y'), (x, y'))(m'^\cff),\]
contradicting that $\phi_M((x,y), (x, y'))(m^\cff)\not\in \im \phi_M((x', y'), (x, y'))$.  Thus, $m^\cff = 0$.  

Similarly, one can show that $m^\cif = m^\cfi = 0$. 
\end{proof}

\subsubsection{The Matching $\chi(f)$}
If $M$ and $N$ are of type $\cc$ and $f:M\to N$ has $\epsilon$-trivial kernel and cokernel, then in fact $f$ is a monomorphism.  By \cref{prop:closedSplit} we may split $f$ into four monomorphisms $f^\dag:M^\dag\to N^\dag$ with $\epsilon$-trivial cokernel.  We take the matching $\chi(f):\B(M)\nrightarrow \B(N)$ to be the disjoint union of four matchings \[\big\{\chi(f^\dag):\B(M^\dag)\nrightarrow \B(N^\dag) \mid \dag\in \{\,\cii, \cif, \cfi, \cff\,\}\,\big\}.\]
For $\dag\in \{\,\cii, \cif, \cfi\,\}$ we define the matching $\chi(f^\dag)$ as follows:
\noindent \begin{itemize}
\item[$\cii$: ]  A morphism between $M^{\cii}$ and $N^{\cii}$ is a monomorphism with $\epsilon$-trivial cokernel if and only if it is an isomorphism. Thus, $\B(M^\cii) = \B(N^\cii)$; we take $\chi(f^\cii)$ to be the identity.  

\item[$\cfi$: ]  $\phi_{M^\cfi}((x,y), (x^\prime, y))$ is an isomorphism for all $x^\prime\leq x$, and similarly for $N^\cfi$, so we may define the matching $\chi(f^\cfi)$ in essentially the same way we defined the induced matching of \cref{Thm:BLK_IMT}\, (i).  The same argument used to prove \cref{Thm:BLK_IMT}\, (i) shows that  $\chi(f^\cfi)$ is bijective, and that if $\chi(f^\cfi)[a,\infty)_\BL = [a', \infty)_\BL$, then $a-\epsilon \leq a'\leq a$.

\item[$\cif$: ] \begin{sloppypar}We define $\chi(f^\cif)$ in essentially the same way as for $\chi(f^\cfi)$.  $\chi(f^\cif)$ is bijective, and if $\chi(f^\cfi)(\infty,b]_\BL = (\infty,b']_\BL$, then $b-\epsilon \leq b'\leq b$. \end{sloppypar}
\end{itemize}
To finish the proof of \cref{Thm:BLK_IMT}\,(iii), it remains to define the matching $\chi(f^\cff)$ and verify that if 
$\chi(f^\cfi)[a,b]_\BL = [a',b']_\BL$, then
\begin{equation*}
a\leq a' \leq a+\epsilon\qquad \text{and}\qquad b-\epsilon\leq b'\leq b.\qedhere
\end{equation*} 
In what follows, we define $\chi(f^\cff)$ via the induced matching construction for free 2-D persistence modules of \cref{sec:isometryFree}. 

\subsubsection{The Matching $\chi(f^\cff)$}
Letting $\iem:\INTR\hookrightarrow \RCat^{\op}\times \RCat$ denote the inclusion, we define an endofunctor $\overleftarrow{(-)}$ on $\Vect^{\RCat^{\op}\times \RCat}$ by

\[\overleftarrow{(-)}:=\Ran_\iem(-) \comp\, (-)|_{\INTR}.\]
Thus, for $(s,t)\in \R^2$ and $\overleftarrow{(s,t)}\subset \U$ given by
\[
\overleftarrow{(s,t)}:=\{(x,y)\in \U\mid x\leq s, y\geq t \},
\]
we have $\overleftarrow{M}_{(s,t)}=\varprojlim M|_{\overleftarrow{(s,t)}}$ for any $\RCat^{\op}\times \RCat$-indexed module $M$.

\paragraph{Properties of $\protect\overleftarrow{(-)}$ on Modules of Type $\ccc$}
The following lemma is illustrated by \cref{fig:freeext}; we omit the proof.  
\begin{lemma}\label{Lem:OverleftarrowForIntervals}
For any $a,b\in \R$, we have 
\[\overleftarrow{I^{[a,b]_\blk}}\cong I^{\GenR{b,a}}.\]
\end{lemma}
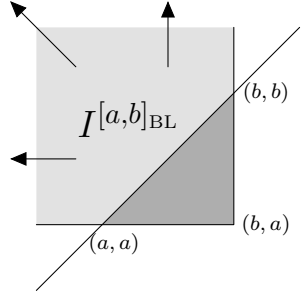
\begin{figure}
\centering
\definecolor{aqaqaq}{rgb}{0.6274509803921569,0.6274509803921569,0.6274509803921569}
\begin{tikzpicture}[line cap=round,line join=round,>=triangle 45,x=.7cm,y=0.7cm, scale=0.25]
\clip(-13.,-10.) rectangle (13.,13.);
\fill[color=aqaqaq,fill=aqaqaq,fill opacity=0.30] (-5.,-5.) -- (-10,-5) -- (-10, 10) -- (5,10) --  (5,5) -- cycle;

\fill[color=aqaqaq, fill=aqaqaq, fill opacity=0.85] (-5,-5) -- (5, -5) -- (5,5) -- cycle;

\draw (-10,-10) -- (10,10);
\draw (-5,-5) -- (-10, -5);
\draw (5,10) -- (5,5) -- (5,-5) -- (-5,-5);
\node[font=\fontsize{15}{144}\selectfont] at (-3,3) {$I^{[a,b]_\BL}$};
\draw[->] (-7,0) -- (-12, 0);
\draw[->] (0,7) -- (0,12);
\draw[->] (-7,7) -- (-12,12);
\begin{scriptsize}
\draw[color=black] (-4.1, -5.1) node[below] {$(a,a)$};
\draw[color=black] (5.1,5) node[right] {$(b,b)$};
\draw[color=black] (5,-5) node[right] {$(b,a)$};

\end{scriptsize}
\end{tikzpicture}
\caption{The image under $\protect\overleftarrow{(-)}$ of the block module $I^{[a,b]_\BL}$ (in light gray) is a free module with a single generator at $(b,a)$}
\label{fig:freeext}
\end{figure}
We say a module $M$ is \emph{of type $\ccc$} if $M$ is of type $\cc$ and $M=M^\cff$.  
\begin{lemma}\label{Lemma:OvLeftArFinite_Dim}
For each module $M$ of type $\ccc$, $\overleftarrow{M}$ is \pfd
\end{lemma}

\begin{proof}
For $(s,t)\in \R^2$, let $v=(v_1,v_2)$, where $v_1=\min(s,t)$ and $v_2=\max(s,t)$.  Note that for $a,b\in \R$, if $(s,t)\in {\GenR{b,a}}$, then $v\in [a,b]_\blk$.  In view of \cref{Lem:OverleftarrowForIntervals} then, it follows that 
\[\dim (\oplus_{\J\in \B(M)}\overleftarrow{I^\J})_{(s,t)}\leq \dim\left(\oplus_{\J\in \B(M)}I^\J\right)_v=\dim M_v<\infty.\]
Thus, $\oplus_{\J\in \B(M)}\overleftarrow{I^\J}$ is \pfd\ \ Since $M$ is also \pfd, it follows from \cref{Rem:Preservation_Of_Products_pfd} that 
\begin{equation}\label{eq:overleftarrowAndSums}
\overleftarrow{M}\cong \overleftarrow{\left( \oplus_{\J\in \B(M)} I^\J\right)}\cong \oplus_{\J\in \B(M)}\overleftarrow{I^\J}.
\end{equation}
In particular, $\overleftarrow{M}$ is \pfd
\end{proof}

\label{sec:closedextend}
\begin{proposition}\label{prop:exfree}
If $M$ is of type $\ccc$, then $\overleftarrow{M}$ is a free $\RCat^\op\times \RCat$-indexed module and \[\B(\overleftarrow{M}) = \{\GenR{b,a} \mid [a,b]_\BL\in \B(M)\}.\]
\end{proposition}

\begin{proof}
This follows immediately from \cref{Lem:OverleftarrowForIntervals} and \eqref{eq:overleftarrowAndSums}.
\end{proof}

\begin{proposition}
If  $M$ and $N$ are of type $\ccc$ and $\fn: M\to N$ is a monomorphism with $\epsilon$-trivial cokernel, then $\overleftarrow{\fn}: \overleftarrow{M}\to \overleftarrow{N}$ is a monomorphism with $\epsilon$-trivial cokernel.  \label{prop:exfreecok}\end{proposition}
\begin{proof} 
We need to show that for each $(s,t)\in \R^2$, $\overleftarrow{\fn}_{(s,t)}$ is an injection, and \[\im \phi_{\overleftarrow{N}}((s,t), (s-\epsilon, t+\epsilon))\subseteq \im \overleftarrow{\fn}_{(s-\epsilon,t+\epsilon)}.\]

First, assume that $s\leq t$.  The universality of limits yields canonical isomorphisms such that the following diagram commutes:
\begin{center}
\includegraphics[scale=1]{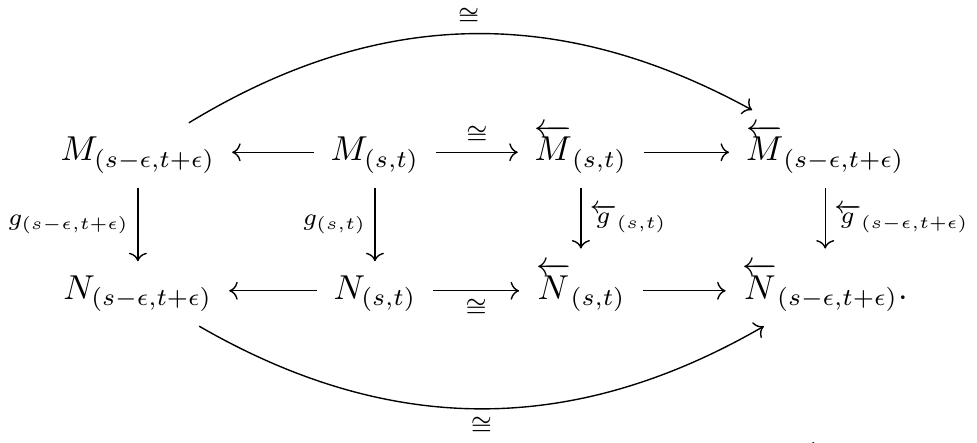}
\end{center}
It follows that $\overleftarrow{\fn}_{(s,t)}$ and $\phi_{\overleftarrow{N}}((s,t), (s-\epsilon, t+\epsilon))$ have the required properties. 

Next we consider the case $s>t$.  If $\overleftarrow{\fn}_{(s,t)}(m) = 0$ then $(\overleftarrow{\fn}_{(t,s)}\circ \phi_{\overleftarrow{M}}((s,t), (t,s)))(m) = 0$ by commutativity.  By the case $s\leq t$ considered above and \cref{prop:exfree}, the two morphisms in the latter composition are injective, so $m=0$.  Hence $\overleftarrow{\fn}_{(s,t)}$ is injective. 

Let $n\in \overleftarrow{N}_{(s,t)}$ and observe that there exist $m_1\in \overleftarrow{M}_{(t-\epsilon, t+\epsilon)}$ and $m_2\in \overleftarrow{M}_{(s-\epsilon, s+\epsilon)}$ 
such that 
\begin{align*}
\overleftarrow{\fn}(m_1) &= \left(\phi_{\overleftarrow{N}}\left((t,t), (t-\epsilon, t+\epsilon)\right)\circ \phi_{\overleftarrow{N}}((s,t), (t, t))\right)(n),\\
\overleftarrow{\fn}(m_2) &= \left( \phi_{\overleftarrow{N}}\left((s,s), (s-\epsilon, s+\epsilon)\right)\circ \phi_{\overleftarrow{N}}((s,t), (s, s))\right)(n).
\end{align*}
This is true because $\overleftarrow{\fn}$ has $\epsilon$-trivial cokernel when restricted to indices $(s',t')$ for which $s'\leq t'$. 

As $\overleftarrow{M}$ is free and $\overleftarrow{\fn}_{(t-\epsilon,s+\epsilon)}$ is an injection, there exists an element $m\in \overleftarrow{M}_{(s-\epsilon, t+\epsilon)}$ such that 
\begin{align*}
\phi_{\overleftarrow{M}}((s-\epsilon, t+\epsilon), (t-\epsilon, t+\epsilon))(m) &= m_1,\\
\quad\phi_{\overleftarrow{M}}((s-\epsilon, t+\epsilon), (s-\epsilon, s+\epsilon))(m) &= m_2.
\end{align*}
 Hence $\overleftarrow{\fn}(m) = \phi_{\overleftarrow{N}}((s,t), (s-\epsilon, t+\epsilon))(n)$ by commutativity and the injectivity of the internal maps in $\overleftarrow{N}$.
\end{proof}

\begin{proof}[Completion of the Proof of \cref{Thm:BLK_IMT}\,(iii)]
\cref{prop:exfree,prop:exfreecok} assure that $\overleftarrow{f^\cff}: \overleftarrow{M^\cff}\to \overleftarrow{N^\cff}$ is a monomorphism of free $\RCat^{\op}\times \RCat$-indexed modules with $\epsilon$-trivial cokernel.  By \cref{teo:freeIMT}, $\chi(\overleftarrow{f^\cff}): \B(\overleftarrow{M^\cff}) \nrightarrow \B(\overleftarrow{N^\cff})$ is a bijective matching such that if \[\chi(\overleftarrow{f^\cff})(\GenR{a, b}) = \GenR{a', b'},\] then
\begin{equation*}
a\leq a' \leq a+\epsilon\qquad \text{and}\qquad b-\epsilon\leq b'\leq b.
\end{equation*}
By \cref{prop:exfree}, $\chi(\overleftarrow{f^\cff})$ induces a bijective matching $\chi(f^\cff):\B(M^\cff)\nrightarrow \B(N^\cff)$ such that if \[\chi(f^\cff)[a, b]_\blk = [a', b']_
\blk,\] then
\begin{equation*}
a\leq a' \leq a+\epsilon\qquad \text{and}\qquad b-\epsilon\leq b'\leq b.\qedhere
\end{equation*} 
\end{proof}

\subsection{Proof of \cref{Thm:BLK_IMT}\,(iv)}
To prove \cref{Thm:BLK_IMT}\,(iv), we apply the induced matching theorem for $\RE$-free modules in a way analogous to the way we applied the induced matching theorem for free 2-D persistence modules in the proof of \cref{Thm:BLK_IMT}\,(iii).  First, we define a functor $\REComp{}{\epsilon}$ sending each module of type $\oo$ to an $\RE$-free module.  

\paragraph{Definition of $\protect\REComp{}{\epsilon}$}
Let us extend $(\RCat^{\op}\times \RCat)\times \{0,1\}$ to a poset $\CCat$ with the same underying set by adding an arrow $(v,0)\to (w,1)$ if and only if $v<w$.  For $i\in\{0,1\}$, let \[\iota_i:\RCat^{\op}\times \RCat\hookrightarrow \CCat\] denote the obvious map sending $\RCat\times\RCat^{\op}$ to  $\RCat\times\RCat^{\op}\times \{i\}$.  We define 
\[\overrightarrow{(-)}:=(-)|_{\iota_1} \comp \Lan_{\iota_0}(-).\] 
Thus, for $(s,t)\in \R^2$ and $\overrightarrow{(s,t)}\subset \RCat^{\op}\times \RCat$ given by
\[
\overrightarrow{(s,t)}:=\{(x,y) \mid s<x \mbox { and } y<t\},
\]
we have $\overrightarrow{M}_{(s,t)}=\varinjlim M|_{\overrightarrow{(s,t)}}$  for any $\RCat^{\op}\times \RCat$-indexed persistence module $M$; the internal maps of $\overrightarrow{M}$ are given by the universality of colimits.

Define \[\REComp{}{\epsilon}:=R_\epsilon \comp{(-)^*} \comp \overrightarrow{(-)},\]
where \[(-)^*:\Vect^{\RCat^{\op}\times \RCat}\to \Vect^{\RCat\times \RCat^{\op}}\] denotes the dualization functor of \cref{Sec:Cld_To_2D}.

\paragraph{Properties of  $\protect\REComp{}{\epsilon}$ on Modules of Type $\oo$}\label{sec:extopen}

\begin{lemma}\label{lem:openlocalfinite}
For $M$ of type $\oo$, $\delta > 0$, and $(s,t) \in \R^2$, there are a finite number of blocks $(a,b)_\BL \in \B(M)$
such that $a\leq s, b\geq t$, and $b-a\geq \delta$. 
\end{lemma}
\begin{proof}
Let $\# (s,t)$ denote the number of blocks $(a,b)_\BL$ with the specified properties.  It is easy to check that since $M_{(s,s)}$ finite dimensional for all $s\in \R$, each $\# (s,s)$ is finite.    If $s<t$, then \[\#(s,t)\leq \#\left(\frac{s+t}{2},\frac{s+t}{2}\right)<\infty.\]  If $s>t$, then choosing a positive integer 
$l$ such that $\min(s,t) + l\delta > \max(s,t)$ we have
\[ \# (s,t) \leq \sum_{i=0}^l \#(\min(s,t) + i \delta, \min(s,t) + i\delta)<\infty.\qedhere\]
\end{proof}

\begin{proposition}\label{prop:openfreebarcode}
If is $M$ of type $\oo$ and $\delta>0$, then $\REComp{M}{\delta}$ is $R_\delta$-free and
\[\B(\REComp{M}{\delta}) = \{\TopRR{a,b}{\delta} \mid (a,b)_\BL\in \B(M)_\delta\}.\]
\end{proposition}
\begin{proof}
As illustrated in \cref{fig:extopen}, for all $a<b\in \R$, $\REComp{I^{(a,b)_\BL}}{\delta} = I^{\TopRR{a,b}{\delta}}$.  By \cref{Prop:Preservation_of_Coproducts}\,(i), $\Lan_{\iota_0}(-)$ preserves direct sums.  Clearly, $(-)|_{\iota_1}$, $(-)^*$, and $R_\delta$ also preserve direct sums, so the composition $\REComp{}{\delta}$ preserves direct sums as well.  Hence,
\[\REComp{M}{\delta} \cong \bigoplus_{(a,b)_\BL\in \B(M)_\delta} \REComp{I^{(a,b)_\BL}}{\delta} = \bigoplus_{(a,b)_\BL\in \B(M)_\delta} I^{\TopRR{a,b}{\delta}}.\]
Thus $\B(\REComp{M}{\delta})$ is as claimed.  

To see that $\REComp{M}{\delta}$ is $R_\delta$-free, let \[F:=\bigoplus_{(a,b)_\BL\in \B(M)_\delta}  I^{\GenRR{a,b}}.\]  $F$ is \pfd by \cref{lem:openlocalfinite} , so since $\REComp{M}{\delta}\cong R_\delta(F)$, the result follows.
\end{proof}

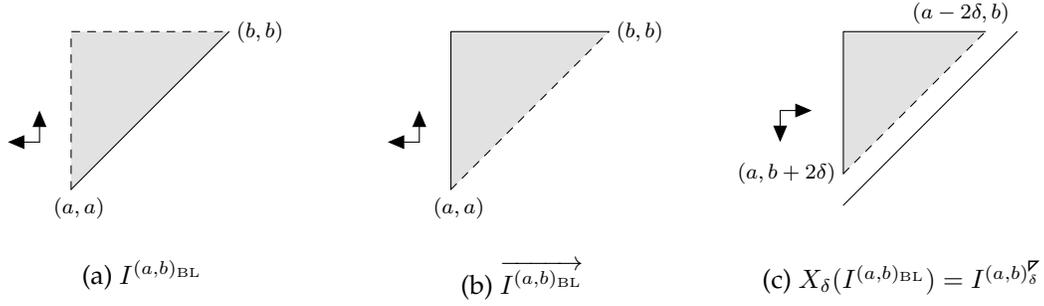
\begin{figure}
\centering
\begin{subfigure}[t]{0.3\textwidth}
\centering
\begin{tikzpicture}[line cap=round,line join=round,>=triangle 45,x=.7cm,y=0.7cm, scale=0.3,baseline=0]
\clip(-10.,-8.5) rectangle (9.,11.);
\fill[color=aqaqaq,fill=aqaqaq,fill opacity=0.30] (-5.,-5.) -- (-5,5) -- (5,5) -- cycle;
\draw (-5,-5) -- (5,5);
\draw[dashed] (-5.,-5.) -- (-5,5) -- (5,5);
\draw[->] (3-10, -2) -- (1-10, -2);
\draw[->] (3-10, -2) -- (3-10, -0); 
\begin{scriptsize}
\draw[color=black] (-4.5, -5) node[below] {$(a,a)$};
\draw[color=black] (5,5) node[right] {$(b,b)$};
\end{scriptsize}
\end{tikzpicture}
\subcaption{$I^{(a,b)_\blk}$}
\end{subfigure}
\begin{subfigure}[t]{0.3\textwidth}
\centering
\begin{tikzpicture}[line cap=round,line join=round,>=triangle 45,x=.7cm,y=0.7cm, scale=0.3,baseline=0]
\clip(-10.,-8.5) rectangle (9.,11.);
\fill[color=aqaqaq,fill=aqaqaq,fill opacity=0.30] (-5.,-5.) -- (-5,5) -- (5,5) -- cycle;
\draw[dashed] (-5,-5) -- (5,5);
\draw[] (-5.,-5.) -- (-5,5) -- (5,5);
\draw[->] (3-10, -2) -- (1-10, -2);
\draw[->] (3-10, -2) -- (3-10, -0); 
\begin{scriptsize}
\draw[color=black] (-4.5, -5) node[below] {$(a,a)$};
\draw[color=black] (5,5) node[right] {$(b,b)$};
\end{scriptsize}
\end{tikzpicture}
\subcaption{$\overrightarrow{I^{(a,b)_\blk}}$}
\end{subfigure}
\begin{subfigure}[t]{0.3\textwidth}
\centering
\begin{tikzpicture}[line cap=round,line join=round,>=triangle 45,x=.7cm,y=0.7cm, scale=0.3,baseline=0]
\clip(-11.6,-8.5) rectangle (9.,11.);
\fill[color=aqaqaq,fill=aqaqaq,fill opacity=0.30] (-5.,-4.) -- (-5,5) -- (4,5) -- cycle;
\draw (-5.,-4.) -- (-5,5) -- (4,5);
\draw[dashed](-5,-4)--(4,5); 
\draw  (-5, -6) -- (6, 5);
\draw[->] (1-10, 0) -- (3-10, 0);
\draw[->] (1-10, 0) -- (1-10, -2); 
\begin{scriptsize}
\draw[color=black] (6,5) node[above left] {$(a-2\delta,b)$};
\draw[color=black] (-5,-4) node[left] {$(a,b+2\delta )$};
\end{scriptsize}
\end{tikzpicture}

\subcaption{$\REComp{I^{(a,b)_\blk}}{\delta}=I^{\TopRR{a,b}{\delta}}$}
\end{subfigure}
\caption{Applying $\protect\REComp{-}{\delta}$ to $I^{(a,b)_\BL}$}
\label{fig:extopen} 
\end{figure}

\begin{lemma}\label{lem:opensurj}
Let $M$ and $N$ be of type $\oo$, and let $f: M\to N$ be a morphism with $\epsilon$-trivial cokernel. Then $f$ is surjective at all indices $(s,t)$ for which $t-s\geq 2\epsilon$. 
\end{lemma}
\begin{proof}
Observe that $\phi_N((s',t'), (s,t))$ is surjective whenever $t'-s'\geq 0$.  If $n\in N_{(s,t)}$ is not in the image of $f$, then neither is any element in $\phi_N((\frac{s+t}{2}, \frac{s+t}{2}), (s,t))^{-1}(n)\neq \emptyset$, contradicting that $f$ has $\epsilon$-trivial cokernel. 
\end{proof}

\begin{proposition}\label{prop:openexmorph}
If $M$ and $N$ are of type $\oo$ and $f: M\to N$ has $\epsilon$-trivial kernel and cokernel,
then $\REComp{f}{\epsilon}$ is a monomorphism with $\epsilon$-trivial cokernel. 
\end{proposition}
\begin{proof}
It follows from \cref{lem:openlocalfinite} that for any module $Q$ of type $\oo$ and $(s,t)\in \R^2$, there exists a $\eta>0$ such that $\phi_Q((s+\eta, t-\eta), (s+\eta', t-\eta'))$ is an isomorphism for all $0<\eta'\leq \eta$. In particular, the natural map $Q_{(s+\eta, t-\eta)} \to \REXD{Q}_{(s,t)}$ is an isomorphism.  Applying this observation four times, we find that there exists $\eta>0$ such that the leftmost and rightmost horizontal maps are isomorphisms in the following commutative diagram:
\begin{center}
\includegraphics[scale=1]{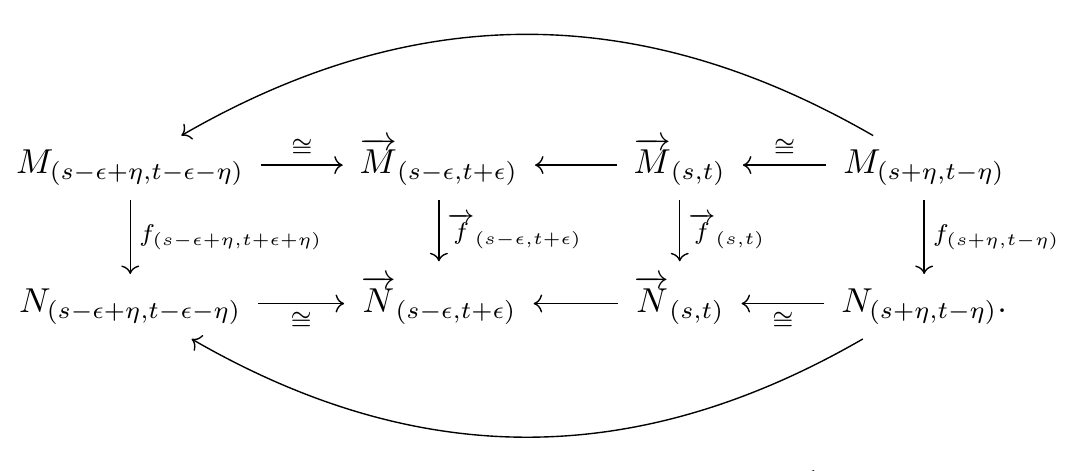}
\end{center}
This shows that $\overrightarrow{f}$ has $\epsilon$-trivial kernel, and by \ref{lem:opensurj}, that $\overrightarrow{f}_{(s,t)}$ is surjective at all indices satisfying $t-s>2\epsilon$.  The result now follows from \cref{Prop:Kernel_Dualization}. 
\end{proof}

\begin{proof}[Proof of \cref{Thm:BLK_IMT}\,(iv)]
Suppose $M$ and $N$ are of type $\oo$ and $f : M \to N$ has $\epsilon$-trivial kernel and cokernel. By \cref{prop:openfreebarcode,prop:openexmorph}, $\REComp{f}{\epsilon}: \REComp{N}{\epsilon} \to \REComp{M}{\epsilon}$ is a monomorphism of $\RE$-free persistence modules with $\epsilon$-trivial cokernel.  By \cref{cor:IMTRFree} and \cref{prop:openfreebarcode}, we obtain matchings
\[\B(M)\nrightarrow \B(\REComp{M}{\epsilon}) \nrightarrow \B(\REComp{N}{\epsilon})\nrightarrow \B(N).\]
The composition of these is our desired matching.
\end{proof}

%% file: discussion.tex
\section{Stability of Almost-Block Decomposable Modules}\label{Sec:Almost_Block_Stability}
In this section, we present a simple extension of the block stability theorem to a slightly more general classes of modules, and discuss an application to the stability of (inter)level set persistent homology.
  
Recall our definition of a block from \cref{Sec:Block_Decomposable_Persistence_Modules}.   We define an \emph{almost-block} $\J$ to be an interval in $\INTR$ for which there exists a block $\J_{\rm BL}$ such that $d_I(I^\J, I^{\J_{\rm BL}})=0$.  Some almost-blocks which are not blocks are shown in \cref{fig:GenIntervals}.  We say $M$ is \emph{almost-block decomposable} if $M$ is interval decomposable, with each interval in $\B(M)$ an almost-block.
\begin{figure}
\centering
\begin{subfigure}{0.24\textwidth}

\begin{tikzpicture}[line cap=round,line join=round,>=triangle 45,x=.7cm,y=0.7cm, scale=0.2]
\clip(-13.,-10.) rectangle (13.,13.);
\fill[color=aqaqaq,fill=aqaqaq,fill opacity=0.30] (-5.,-5.) -- (-5,5) -- (5,5) -- cycle;
\draw (-10,-10) -- (10,10);
\draw (-5.,-5.) -- (-5,5) -- (5,5);

\begin{scriptsize}
\draw[color=black] (-4.5, -5) node[right] {$(a,a)$};
\draw[color=black] (5,5) node[right] {$(b,b)$};
\end{scriptsize}
\end{tikzpicture}
\end{subfigure}
\begin{subfigure}{0.24\textwidth}
\begin{tikzpicture}[line cap=round,line join=round,>=triangle 45,x=.7cm,y=0.7cm, scale=0.2]
\clip(-13.,-10.) rectangle (10.2,13.);
\fill[color=aqaqaq,fill=aqaqaq,fill opacity=0.30] (-5.,-5.) -- (-10,-5) -- (-10,5) -- (5,5) -- cycle;
\draw (-10,-10) -- (10,10);
\draw (-5,-5) -- (-10,-5);
\draw[->] (-7,0) -- (-12, 0); 
\draw (-10, 5) -- (5,5); 
\begin{scriptsize}
\draw[color=black] (-4.5, -5) node[right] {$(a,a)$};
\draw[color=black] (5,5) node[right] {$(b,b)$};
\end{scriptsize}
\end{tikzpicture}
\end{subfigure}
\begin{subfigure}{0.24\textwidth}
\begin{tikzpicture}[line cap=round,line join=round,>=triangle 45,x=.7cm,y=0.7cm, scale=0.2]
\clip(-10.,-10.) rectangle (13.,13.);
\fill[color=aqaqaq,fill=aqaqaq,fill opacity=0.30] (-5.,-5.) -- (-5,10) -- (5,10) -- (5,5) -- cycle;
\draw (-10,-10) -- (10,10);
\draw[dashed] (5, 10) -- (5,5);
\draw[->] (0,7) -- (0,12);
\draw (-5, -5) -- (-5, 10);
\begin{scriptsize}
\draw[color=black] (-4.5, -5) node[right] {$(a,a)$};
\draw[color=black] (5,5) node[right] {$(b,b)$};
\end{scriptsize}
\end{tikzpicture}
\end{subfigure}
\begin{subfigure}{0.24\textwidth}
\begin{tikzpicture}[line cap=round,line join=round,>=triangle 45,x=.7cm,y=0.7cm, scale=0.2]
\clip(-13.,-10.) rectangle (13.,13.);
\fill[color=aqaqaq,fill=aqaqaq,fill opacity=0.30] (-5.,-5.) -- (-10,-5) -- (-10, 10) -- (5,10) --  (5,5) -- cycle;
\draw (-10,-10) -- (10,10);
\draw[dashed] (-5, -5) -- (-10,-5);
\draw[dashed] (5,10) -- (5,5);

\draw[->] (-7,0) -- (-12, 0);
\draw[->] (0,7) -- (0,12);
\draw[->] (-7,7) -- (-12,12);

\begin{scriptsize}
\draw[color=black] (-4.5, -5) node[right] {$(a,a)$};
\draw[color=black] (5,5) node[right] {$(b,b)$};
\end{scriptsize}
\end{tikzpicture}
\end{subfigure}

\caption{Four almost-blocks that are not blocks.}
\label{fig:GenIntervals}
\end{figure} 

\begin{corollary}[Almost-Block Stability]\label{Cor:Almost_Block_Stability}
For \pfd almost-block decomposable modules $M$ and $N$,
\[d_I(M,N) \leq d_b(\B(M), \B(N))\leq \frac{5}{2} d_I(M,N).\]
\end{corollary}

\begin{proof}[Sketch of proof]
For any $\delta>0$, there exist \pfd block decomposable modules $M'$ and $N'$ with \[d_I(M,M'),d_I(N,N'),d_b(M,M'),d_b(N,N')\leq \delta.\]  Given this, the inequality $d_b(\B(M), \B(N))\leq \frac{5}{2} d_I(M,N)$ follows easily from \cref{teo:IMTinterleaving}, together with the triangle inequalities for $d_I$ and $d_b$.  

It follows from \cref{lem:converseAST} that $d_I(M,N) \leq d_b(\B(M), \B(N))$.
\end{proof}

\paragraph{Almost-Block Stability and Interlevel Set Persistent Homology}
Almost-block decomposable persistence modules can arise as the interlevel set persistent homology of non-Morse functions, as the following example illustrates:

\begin{example}\label{Ex:Not_Morse}
The function $\gamma: (0,1)\to \R$ given by $\gamma(t)=t$ is not of Morse type.  $H_0(\FI\gamma)$ is almost-block decomposable but not block decomposable; $\B(H_0(\FI\gamma))$ consists of a single interval $\J$ with $d_b(\J,[0,1]_\blk)=0$.
\end{example}
In fact, we hypothesize that  \cref{Thm:Morse_Block_Decomposition}\,(i) generalizes as follows:

\begin{conjecture}\label{Conj:Decomposition}
For any topological space $\TopSpace$ and continuous function $\gamma:\TopSpace\to \R$, if $H_i(\FI{\gamma})$ is \pfd then it is almost-block decomposable.
\end{conjecture}

If \cref{Conj:Decomposition} is true, then the definition of level set barcodes of \cref{Sec:Levelset_Persistence} extends to any $\R$-valued function with \pfd interlevel set  homology, and a stability result for the interlevel and level set barcodes of such functions follows immediately from \cref{Cor:Almost_Block_Stability}.

\begin{remark}
In \cite{carlsson201Xparameterized}, Carlsson, de Silva, Kali\v snik, and Morozov use the formalism of rectangle measures \cite{chazal2012structure} to define levelset barcodes of $\R$-valued functions in a general setting, and establish a stability result for these barcodes.  \cref{Conj:Decomposition} is inspired by discussions with de Silva and Kali\v snik about that work.
\end{remark}

\begin{remark}
Subsequent to the first iteration of this paper, Cochoy and Oudot \cite{cochoy2016decomposition} have established a structure theorem for a certain class of 2-D persistence modules which yields as corollaries two variants of \cref{Conj:Decomposition}: 
\begin{enumerate}[(i)]
\item Let $\FI{\gamma}^\circ$ be the $\U$-indexed module given by $\FI{\gamma}^\circ_{(a,b)} := \gamma^{-1}((a,b))$ if $a<b$, and $\FI{\gamma}^\circ_{(a,b)}=0$ otherwise.  Then $H_i({\FI{\gamma}}^\circ)$ is almost-block decomposable. 
\item Let $M$ be the $\U$-indexed module obtained from $H_i(\FI{\gamma})$ by setting to 0 each vector space on the diagonal line $y=x$.  Then $M$ is almost-block decomposable. 
\end{enumerate}
\end{remark}

\section{Discussion}\label{Sec:Discussion}

\paragraph{Towards a General Theory of Algebraic Stability}
In this paper, we have introduced an algebraic stability theorem for block decomposable modules which, as an easy corollary, yields a stability result for zigzag modules.  It is natural to ask whether our results generalize to an algebraic stability theorem for arbitrary interval decomposable $\RCat^n$-indexed modules.  In answer to this question, the following example shows that for interval decomposable $\RCat^2$-indexed modules $M$ and $N$, the ratio \[\frac{d_b(\B(M),\B(N))}{d_I(M,N)}\] can be arbitrarily large.  

\begin{example}\label{nD_Iso_Counterexample}
For fixed $a\geq 0$, let $\J_1\subset \R^2$ be the polygonal interval whose outer edge is specified by the following sequence of vertices: 
\[(5,-a),(9+a,-a),(9+a,4),(6,4),(6,6),(4,6),(4,9+a),(-a,9+a),(-a,5),(5,5);\]
Let $\J_2$ be the square interval with vertices $(6,1-a),(10+a,1-a),(10+a,5),(6,5)$; and let $\J_3$ be the square with vertices
$(1-a,6),(5,6),(5,10+a),(1-a,10+a)$; see \cref{fig:nD_Iso_Counterexample}.
For $M=I^{\J_1}$ and $N=I^{\J_2}\oplus I^{J_3}$, we have \[d_I(M,N)=1,\quad d_b(\B(M),\B(M))=2+a/2.\]
\end{example}
\begin{figure}
\centering
\begin{tikzpicture}[scale=0.6]
\draw[fill=black, opacity=0.3] (5,0) -- (9,0) -- (9,4) -- (6,4) -- (6,6) -- (4,6) -- (4,9) -- (0,9) -- (0,5) -- (5,5) -- cycle;
\draw[fill=blue, opacity=0.3] (6,1) -- (10,1) -- (10,5) -- (6,5) -- cycle; 
\draw[fill=red, opacity=0.3] (1,6) -- (5,6) -- (5,10) -- (1,10) -- cycle;
\draw[->] (6.5, 9.5) node[right]{$\J_3$} -- (4.5,9.5);
\draw[->] (9.5, 6.5)node[right]{$\J_2$} -- (9.5, 4.5);
\draw[->, black] (2, 3.5)node[left]{$\J_1$} -- (2, 5.5);
\end{tikzpicture}
\caption{An illustration of \cref{nD_Iso_Counterexample} in the case $a=0$.}
\label{fig:nD_Iso_Counterexample}
\end{figure}
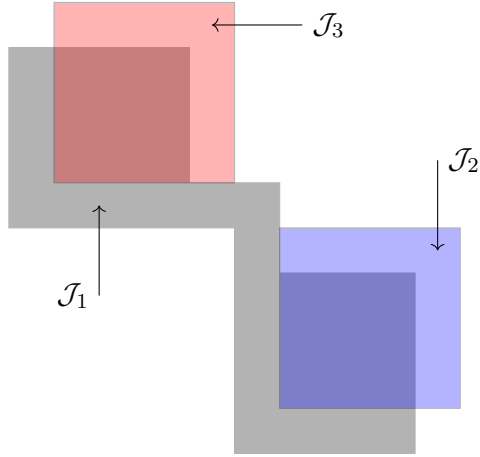

\cref{nD_Iso_Counterexample} makes clear that to formulate a general algebraic stability result for interval decomposable $\RCat^n$-indexed modules, we need either to constrain the shape of the intervals in our barcodes, or to work with a distance on barcodes other than the bottleneck distance.  

Let us say an $\R^2$-indexed module $M$ is \emph{rectangle decomposable} if $M$ is interval decomposable and $\B(M)$ is a collection of rectangles.  A preliminary version of this paper \cite{botnan} conjectured that the isometry theorem holds for interval decomposable $\RCat^n$-indexed modules whose barcodes consist of convex intervals.  However, Bjerkevik has subsequently given an example of rectangle decomposable $\R^2$-indexed modules $M$ and $N$ with \[d_b(\B(M),\B(N))=3\, d_I(M,N),\] disproving the conjecture \cite{bjerkevik2016stability}.  We thus weaken the conjecture as follows:

\begin{conjecture}[Generalized Algebraic Stability]\label{nD_Iso_Conjecture}
For each $n\in \{1,2,\ldots\}$, there is a constant $c_n$ such that for $M$ and $N$ interval decomposable $\RCat^n$-indexed modules with each interval in $\B(M)$ and $\B(N)$ convex, we have \[d_b(\B(M),\B(N))\leq c_n\,d_I(M,N).\]  
\end{conjecture}

\cite{bjerkevik2016stability} provides positive answers to this conjecture in the case of free and rectangle decomposable modules, using arguments similar to the one used there to strengthen the block stability theorem. 

\paragraph{Single Morphism Algebraic Stability}
We have proven the block stability theorem by way of an induced matching result for block decomposable modules, \cref{Thm:BLK_IMT}.  While \cref{Thm:BLK_IMT}\,(i)-(iii) are tight,  \cref{Thm:BLK_IMT}\,(iv) (concerning modules of type $\oo$) is not tight: A simple application of the tight form of the block stability theorem appearing in \cite{bjerkevik2016stability} gives that under the assumptions of \cref{Thm:BLK_IMT}\,(iv), there exists a 2-matching between the barcodes in question; this improves on the constant of $\frac{5}{2}$ appearing in \cref{Thm:BLK_IMT}\,(iv), albeit with matchings that are not explicitly given.  On the other hand, for modules of typo $\oo$, the best lower bound we know for single morphism algebraic stability is $\frac{3}{2}$.  The problem of establishing a tight single morphism algebraic stability result for block decomposable modules thus remains open.   The same problem is also of interest for more general interval decomposable $\RCat^n$-indexed modules.

As with the proof of the induced matching theorem in 1-D given in \cite{bauer2015induced}, we have proven  \cref{Thm:BLK_IMT}\,(iv) by factoring a morphism of block decomposable persistence modules into morphisms with simpler structure, and then defining induced matchings for each of the factors.  We wonder whether this strategy could be pushed further to yield stronger, more general single morphism stability results.  The central difficulty is that the interpolating modules one obtains via our factorization are typically not interval decomposable.  In our study of block decomposable modules, we have circumvented this issue by working with certain truncations of the interpolating modules which are interval decomposable.  

A potential alternative strategy would be to avoid truncation, and instead perturb our morphism $f:M\to N$ to obtain another morphism $f':M\to N$ whose associated interpolants are interval decomposable, while controlling the persistence of $\ker f'$ and $\coker f'$.  It seems plausible that such an approach could yield stronger and more general results.  